\documentclass[11pt,a4paper]{article}
\usepackage[utf8]{inputenc}
\usepackage[english]{babel}
\usepackage{amsmath}
\usepackage{amsthm}
\usepackage{amsfonts}
\usepackage{amssymb}
\usepackage{array}  
\usepackage{graphicx} 
\usepackage{color}
\usepackage{mathrsfs}
\usepackage{dsfont}
\usepackage{subcaption}
\usepackage{multicol}
\usepackage{lipsum}
\usepackage{constants}
\usepackage{enumitem}
\usepackage{geometry}
\usepackage[final]{hyperref}
\hypersetup{
    linktoc=page,
    linkcolor=red,
    citecolor=blue,
    filecolor=blue,
    urlcolor=cyan,
    colorlinks=true}

\author{Hugo Duminil-Copin\footnote{Institut des Hautes Études Scientifiques and Université de Genève}, Alejandro Rivera\footnote{École Polytechnique Fédérale de Lausanne},\\[0.3em]
Pierre-Fran\c cois Rodriguez\footnote{Imperial College London}, Hugo Vanneuville\footnote{CNRS and Université Grenoble Alpes}}
\title{\textbf{\Large Existence of an unbounded nodal hypersurface for smooth Gaussian fields in dimension $d \geq 3$}}

\date{}

\theoremstyle{plain}
\newtheorem{theorem}{Theorem}[section]
\newtheorem{corollary}[theorem]{Corollary}
\newtheorem{proposition}[theorem]{Proposition}
\newtheorem{conjecture}[theorem]{Conjecture}
\newtheorem{claim}[theorem]{Claim}
\newtheorem{lemma}[theorem]{Lemma}

\theoremstyle{definition}
\newtheorem{definition}[theorem]{Definition}

\newtheorem{remark}[theorem]{Remark}

\newtheorem{ass}[theorem]{Assumption}

\newcommand{\E}{\mathbb{E}}
\newcommand{\N}{\mathbb{N}}
\newcommand{\R}{\mathbb{R}}
\newcommand{\Z}{\mathbb{Z}}
\newcommand{\Pro}{\mathbb{P}}

\newcommand{\cross}{\text{\textup{Cross}}}

\newconstantfamily{c}{symbol=c}

\def\calC{\mathcal{C}}

\def\calE{\mathcal{E}}
\def\calF{\mathcal{F}}

\def\calK{\mathcal{K}}

\def\calN{\mathcal{N}}
\def\calP{\mathcal{P}}

\def\calT{\mathcal{T}}

\newcommand{\prob}{\mathbb{P}}

\def\eps{\varepsilon}

\makeatletter
\newcommand\footnoteref[1]{\protected@xdef\@thefnmark{\ref{#1}}\@footnotemark}
\makeatother

\usepackage{titletoc}

\dottedcontents{section}[4em]{}{2.9em}{0.7pc}
\dottedcontents{subsection}[0em]{}{3.3em}{1pc}

\begin{document}

\maketitle
\thispagestyle{empty}

\begin{abstract}
For the Bargmann--Fock field on $\R^d$ with $d\ge3$, we prove that the critical level $\ell_c(d)$ of the percolation model formed by the excursion sets $\{ f \ge \ell \}$ is strictly positive. This implies that for every $\ell$ sufficiently close to 0 (in particular for the nodal hypersurfaces corresponding to the case $\ell=0$), $\{f=\ell\}$ contains an unbounded connected component that visits ``most'' of the ambient space. Our findings actually hold for a more general class of positively correlated smooth Gaussian fields with rapid decay of correlations. The results of this paper show that the behaviour of nodal hypersurfaces of these Gaussian fields in $\mathbb R^d$ for $d\ge3$ is very different from the behaviour of nodal lines of their two-dimensional analogues.
\end{abstract}

{\small
\tableofcontents
}

\section{Introduction}

Let $f$ be a stationary, isotropic, centered and smooth Gaussian field on $\R^d$, and let $\ell \in \R$. Percolation properties of level sets $\{ f = \ell \}$ and of excursion sets $\{ f \ge \ell \}$ have been extensively investigated in recent years in the case $d=2$, starting with the work of Beffara and Gayet \cite{BG} on box-crossing properties. In particular, under very mild conditions, it has been proved that the critical level is $0$ if $d=2$, in the sense that if $\ell > 0$ then a.s.\ there is no unbounded component in $\{ f \ge \ell \}$ while such an unbounded component exists a.s.\ if $\ell < 0$ \cite{MRVK}. When $f$ is positively correlated and if $d=2$, it is also known that a.s.\ there is no unbounded component in $\{ f \ge 0 \}$ \cite{Alex}.

\medskip

In the present paper, we study these questions in the case $d \ge 3$. Our main result is that for a class of positively correlated Gaussian fields with fast decay of correlations -- including the Bargmann--Fock field introduced below -- the critical level is strictly larger than $0$. In particular, contrary to the case $d=2$, a.s.\ there exists an unbounded component in $\{ f = 0 \}$. Let us first state our main results in the specific case of the Bargmann--Fock field.

\subsection{Existence of an unbounded nodal hypersurface for the Bargmann--Fock field in dimension $d \ge 3$}\label{ssec:bf}

Let $d \ge 2$. The Bargmann--Fock field in $\R^d$ is the analytic centered Gaussian field $f$ defined by the following covariance kernel:
\[
\forall x,y\in\mathbb R^d,\qquad\mathbb E[f(x)f(y)]=\exp(-\tfrac{1}{2}|x-y|^2).
\]
This field can be realized as the following entire series, where $a_{i_1,\dots,i_d}$ are i.i.d.\ standard Gaussian variables:
\[
f(x)=\exp(-\tfrac{1}{2}|x|^2)\sum_{i_1,\dots,i_d \in \N} a_{i_1,\dots,i_d} \frac{x_1^{i_1} \cdots x_d^{i_d}}{\sqrt{i_1! \dots i_d!}}.
\]
One can note that if $d' \le d$ then $f_{|\R^{d'}}$ is the Bargmann--Fock field in $\R^{d'}$. As explained for instance in the introduction of \cite{BG}, this field arises naturally from real algebraic geometry when considering suitably rescaled random homogenous polynomials of degree $n$ in $d+1$ variables in the limit when $n \to \infty$. In the present paper, given some level $\ell \in \R$, we study the connectivity properties of the random sets
\[
\{ f = \ell \} := \{ x \in \R^d : f(x) = \ell \} \hspace{1em} \text{and} \hspace{1em} \{ f \ge \ell \} := \{ x \in \R^d : f(x) \ge \ell \}.
\]
By ergodicity (see e.g.\ \cite[Theorem 6.5.4]{Adl10}), for every $\ell \in \R$, the event that there is an unbounded component in $\{f \ge \ell\}$ (resp.\ $\{ f = \ell \}$) has probability either $0$ or $1$. The critical level is defined as follows:
\begin{equation}\label{eq:lc}
\ell_c(d) := \sup \big\{ \ell \in \R : \Pro \big[\exists \text{ an unbounded component in } \{f \ge \ell\} \big] = 1 \big\}.
\end{equation}
It is known that $\ell_c(2) = 0$ \cite{RVb}. Moreover, still if $d=2$, it is known that there is no unbounded connected component in $\{ f \ge 0 \}$ \cite{Alex}. In the present paper, we prove that the critical level strictly increases between dimensions $2$ and $3$.
\color{black}

\begin{theorem}\label{thm:existence}
If $d \ge 3$ then $\ell_c(d)>0$.
\end{theorem}

We refer to \cite{CR85,GM90,AG} and \cite{DPR18} for analogous results in the context of Bernoulli percolation and the discrete Gaussian free field respectively. Note that Theorem~\ref{thm:existence} is a pure existence result. For these other examples, a (strong) uniqueness result also holds, cf.~\cite{CCN,GM90, DPR18.2}. By analogy, we expect that the unbounded component of the Bargmann--Fock field is unique (see Section \ref{ssec:main_results} for a general conjecture).

\medskip

From Theorem \ref{thm:existence}, we deduce the following result.

\begin{corollary}\label{cor:existence}
There exists $\delta > 0$ such that the following holds for any $\ell \in [-\delta,\delta]$. If $d \ge 3$ then a.s.\ there exists an unbounded component in $\{ f = \ell \}$.
\end{corollary}

We actually prove more precise results, namely that
\begin{itemize}[noitemsep]
\item Such an unbounded connected component still exists if we restrict the field to a thick slab;
\item There is an unbounded component that visits ``most of the space'';
\end{itemize}
see Theorem \ref{thm:main} for a formal (and more general) statement, which also directly implies Theorem~\ref{thm:existence} and Corollary \ref{cor:existence}.

\medskip

Note that the level $\ell=0$ plays a special role since it is the planar critical level. To illustrate this, it seems worth restating here the following conjecture. For $t\in \mathbb R$, let $\mathcal{P}_t$ denote the plane $\{ x \in \R^3 : x_3=t\}$.

\begin{conjecture}[\cite{GV}]

Assume that $d=3$ and let $t<t'$. A.s.\ there exists $s \in [t,t']$ such that $\{ f = 0 \} \cap \mathcal{P}_s$ contains an unbounded component. Moreover, the Hausdorff dimension of the set $\{ t \in \R : \{ f = 0 \} \cap \mathcal{P}_t \text{ contains an unbounded component} \}$ equals $2/3$ a.s.
\end{conjecture}

\subsection{Extension to a family of Gaussian fields}\label{ssec:main_results}

The previous results extend to a class of positively correlated smooth Gaussian fields with fast decay of correlations. We now state the precise assumptions. Let $d \geq 2$ and $f$ be a continuous modification of
\begin{equation}\label{eq:hh1}
q \star W,
\end{equation}
where $\star$ denotes the convolution, $W$ is a $d$-dimensional $L^2$-white noise and $q : \R^d \rightarrow \R$ satisfies Assumption \ref{ass1} below.
\begin{ass}\label{ass1}
Let $\beta>d$. We say that  $q$ {\em satisfies Assumption~\ref{ass1} for $\beta$} if
\begin{itemize}[noitemsep]
\item (regularity) $q$ is $C^{10}$ and there exists $\eps>0$ such that $|\partial^\alpha q| (x) = O(|x|^{-(d/2+\varepsilon)})$ for every multi-index $\alpha$ with $|\alpha| \leq 10$;
\item (decay of correlations) $|\partial^\alpha q|(x)=O(|x|^{-\beta}) $ for every multi-index $\alpha $ with $|\alpha| \leq 1$;
\item (isotropy) $q$ is radial;
\item (positivity) $q \geq 0$;
\item (non-triviality) $q$ is not identically equal to $0$.
\end{itemize}
\end{ass}

If $q$ satisfies Assumption \ref{ass1} (for some $\beta>d$) then $f$ defined by \eqref{eq:hh1} is a stationary, isotropic and centered Gaussian process with covariance
\[
\E [f(x) f(y) ] = (q \star q) (x-y).
\]
Moreover, $q \star q$ is $C^{10}$ so $f$ is a.s.\ $C^4$ (see for instance Sections A.3 and A.9 of \cite{NS}) and one can check that
\[
(q \star q)(x) = O\big(|x|^{-\beta}\big).
\]
\begin{remark}
In several lemmas, conjectures etc., one could also assume that $\beta \in (d/2,d]$ and ask some conditions on the spectral measure (as e.g.\ in \cite{MV}), but we have chosen to assume that $\beta>d$ to simplify the statements. In particular, this implies that the spectral measure of $f_{|\R^{d'}}$ is continuous and strictly positive at $0$ for all $d' \in [1,d]$.
\end{remark}

\begin{remark} If $q(x):=(2/\pi)^{d/4}e^{-|x|^2}$ (which satisfies Assumption \ref{ass1} for every $\beta$) then $f$ is the Bargmann--Fock field.
\end{remark}

Our main result is the following. Here and elsewhere, if $1 \le d ' \le d$ and $D \subset \R^{d'}$, we identify $D$ with $D \times \{ 0 \}^{d-d'}$.

\begin{theorem}\label{thm:main}
Let $d \ge 3$. There exists $\beta_0 > d$ such that the following holds. Let  $q : \R^d \rightarrow \R$ satisfy Assumption \ref{ass1} for some $\beta > \beta_0$. Then, there exist $L,N,\delta>0$ such that, for every $\ell \in [-\delta,\delta]$, a.s.\
\begin{itemize}
\item[i)] there exists an unbounded component in $\{ f = \ell \} \cap (\R^2 \times [0,L])$,\footnote{Recall the convention stated before the theorem: $\R^2 \times [0,L] = \R^2 \times [0,L] \times \{0\}^{d-3}$.} and
\item[ii)] for every $d' \in [3,d]$ there exist an unbounded component $\mathcal{C}$ of $\{ f = \ell \} \cap \R^{d'}$ and some (random) $R_0>0$ such that, for every $R \ge R_0$, $\mathcal{C}$ intersects all the (Euclidean, closed) balls of radius $(\log R)^N$ that are included in $[-R,R]^{d'}$.
\end{itemize} 
\end{theorem}
%

It is instructive to compare with the planar case, for which we have the following result. Let $d \ge 2$ and $q : \R^d \rightarrow \R$ satisfy Assumption~\ref{ass1}. A result from \cite{Alex} implies that for $\ell \ge 0$ there is no unbounded component in $\{ f \ge \ell \} \cap \R^2$ a.s. In particular, for all $\ell \in \R$ there is no unbounded component in $\{ f = \ell \} \cap \R^2$ a.s. Conversely,  \cite{MRVK} (see also \cite{RVb,MV,Riv,GV} for previous yet less general results) implies that for $\ell < 0$, there exists a unique unbounded component in $\{ f \ge \ell \} \cap \R^2$ a.s. The results at $\ell \neq 0$ are also known without any positivity assumption, see \cite{MRVK}.

\medskip

We conjecture that the unbounded component of $\{ f = \ell \}$ (or $\{ f \ge \ell \}$) is a.s.\ unique when it exists. We state this conjecture as we believe it to be of importance to understand further the properties of the nodal hypersurfaces of Gaussian fields.

\begin{conjecture}
Let $d \ge 3$ and $q$ satisfy Assumption \ref{ass1} for some $\beta> d$. Then, the unbounded component of $\{ f = \ell \}$ is a.s.\ unique when it exists (and similarly for $\{ f \ge \ell \}$).
\end{conjecture}

\begin{remark}
In fact, much more is likely true. In particular, we expect the following local uniqueness statement to hold: the probability that
\begin{equation}\label{eq:localuniqueness}
\begin{array}{c}
\text{$\{f \ge \ell\} \cap [-R,R]^d $ contains two or more macroscopic connected components}\\
\text{(of diameter at least $R/100$, say) which are not connected inside $[-R,R]^d$}\\
\end{array}
\end{equation}  
decays very rapidly in $R$ for all $\ell < \ell_c$ (in particular for $\ell=0$). For Bernoulli or Gaussian free field percolation in $d\geq 3$, such a result holds \cite{GM90, DPR18.2}. Item ii) in Theorem~\ref{thm:main} can be regarded as evidence towards its validity for $f$. Another piece of evidence is the recent sharpness result proved by Severo, see Theorem 1.2 of \cite{severo_2021}: it is shown therein that, for any $f$ satisfying Assumption \ref{ass1}, the connection probabilities decay exponentially fast in the subcritical phase while the probability that $[-R,R]^d$ is not connected to infinity is less than $e^{-cR^{d-1}}$ for some $c>0$ in the supercritical phase.
\end{remark}

We further expect the results of the present paper to hold more generally, notably when $f$ denotes the monochromatic random wave, which is the centered stationary smooth Gaussian field whose covariance function is the Fourier transform of the uniform measure on the $(d-1)$-dimensional sphere. The monochromatic random wave is not positively correlated and its covariance decays as $|x|^{-(d-1)/2}$. The following conjecture was stated in \cite{Sa17}. We refer to video simulations\footnote{\label{Garnett}available at \url{https://math.ethz.ch/fim/activities/conferences/past-conferences/2017/random-geometries-topologies/talks/videos-barnett.html}} by A.~Barnett for supporting numerical evidence.

\begin{conjecture}[Sarnak]\label{conj:sarnak}
Let $d\geq 3$ and let $f$ be the monochromatic random wave in $\R^d$. There exists $\delta > 0$ such that, for any $\ell \in [-\delta,\delta]$, a.s.\ there exists an unbounded component in $\{ f = \ell \} \cap \R^3$.
\end{conjecture}

In a forthcoming companion paper, the second author derives this conjecture for large values of $d$ by using the main intermediate result of the present paper and a comparison argument with the Bargmann-Fock model.
\begin{theorem}[\cite{alejandro}]
There exist $d_0\geq 3$ and $\delta > 0$ such that the following holds for each $d\geq d_0$. Let $f$ be the monochromatic random wave in $\R^d$. Then, for any $\ell \in [-\delta,\delta]$, a.s.\ there exists an unbounded component in $\{ f = \ell \} \cap \R^3$.
\end{theorem}

\subsection{Strategy of the proof}

In this section, we explain the general strategy of the proof that there exists an unbounded component in $\{ f \ge 0 \} \cap (\R^2 \times [0,L])$ if $L$ is sufficiently large. Let us fix some small $a>0$ and some large $\beta=\beta(a)$, and let us assume that $q$ satisfies Assumption \ref{ass1} for such a~$\beta$. Below (and in all the paper), for any $d'\leq d$, we routinely view $D\subset\R^{d'}$ as the subset $D \times \{ 0\}^{d-d'}\subset\R^d$.

\medskip

\textbf{A quasi-planar uniqueness property.} The first important aspect of our proof is an analysis of so-called crossing probabilities. While it is known from Russo--Seymour--Welsh type arguments (see Section \ref{ss:rsw}) that in \textit{planar} rectangles crossing probabilities for $\{f\ge0\}$ are bounded away from 0 and 1, it is not difficult to show that these crossing probabilities tend to 1 as soon as one works in a ``thick'' rectangle. More precisely, for instance by approximating $f$ by an $R^\eps$-dependent field (for some suitable $\eps=\eps(\beta)$) and by looking at the 2D slices $[0,R]^2 \times \{kR^\eps\}$ for all $k \in \{0,\dots,R^{a-\eps}\}$, one can deduce that
\[
\lim_{R\rightarrow\infty}\mathbb P \big[\{f\ge0\}\text{ contains a path from left to right in $[0,R]^2 \times [0,R^a]$} \big]=1.
\]
While this may be interpreted as a sign of supercriticality (dimension 3 is expected to be below the upper critical dimension, so that crossing probabilities should remain bounded away from 1 at criticality, and tend to $0$ exponentially fast below criticality) there is still some work to be done to construct the unbounded connected component of $\{f\ge0\}$ ``by hand''.

\medskip

One difficulty comes from the fact that crossings in thick rectangles are not straightforward to combine into longer paths as one can often do in the planar case. In order to circumvent this difficulty, we will prove (a variant of) the following uniqueness property for crossings of such thick rectangles:
\begin{equation}\label{eq:2armsintro}
\begin{array}{c}
\text{With high probability, for large $R$, any two components of}\\
\text{of $\{f \ge 0 \} \cap [0,R]^2$ with diameter $\ge R/100$}\\
\text{are connected by a path in $\{f \ge 0 \} \cap ([0,R]^2 \times [0,R^a])$.}
\end{array}
\end{equation}

Once \eqref{eq:2armsintro} is proved, one can launch a renormalisation procedure to construct an unbounded connected component in $\{f\ge0\}\cap(\mathbb R^2\times[0,R^a])$ for some $R$ sufficiently large
.

\medskip

\textbf{Setting of the proof of \eqref{eq:2armsintro}.} The previous discussion highlights the fact that the proof of \eqref{eq:2armsintro} is the heart of our paper. We now describe it in some detail. First of all it is a good place to mention that we will work with a finite-range approximation $f_r$ (see the definition in Section \ref{ssec:not}) of $f$ instead of $f$ itself, where $f_r$ has range of dependence~$r$. Our goal is essentially to show that the following holds with very high probability for some well-chosen $\gamma\in(0,a)$:
\begin{equation}\label{eq:2armsintrosuite}
\begin{array}{c}
\text{Let us condition on $f_{R^\gamma}$ restricted to $\R^2=\R^2\times\{0\}^{d-2}$ and let $\calC,\calC'$}\\
\text{be two components of $\{ f_{R^\gamma} \ge 0 \} \cap [0,R]^2$ of diameter at least $R/100$.}\\
\text{Then, the (conditional) probability that $\mathcal{C}$ and $\mathcal{C}'$ are connected by}\\
\text{a path in $\{ f_{R^\gamma} \ge 0 \} \cap ([0,R]^2 \times [0,R^a])$ is very close to $1$.}
\end{array}
\end{equation}

\medskip

\textbf{The proof for Bernoulli percolation.} To explain the proof of \eqref{eq:2armsintrosuite}, let us make a brief detour to present an analogous strategy in the context of Bernoulli percolation at $p=1/2$ on $\Z^3$, for which the argument is maybe more transparent. For this model, we can show the result obtained from \eqref{eq:2armsintro} by replacing $R^a$ by a large (but independent of $R$) number $L$. Let us explain this proof. We recall that in this model each edge is open with probability $1/2$ and closed otherwise, and that the states of the edges are independent of each other.

In order to connect two large connected components $\calC$ and $\calC'$ of the square $[0,R]^2$ in $[0,R^2] \times [0,L]$ for some well-chosen $L$, one may proceed as follows. First, recall that by the Russo--Seymour--Welsh theory, two macroscopic connected sets in some planar domain are connected with probability uniformly bounded away from $0$ and $1$. By looking at Bernoulli percolation in the vertical translates $[0,R]^2 \times \{k\}$ of the square (for $k \in \{1,\dots,L\}$) and by using the independence structure of Bernoulli percolation, one obtains that if $L$ is large then with high probability there exists some height $k$ and a cluster $\calC''$ in $[0,R]^2 \times \{ k \}$ whose projection intersects both $\calC$ and $\calC'$. One can actually further prove that with high probability, the projection of $\calC''$ intersects $\calC$ and $\calC'$ in \textit{many} places. Then, one can try to connect $\mathcal{C}$ and $\mathcal{C}'$ to $\mathcal{C}''$ by using the vertical paths made of $k$ edges at each of these many places, and one obtains that $\mathcal{C}$ and $\mathcal{C}'$ are connected in $[0,R]^2 \times [0,L]$ with high probability.

\medskip

\textbf{Adapting the strategy to smooth Gaussian fields.} Adapting this strategy to smooth Gaussian fields raises substantial challenges. First, the field $f$ itself is not finite range and this is the reason for working with an approximation $f_{R^\gamma}$ of it. Second (and more importantly), a new difficulty emerges from the lack of a fundamental property of Bernoulli percolation and various other dependent percolation models: the {\em finite-energy property}, by which we mean that conditioned on everything outside a finite set $A$, the configuration in $A$ takes any possible value with probability bounded from below uniformly in what happens outside. This seemingly harmless property is in fact extremely powerful. For instance when conditioning on $\mathcal C,\mathcal C',\mathcal C''$, one could use it to create vertical paths connecting the connected components with reasonable probability. Not being able to invoke it substantially complicates our argument. In order to overcome this issue, we will rely on one of the main innovations of this paper and prove that large planar clusters in $\{f_{R^\gamma}\geq 0\}$ typically belong to clusters in $\{f_{R^\gamma}\geq 0\}$ which are not confined to thin slabs. More precisely, we show the following result, which we raise to the level of a proposition to highlight its importance in our argument.

\medskip

For $R\ge0$, let $D(R):=\{ x \in \R^2 : |x| \leq R\}$, $\mathcal{P}_t := \{ x \in \R^3 : x_3=t\}$ and recall that if $d' \le d$ and $D\subset \R^{d'}$ then we identify $D$ with $D \times \{0\}^{d-d'}$. We also refer to the definition of $f_r$ (and $r_q$) in the next section.

\begin{proposition}\label{prop:mw}
Let $d \ge 3$ and let $q$ satisfy Assumption \ref{ass1} for some $\beta>d$. There exist $a,\gamma,c,R_0>0$ such that for every $R\ge R_0$, $r\in [r_q,R^\gamma]$ and $\ell \ge 0$,
\[
\mathbb P\Bigg[\begin{array}{c}\text{Every continuous~path in $\{ f_r \geq \ell \} \cap D(2R)$}\\
\text{from $D(R)$ to $\partial D(2R)$ belongs to a connected component}\\
\text{of $\{ f_r \ge \ell \}\cap (D(2R)\times[0,R^a])$ that intersects $\mathcal P_{R^a}$}\end{array}\Bigg]\ge 1-\exp(-R^c).
\]
\end{proposition}


At the risk of repeating ourselves, this proposition will be one of the main novelties of this paper. It uses in a central fashion the fact that $f_r$ is defined on a continuous space, which makes the argument one of the few instances where continuous fields are easier to handle than their lattice analogues. The argument, inspired by the Mermin--Wagner theorem \cite{MerWag} -- see in particular the arguments developed by Pfister \cite{pfister_1981} -- may have other applications of the same kind in the future.

\medskip 

\textbf{End of the proof of \eqref{eq:2armsintro} (for $f_{R^\gamma}$).} With Proposition \ref{prop:mw} at hand, the end of the proof that \eqref{eq:2armsintrosuite} holds with high probability consists in
\begin{itemize}[noitemsep]
\item observing, by Proposition \ref{prop:mw}, with high probability, one can condition on the fact that several random points $u_i$ at a vertical distance $R^a$ from $\calC$ (resp.\ $\calC'$) are likely to be connected to $\calC$ (resp.\ $\calC'$);
\item then trying to connect the $u_i$'s together.
\end{itemize}
While the model is very biased in the neighbourhood of $\calC$ and $\calC'$, we have much more freedom in the neighborhood of the $u_i$'s. This will be crucial for us. Using this property and the previous observation that Russo--Seymour--Welsh type arguments enable to connect such areas with relatively good probabilities will conclude the proof of \eqref{eq:2armsintrosuite} and as a result of \eqref{eq:2armsintro} also (we are omitting a fair amount of details here and postpone the discussion to the corresponding section).

\medskip

\textbf{Organization of the paper.} In Section \ref{sec3}, we give some notations and recall some preliminary results. Section \ref{sec2} contains the proof of Proposition \ref{prop:mw}. Section~\ref{sec4} is devoted to the proof of a Russo--Seymour--Welsh type result with many contact points. At this point, we will have all the tools in hand to prove (an analogue of) \eqref{eq:2armsintro}. This is done in Section~\ref{sec5}. In Section~\ref{sec6}, we state and prove a general renormalization scheme that will be used together with (the analogue of)~\eqref{eq:2armsintro} in the two last sections in order to conclude the proof of our main result Theorem~\ref{thm:main}. The first and second parts $i)$ and $ii)$ of this theorem are proved in Sections \ref{sec7} and \ref{sec8} respectively. 

\medskip

\textbf{Acknowledgments.} The problem was suggested to us by Peter Sarnak. We thank him as well as Vincent Beffara and Damien Gayet for inspiring discussions. We also thank Matthis Lehmkühler very much for providing the proofs of the FKG inequalities from Sect.\ \ref{ssec:app_fkg} to us. Moreover, we thank Damien Gayet who provided the proof of Lem.~\ref{L:top} to us and Thomas Letendre for help with such topological questions. An important part of this work was done while AR and HV were visiting HD-C and P-FR in IHES, that we thank for hospitality.

Finally, we wish to thank an anonymous referee for helpful comments and his or her careful reading of our paper, and we are extremely grateful to David Vernotte for his careful reading of our paper and for drawing our attention to a mistake in Sect.\ \ref{sec5} that is corrected in the present version -- this has actually led to an improvement of the paper (now, we do not need any sprinkling in Prop.\ \ref{prop:mw}).

This project has received funding from the European Research Council (ERC) under the European Union's Horizon 2020 research and innovation programme (grant agreements No.~757296 and No.~851565). The authors acknowledge funding from the SwissMap funded by the Swiss FNS.

\section{Notations and preliminary properties}
\label{sec3}

In this section, we state some notation used in all the paper as well as two preliminary properties, namely the FKG inequalities and the RSW theorem. Let us note that these two properties are not used in Section \ref{sec2}, which contains the proof of Proposition \ref{prop:mw}.

\subsection{Truncation and other notations}\label{ssec:not}

Let $q : \R^d \rightarrow \R^d$ satisfy Assumption \ref{ass1} for some $\beta>d$. Most of our intermediate results deal with a finite-range approximation of $f$ defined by truncating $q$. Below, $|\cdot|$ denotes the Euclidean norm. For $r \geq 1$, let $\chi_r : \R^d \rightarrow [0,1]$ be a smooth isotropic function satisfying
\[
\chi_r(x)=
\begin{cases} 1 \quad \text{if } |x| \leq r/2-1/4,\\
0 \quad \text{if } |x| \geq r/2,
\end{cases}
\]
and whose $k^{th}$ derivatives, for all $k \geq 1$, are uniformly bounded in $x \in \R^d$ and $r \geq 1$. We let $q_r=q\chi_r$ and define $(f,f_r)$ as a continuous modification of the pair
\[
(q\star W,q_r \star W)\, .
\]
The field $f_r$ is $r$-dependent if the sense that $\E [f_r(x)f_r(y)]=0$ if $|x-y| \ge r$. In all the paper, we let
\[
r_q = 1 + \sup\{ r \ge 1 : \text{$q_r$ is identically equal to $0$}\}
\]
(with $\sup \emptyset := 1$). One can note that $q_r$ satisfies Assumption~\ref{ass1} for every $r \ge r_q$ and that the $O$ estimates from this assumption are uniform in $r$. To see that $f$ and $f_r$ are close to each other if $r$ is large, we will use some classical approximation techniques (Cameron--Martin, Kolmogorov and BTIS lemmas), see Section \ref{ssec:approx}.

\medskip

We will use the following notation/conventions in all the paper.

\begin{itemize}[noitemsep]
\item We let $B(R)$ denote the Euclidean (closed) ball of radius $R$ centered at $0$ and we let $B(x,R):=x+B(R)$.
\item If $1 \le d ' \le d$, if we work in $\R^d$ and if $D \subset \R^{d'}$, then we identify $D$ with $D \times \{ 0 \}^{d-d'}$.
\item We let $D(R):=\{ x \in \R^2 : |x| \leq R\}$ and $D(x,R):=x+D(R)=x+(D(R)\times \{0\}^{d-2})$ for any $x \in \R^d$.
\item We let $\mathcal P_t:=\R^2\times\{t\}$.
\item If $U \subset \R^d$, let $\mathcal{F}_U$ be the $\sigma$-algebra on the set of continuous functions $C(\R^d)$ generated by the projections $u \mapsto u(x)$ for $x \in U$. We say that $\phi : C(\R^d) \rightarrow \R$ is measurable if it is $\calF_{\R^d}$-measurable. Let us note that $\calF_{\R^d}$ is the Borel $\sigma$-algebra for the topology of uniform convergence on every compact subset.
\end{itemize}

\subsection{The Gaussian FKG inequalities}\label{ss:FKG}

Let us first state the Fortuin--Kasteleyn--Ginibre (or FKG) inequality for continuous Gaussian fields.

\begin{lemma}[Continuous Gaussian FKG inequality]\label{lem:FKG1*}
Let $\phi,\psi : C(\R^d) \rightarrow \R$ be two non-decreasing bounded measurable functions and let $f$ be a centered continuous Gaussian field on $\R^d$. Assume that $\E [f(x)f(y) ] \ge 0$ for every $x,y \in \R^d$. Then,
\[
\E \left[ \phi(f)\psi(f) \right] \ge \E \left[ \phi(f) \right] \E \left[ \psi(f) \right].
\]
\end{lemma}

This lemma is proven in Section \ref{sssec:FKG} by relying on the analogous result for Gaussian vectors proven by Pitt \cite{Pit}. In the present paper, we also need a generalization of this lemma to ``locally monotone'' functions. Lemma \ref{lem:localFKG} is such a generalization. In this section, we state Corollary \ref{cor:FKG} which is a consequence of Lemma \ref{lem:localFKG} and which is sufficient for us (to see that Corollary \ref{cor:FKG} is indeed a consequence of Lemma \ref{lem:localFKG}, the reader can note that if $U \subset \R^d$ and $\phi : C(\R^d) \rightarrow \R$ is $\calF_U$-measurable, then $\phi$ does not depend on $U^c$ in the sense that $\phi(u)=\phi(v)$ for every $u,v$ that agree in $U$).

\medskip

We first need some vocabulary/notation. Given some $U \subset \R^d$, let $U^r$ denote the $r$-neighborhood of $U$, i.e.\ $U^r := \{ x \in \R^d : \exists y \in U, |x-y|<r \}$. If $\phi : C(\R^d) \rightarrow \R$, say that $\phi$ is non-decreasing in $U$ if $\phi(u) \ge \phi(v)$ for every $u,v$ that agree outside of $U$ and satisfy $u_{|U} \ge v_{|U}$.

\begin{corollary}[Local FKG inequality]\label{cor:FKG}
Let $\phi,\psi : C(\R^d) \rightarrow \R$ be two bounded measurable functions and let $f$ be a centered continuous Gaussian field on $\R^d$. Also, let $r \in (0,+\infty]$ such that $f$ is $r$-dependent (i.e.\ $\E[f(x)f(y)]=0$ for every $x,y \in \R^d$ satisfying $|x-y|\ge r$). Finally, let $\delta>0$ and $U,V \subset \R^d$ and assume that
\begin{itemize}[noitemsep]
\item $\phi$ is non-decreasing in $V^{r+\delta}$ and $\mathcal{F}_U$-measurable,
\item $\psi$ is non-decreasing and $\mathcal{F}_V$-measurable,
\item for every $x,y \in U^{r+\delta}$, $\E[f(x)f(y)]\ge 0$.
\end{itemize}
Then,
\[
\E \left[ \phi(f)\psi(f) \right] \ge \E \left[ \phi(f) \right] \E \left[ \psi(f) \right].
\]
\end{corollary}

\subsection{Russo--Seymour--Welsh theory}\label{ss:rsw}

Box-crossing properties for excursion and level sets of smooth positively correlated Gaussian fields with sufficiently fast decay of correlations have been proven in \cite{BG,BM,bmw,RVa,MV} by relying on \cite{T}. Recently, Köhler-Schindler and Tassion \cite{KT} have developed a strategy that enables to remove the assumption on the speed of decay of the correlations in the case of the excursion sets. Moreover, the constants obtained in \cite{KT} do not depend on the model. In the context of the present paper, this gives Theorem \ref{thm:rsw} below. Note that the theorem also applies to $f_r$ for any $r \ge r_q$. Before stating it, we note that it relies on the following simple lemma, which guarantees planar duality for crossings.

\begin{lemma}\label{lem:mani*}
Assume that $q$ satisfies Assumption \ref{ass1} for some $\beta>d$. Let $g=f_{|\R^2}$, $\ell \in \R$ and let $L \subset \R^2$ be a line. The following holds a.s.:
\begin{itemize}[noitemsep]
\item the sets $\{ g \geq \ell\}$ and $\{ g \leq \ell \}$ are two $\mathcal{C}^1$-smooth $2$-dimensional manifolds with boundary,
\item $\partial \{ g \geq \ell \} = \partial \{ g \leq \ell \} = \{ g = \ell \}$,
\item $\{ g = \ell \}$ intersects $L$ transversally.
\end{itemize}
\end{lemma}

\begin{proof}
It is for instance a direct consequence of Lemma A.9 of \cite{RVa} (in this lemma, one needs the additional hypothesis that the field is not degenerate, which is actually a consequence of the fact that the support of the Fourier transform of $q$ contains an open set -- which holds since $q$ is non-negative, not identically equal to $0$ and $L^1$ -- see e.g.\ Theorem 6.8 of \cite{Wen}).
\end{proof}

Given a continuous field $f$ in $\R^d$ and $\rho_1,\rho_2 > 0$, we let $\cross(\rho_1,\rho_2)$ be the event that there is a path included in $\{f\ge 0\} \cap ([0,\rho_1] \times [0,\rho_2])$ from the left side of $[0,\rho_1]\times[0,\rho_2]$ to its right side.

\begin{theorem}[\cite{KT}]\label{thm:rsw}
Assume that $q$ satisfies Assumption \ref{ass1} for some $\beta>d$. Then, for every $\rho>0$, there exists a constant $c>0$ that depends only on $\rho$ such that, for every $R>0$,
\[
c<\Pro \left[ \cross(\rho R,R) \right] <1-c.
\]
\end{theorem}

It is important to note that $c$ does not depend on $q$.

\begin{proof}
Let $a>0$. In \cite{KT}, Köhler-Schindler and Tassion prove a box-crossing property (with constants that only depend on $a$) for any FKG bond percolation model on $\Z^2$ that satisfies the following two properties: the model is invariant under the symmetries of $\Z^2$ and the probability of the left-right crossing of $[0,n]^2$ is at least $a$ for every $n \ge 1$.

The proof directly extends to positively correlated, stationary, isotropic, centered Gaussian fields as soon as one can prove some smoothness properties about the nodal lines. More precisely, by using for instance Lemma \ref{lem:mani*} (with $\ell=0$) one obtains that $\Pro [ \cross(R,R)]=1/2$ for every $R>0$ and that all the geometric constructions from \cite{KT} hold a.s., and one thus obtains Theorem \ref{thm:rsw}.
\end{proof}

For future reference, let us note that Lemma \ref{lem:mani*} also holds in higher dimensions (the proof is the same):
\begin{lemma}\label{lem:mani_d*}
Assume that $q$ satisfies Assumption \ref{ass1} for some $\beta>d$. Let $d' \le d$, $\ell \in \R$ and let $g=f_{|\R^{d'}}$. Then, the following holds a.s.: the sets $\{ g \geq \ell\}$ and $\{ g \leq \ell \}$ are two $\mathcal{C}^1$-smooth $d'$-dimensional manifolds with boundary. Moreover, $\partial \{ g \geq \ell \} = \partial \{ g \leq \ell \} = \{ g = \ell \}$ a.s.
\end{lemma}

Using Theorem \ref{thm:rsw} in conjunction with Lemma \ref{lem:FKG1*}, one can easily construct various gluing patterns by considering crossings of suitable rectangles. These gluing constructions are classical in planar percolation theory and will be used freely throughout this article. We briefly list a few simple geometric facts related to these constructions and refer the reader to \cite{Werner_notes} for more details on this matter.
\begin{itemize}[noitemsep]
\item[i)] For each $\delta>0$, there exist $c,N>0$ such that for any $R>0$ there exists a collection of less than $N$ rectangles of size $3c R\times c R$ such that any continuous path in $[-R,R]^2$ of diameter at least $\delta R$ must cross at least one of these rectangles widthwise.
\item[ii)] There exists an integer $k$ and rectangles $Q_1,\dots,Q_k$ such that if $\gamma_i$ is a continuous path crossing $Q_i$ in the long direction for $i=1,\dots,k$, then, $\cup_i \gamma_i$ contains a circuit in $D(2)$ surrounding $D(1)$.
\item[iii)] There exists an integer $k'$ and $x_1,\dots,x_{k'}\in\R^2$ such that if $\eta_i$ is a circuit around $D(x_i,1)$ in $D(x_i,2)$ for $i=1,\dots,k'$, then $\cup_i \eta_i$ contains a crossing in the long direction of $[0,20]\times [0,10]$.
\item[iv)] If $0<r_1\le r_2 \le r_3$ with $r_1\le r_2/2$ and $r_2 \le r_3/2$, if there exist three continuous paths $\gamma_1,\gamma_2,\gamma_3$ that cross the annuli $D(r_2) \setminus D(r_1)$, $D(2r_2) \setminus D(r_2/2)$ and $D(r_3) \setminus D(r_2)$ respectively, and if there exist two circuits $\gamma_4,\gamma_5$ in the annuli $D(2r_2) \setminus D(r_2)$ and $D(r_2) \setminus D(r_2/2)$ respectively, then $\cup_i \gamma_i$ contains a path that crosses the annulus $D(r_3) \setminus D(r_1)$ from inside to outside. The same also holds in a half-plane for instance.
\item[v)] Item ii) together with Theorem \ref{thm:rsw} and Lemma \ref{lem:FKG1*} imply that there exists $\eta>0$ such that for each $R\geq 1$ and $r \ge r_q$ the probability that $\{f_r\geq 0\}$ contains a path connecting $D(R)$ to $\partial D(2R)$ is at most $\eta$. If furthermore $R \ge r$, then by considering a well-chosen family of $O(\log(R/r))$ concentric annuli at mutual distance at least $r$, one deduces that the probability that $D(r)$ is connected to $\partial D(R)$ in $\{f_r\geq 0\}$ is at most $C(r/R)^c$ for some universal $c,C>0$.
\item[vi)] By considering the endpoint of a square crossing, one deduces that there exists $c>0$ such that for every $R>0$ and $r \ge r_q$, the probability that $0$ is connected to $\partial D(R)$ by a path in $\{f_r\geq 0\} \cap \{ x \in \R^2 : x_2 \ge 0 \}$ is at least $cR^{-1}$. In fact, this bound is not optimal at all but it will be sufficient for our purposes.
\end{itemize}

\begin{figure}
\begin{center}
\includegraphics[scale=0.47]{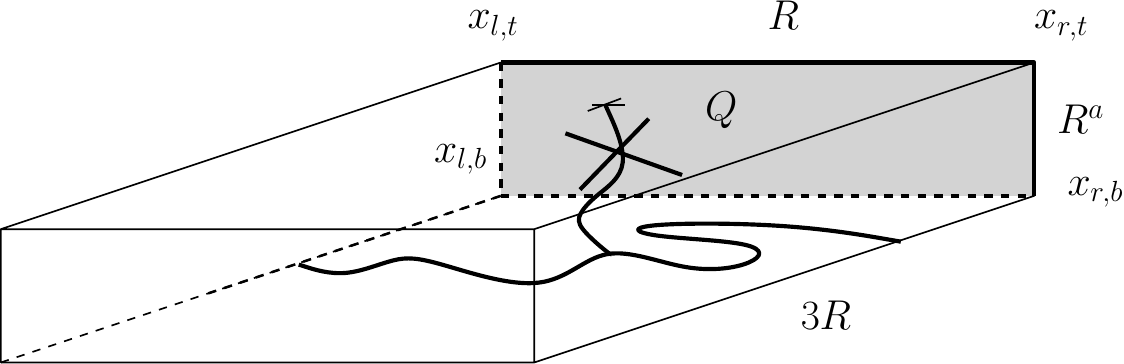}
\vspace{1cm} \includegraphics[scale=0.47]{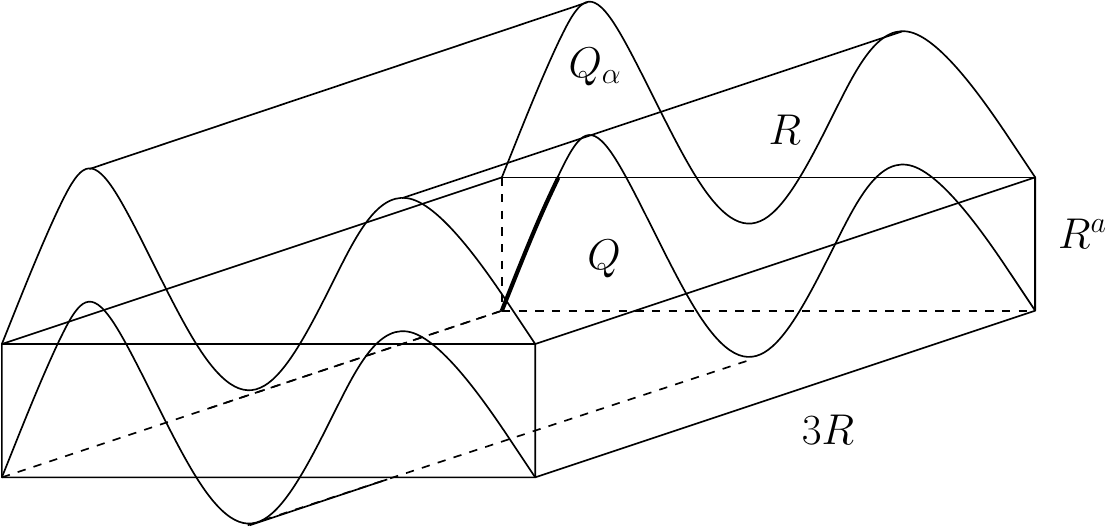}
\caption{i) The quad $Q$ (in grey) and the event $E_\ell(R)$ in $[0,3R] \times Q$; ii) a quad $Q_\alpha$ and the set $[0,3R] \times Q_\alpha$. A segment (in bold) of the bottom side of $Q_\alpha$ connects the left-bottom corner of $Q$ to its top side.\label{fig:Q_BR}}
\end{center}
\end{figure}

\section{Clusters are not confined in slabs}
\label{sec2}


In this section, we prove Proposition \ref{prop:mw}.

\subsection{Setting of the proof}

In this section, we fix some $d \ge 3$ and a function $q : \R^d \rightarrow \R$ that satisfies Assumption~\ref{ass1} for some $\beta>d$.  Also, we work with three numbers $0 < \gamma < a < b < 1$. Throughout the section, $R>0$ and $r \ge r_q$ will denote two scale parameters and $\ell \in \R$. Unless ortherwise stated, the constants below may depend on $q,\gamma,a,b$, but are uniform in $R$, $r$ and $\ell$. Recall that for any $d' \in [1,d]$ and any set $D \subset \R^{d'}$, we identify $D$ with $D \times \{ 0\}^{d-d'}$.

\medskip

We start with a definition.

\begin{definition}\label{def:E_R_H_delta_R}
Let $R>0$ and $\ell \in \R$. We define the event $E_\ell(R)$ as follows: a function $u \in C(\R^d)$ belongs to $E_\ell(R)$ if there exists a connected component $\mathcal C$ of $\{u\geq \ell\}\cap ([0,3R]\times[0,R])$ such that
\begin{enumerate}[noitemsep]
\item[i)] 
$\mathcal C$ contains a crossing from top to bottom in $[0,3R]\times[0,R]$ (i.e.\ from $[0,3R] \times \{0\}$ to $[0,3R] \times \{R\}$) and
\item[ii)] $\calC$ is \textit{not} connected to $[0,3R]\times[0,R]\times\{R^a\}$ by a continuous path in $\{u\geq \ell\}\cap([0,3R]\times[0,R]\times[0,R^a])$.
\end{enumerate}
\end{definition}

As we will see at the end of this subsection, the proof of Proposition~\ref{prop:mw} essentially boils down to bounding the probability of $\Pro[f_r\in E_{\ell}(R)]$. We will do that by introducing a family of events $(E^\alpha_\ell(R) : \alpha \in A)$ for some very large set $A$ that satisfy
\[
\Pro[f_r \in E_\ell(R)] \approx \Pro[f_r \in E^\alpha_\ell(R)]
\]
for every $\alpha$ and some constant $c>0$. If the events $E^\alpha_\ell(R)$ were mutually {\em disjoint}, this would imply -- by summation over $\alpha$ -- that $\Pro[E_\ell(R)]$ is small. What we will actually do is relate the $E^\alpha_\ell(R)$'s to some events $A^\alpha_\ell(x,R)$'s which are mutually disjoint and deduce the desired bound. The events $E^\alpha_\ell(R)$ will be constructed from $E_\ell(R)$ by {\em slightly deforming} the ambient space and proving that the law of the field is not altered too much by this deformation. This is a feature that is specific to this percolation model in the continuum and we are unaware of similar arguments for percolation models on lattices. 

Let us note that similar disjoint events have been used in \cite{Wer95} in order to study Brownian disconnection exponents (see in particular Section 4 therein).

\medskip

The events $(E^\alpha_\ell(R):\alpha\in A)$ are defined as 
\[
\{u\in E^\alpha_\ell(R)\}:=\{u\circ(Id_1 \times F_\alpha \times Id_{d-3}) \in E_\ell(R)\},
\]
where the maps $(F_\alpha,\alpha\in A)$ are introduced in Lemma \ref{lem:rippling_quads} below and, by definition, for $x\in\R^d$,
\[
u\circ(Id_1 \times F_\alpha \times Id_{d-3})(x)=u(x_1,F_\alpha(x_2,x_3),x_4,\dots,x_d)\, .
\]
In a first stage, the reader may skip the details of this lemma and just keep track in their mind of the ``nicknames'' of the three items, and come back to it in the next subsection.

\begin{lemma}[Existence of a large family of rippling quads]\label{lem:rippling_quads}
Recall that we work with some numbers $0<a<b<1$. There exist $c_1,\dots,c_4,R_0 > 0$ such that for every $R\ge R_0$, there exists a family $(F_\alpha)_{\alpha\in A}$ of embeddings (more precisely, the $F_\alpha$'s are $C^2$ and are $C^1$ diffeomorphisms onto their images)
\[
 F_\alpha \, : \, \mathbb R\times[-R^a,2R^a] \rightarrow \R^2
\]
mapping 0 to 0
with the following three properties: 
\begin{enumerate}
\item ($A$ is large) $|A|\geq \exp(c_1R^{1-b})$.
\item (the $F_\alpha$'s are sufficiently different) For each distinct $\alpha,\alpha'\in A$, the following or its analogue obtained by interchanging $\alpha$ and $\alpha'$ holds: there exists a continuous path included in $F_\alpha([0,R]\times\{0\}) \cap F_{\alpha'}([0,R]\times[0,R^a])$ going from  $0$ to $F_{\alpha'}([0,R]\times\{R^a\})$.
\item (the $F_\alpha$'s are almost affine isometries) The $F_\alpha$'s are $C^2$ and for each $\alpha \in A$ and $x \in \mathbb R\times[-R^a,2R^a]$ we have the following:
\begin{itemize}[noitemsep]
\item[i)] ($F_\alpha$ is quasi-affine) $|d^2_xF_\alpha|\leq c_2 R^{2(a-b)}$ (where $d_x^2$ is the Hessian);
\item[ii)] ($F_\alpha$ is a quasi-isometry) there exists a rotation $J^\alpha_x\in O(2)$ such that
\[
|d_xF_\alpha-J^\alpha_x|\leq c_3 R^{2(a-b)};
\]
\item[iii)] ($F_\alpha$ is bi-Lipschitz) for each $y\in \mathbb R\times[-R^a,2R^a]$, 
\[
|x-y|/c_4\leq |F_\alpha(x)-F_\alpha(y)|\leq c_4|x-y|.\]
\end{itemize}
\end{enumerate}
\end{lemma}

The third property above implies that $F_\alpha$ is sufficiently close to an affine isometry. Since $\chi_r$ and $q$ are radial, composition by $Id_1\times F_\alpha\times Id_{d-3}$ will almost preserve the law of $f_r$. In the following lemma, we will use this to compare the probability of the events $E_\ell(R)$ and $E^\alpha_\ell(R)$.

\medskip

\begin{lemma}\label{lem:comparison}
There exists $\eta>0$ that depends only on the dimension $d$ such that we have the following as soon as $\gamma$, $a$ and $1-b$ are less than $\eta$: There exist $c,R_0>0$ such that for every $\ell\ge 0$, $R \ge R_0$ and $r \in [r_q,R^\gamma]$, there exists $\ell' \in \R$ such that
\[
\Pro \left[f_r \in E_\ell(R)\right] \le \min_{\alpha \in A}\Pro \left[f_r \in E^\alpha_{\ell'}(R) \right] + R^{-c}.
\]
\end{lemma}

\begin{remark}
In fact, we prove this result with $\ell' \in [\ell,\ell+CR^{-c}]$ for some $C,c>0$.
\end{remark}

In Section~\ref{sec:proof lemma disjoint}, we will relate the events $E^\alpha_\ell(R)$ to certain events $A^\alpha_\ell(x,R)$ which we will show are mutually disjoint for different $\alpha\in A$ using the second property in Lemma~\ref{lem:rippling_quads}. With some work, this will imply the following upper bound on the probability of $E_\ell^\alpha(R)$.

\begin{lemma}\label{lem:disjoint}
There exist $C,R_0>0$ such that for every $\ell \in \R$, $R \geq R_0$ and $r\ge r_q$,
\[
\min_{\alpha \in A} \Pro \left[ f_r \in E_\ell^\alpha(R) \right] \leq \frac{CR}{|A|}.
\]
\end{lemma}

The proofs of the three lemmata are dispatched in the following three subsections. Before diving into these proofs, let us explain how Proposition~\ref{prop:mw} is derived.

\begin{proof}[Proof of Proposition~\ref{prop:mw}]
Let us start by observing that Lemma \ref{lem:comparison} and Lemma \ref{lem:disjoint} (applied to the level $\ell'$ from Lemma \ref{lem:comparison}) and Item 1 of Lemma \ref{lem:rippling_quads} imply that there exist $\gamma,a,\varepsilon,c,R_0>0$ such that, for every $R \ge R_0$, $r \in [r_q,R^\gamma]$ and $\ell\ge 0$,
\begin{equation}\label{eq:h2}
\Pro \left[ f_r \in E_\ell(R^{1-\eps}) \right] \leq R^{-c}.
\end{equation}
We now implement a fairly classical renormalization argument. Introduce the event $E_\ell(R,\tau)$ as $E_\ell(R)$ except that the height is $\tau$ instead of $R^a$.  Let 
\[
p(r,\tau,R)=\mathbb P[f_r\in E_\ell(R,\tau)].
\]
We now claim that there exists $C>0$ such that for every $L \ge R^\gamma$,
\[
p(r,\tau,7L) \leq (C p(r,\tau,L))^2.
\]
To see this, one can proceed as follows: First, due to independence at distance $L$, one bounds $p(r,\tau,7L)$ by the square of the analogous quantity for a $(21L \times 3L)$-rectangle. Now, one covers each $(21L \times 3L)$-rectangle using $40+21$ $(3L \times L)$ rectangles and observes that at least one is crossed in the easy direction with a crossing that does not go up to height $\tau$, thus implying the above with $C=40+21$.

\medskip

By iterating this argument, we obtain that
\[
p(r,\tau,7^kL) \leq C^{2+2^2+\cdots+2^k} p(r,\tau,L)^{2^k} \leq( C^2 p(r,\tau,L))^{2^k}.
\]
We now fix $\tau=R^{a(1-\varepsilon)}$ and $L=R^{1-\eps}$. Also, we choose $k$ so that $\tfrac17R^\eps<7^k\le R^\eps$ and set $\overline R=7^kR^{1-\eps}\in [R/7,R]$. Using \eqref{eq:h2}, we may find $R_1,c'>0$ such that for $R\ge R_1$,
\[
p(r,R^{a(1-\varepsilon)},\overline R) \le (C^2R^{-c})^{2^k} \le \exp(-R^{c'}).
\]
It only remains to observe that for the event in the statement of Proposition~\ref{prop:mw} to occur (with $2R$ instead of $R$), there must be one out of $O(1)$ rectangles of size $3\overline R\times\overline R$ that are crossed by a path which is not in a connected component connected to the top. The claim therefore follows by a union bound.
\end{proof}

\subsection{Proof of Lemma~\ref{lem:rippling_quads}}

Let $R>0$ and introduce the notation $Q:=[0,R]\times [0,R^a]$ and $S:=\R \times [-R^a,2R^a]$.

\medskip

\textbf{Construction of the family $(F_\alpha)_{\alpha\in A}$.} Fix $N=\lfloor R^{1-b}/3 \rfloor$ and $A=\{-1,1\}^N$. 
Fix $e_1:[0,3]\rightarrow[0,1]$ a smooth function which vanishes on $[0,1]$ and is equal to $1$ on $[2,3]$ and write $e_{-1}=-e_1$. 

Given $\alpha=(\alpha_1,\dots,\alpha_N)\in A$ for some $N\geq 1$, define a smooth function $e_\alpha:\R\rightarrow \R$ recursively (note that $e_\alpha$ is piecewise constant outside of $[0,3N]$) as follows: 
\begin{itemize}[noitemsep]
\item[(a)] $e_\alpha(x)=0$ for $x\le 0$, 
\item[(b)] for each $0\le k<N$ and $3k\le x\le 3(k+1)$, $e_\alpha(x):=e_{\alpha_{k+1}}(x-3k)+e_\alpha(3k)$,
\item[(c)] $e_\alpha(x)=e_\alpha(3N)$ for $x\ge 3N$.
\end{itemize}
We then define
\begin{align*}
\gamma_\alpha:\quad \mathbb R&\rightarrow  \R^2,\\
t&\mapsto (t,R^ae_\alpha(t/R^b)).
\end{align*}
In addition, define $\nu_\alpha(t)$ as the $\pi/2$-counterclockwise rotation of $\gamma_\alpha'(t)$ 
and finally $F_\alpha$ as follows:
\begin{align*}
F_\alpha:\quad S\ \ &\rightarrow F_\alpha(S) \subset \R^2,\\
(t,s)&\mapsto \gamma_\alpha(t)+s\nu_\alpha(t).
\end{align*}
For future reference, let $Q_\alpha=F_\alpha(Q)$ (see Figure \ref{fig:Q_alpha_alpha_prime}).

\begin{figure}
\begin{center}
\includegraphics[scale=0.6]{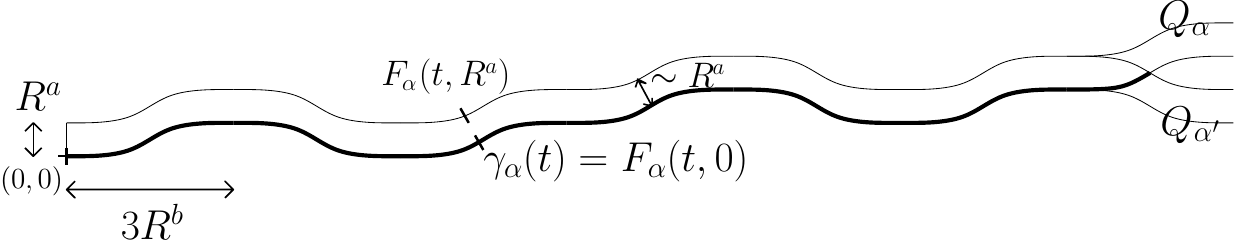}
\caption{Two quads $Q_\alpha$ and $Q_{\alpha'}$ until they differ and (in bold) the path included in the bottom side of $Q_\alpha$ that connects $0$ to the top side of $Q_{\alpha'}$.\label{fig:Q_alpha_alpha_prime}}
\end{center}
\end{figure}

\medskip

\textbf{Checking that the construction satisfies the required properties.} We first note that the first point is obvious as
\[
|A| = 2^N\geq \exp(c_1R^{1-b})
\]
for some constant $c_1>0$ independent of everything else. Next, the second point of the lemma is easily checked to hold for every two distinct $\alpha,\alpha'\in A$ by considering the first letter where the two words differ (see Figure \ref{fig:Q_alpha_alpha_prime}). Let us also note that the third point will imply that the $F_\alpha$'s are embeddings (if $R$ is sufficiently large). Indeed, by 3.iii) they are bijections and by 3.ii) they are local $C^1$-diffeomorphisms.

\medskip

All that remains is to show the third point of the lemma. To show it, we need the following claims whose proofs follow by direct computation and are left for the reader. In these claims, the constants only depend on the choice of the function $e_1$.
\begin{claim}\label{cl:path_regularity_estimates_0}
There exists $C_0>0$ such that the following holds for each $\alpha \in A$ and $t \in \R$,
\begin{itemize}[noitemsep]
\item $1-|\gamma'_\alpha(t)| = 1-|\nu_\alpha(t)|\le C_0R^{2(a-b)}$;
\item the absolute value of the second component of $\gamma'_\alpha(t)$ is no greater than $C_0R^{a-b}$.
\end{itemize}
\end{claim}
\begin{claim}\label{cl:path_regularity_estimates}
There exists $C_1>0$ such that for each $\alpha\in A$ and each $t\in\R$,
\[
|\gamma_\alpha''(t)|+|\gamma_\alpha'''(t)|\leq C_1 R^{a-2b} \mathbf 1_{0\le t\le R}.
\]
\end{claim}
Let us now conclude the proof of the third point. Fix $\alpha\in A$. To begin with,
\begin{align*}
d
F_\alpha&=(\gamma_\alpha'(t)+s\nu_\alpha'(t))dt+\nu_\alpha(t)ds,\\ d^2
F_\alpha&=(\gamma_\alpha''(t)+s\nu_\alpha''(t))dt^2+2\nu_\alpha'(t)dsdt .
\end{align*}
By Claim~\ref{cl:path_regularity_estimates}, we get
\begin{align*}
dF_\alpha&=\gamma_\alpha'(t)dt+\nu_\alpha(t)ds+O(R^{2(a-b)}),\\
 d^2F_\alpha&=O(R^{2(a-b)}).
\end{align*}
The second estimate establishes the quasi-affine property i). Moreover, by the first estimate of Claim~\ref{cl:path_regularity_estimates_0} (and since $\nu_\alpha(t)$ and $\gamma_\alpha'(t)$ are orthogonal), this also establishes the quasi-isometry property ii).

\medskip

All that remains is to prove iii). Let $(s_1,t_1),(s_2,t_2) \in S$ and let us note that
\[
F_\alpha((s_1,t_1))-F_\alpha((s_2,t_2))=\gamma_\alpha(t_1)-\gamma_\alpha(t_2)+
(s_1-s_2)\nu_\alpha(t_1)+s_2(\nu_\alpha(t_1)-\nu_\alpha(t_2)).
\]
By Claim \ref{cl:path_regularity_estimates}, 
\[
\|s_2(\nu_\alpha(t_1)-\nu_\alpha(t_2))\|=O(|t_1-t_2|R^{2(a-b)})=o(|t_1-t_2|).
\]
The first component of $\gamma_\alpha(t_1)-\gamma_\alpha(t_2)$ equals $t_1-t_2$ and by Claim \ref{cl:path_regularity_estimates_0} its second component is $O(|t_1-t_2|R^{a-b})=o(|t_1-t_2|)$. Moreover, still by Claim \ref{cl:path_regularity_estimates_0}, the first component of $(s_1-s_2)\nu_\alpha(t_1)$ is $O(|s_1-s_2|R^{a-b})=o(|s_1-s_2|)$ and  the absolute value of its second component is $\Theta(|s_1-s_2|)$. As a result, the first component of $F_\alpha((s_1,t_1))-F_\alpha((s_2,t_2))$ is $\Theta(|t_1-t_2|)+o(|s_1-s_2|)$ and its second component is $\Theta(|s_1-s_2|)+o(|t_1-t_2|)$. This implies the desired bi-Lipschitz property (for instance by distinguishing between the cases $|t_1-t_2|\geq |s_1-s_2|$ and $|t_1-t_2|\leq |s_1-s_2|$).

\subsection{Proof of Lemma \ref{lem:comparison}}

We keep the notation $Q=[0,R]\times [0,R^a]$ and $S=\R \times [-R^a,2R^a]$ from the previous section. Recall that $q_r=q\chi_r$ and let $G_\alpha = Id_{1} \times F_\alpha \times Id_{d-3}$ for every $\alpha \in A$.

\medskip

We first note that $f_r \in E^\alpha_\ell(R)$ if and only if $g^\alpha_r \in E_\ell(R)$, where $g^\alpha_r$ is the Gaussian field defined for $x \in \R \times Q \times \R^{d-3}$ by
\begin{equation}
\label{eq:g_alpha_restricted_integral}
g^\alpha_{r}(x):= f_r(G_\alpha(x)) = \int q_r(G_\alpha(x)-y)dW_y =\int_{G_\alpha(\R \times S \times \R^{d-3})} q_r(G_\alpha(x)-y)dW_y,
\end{equation}
where the last equality is justified for $R$ sufficiently large since $q_{r}$ is supported in $B_{r}$, $r\le R^\gamma$, and by Item 3.iii) of Lemma \ref{lem:rippling_quads}. We wish to prove that $g^\alpha_{r}$ behaves like $f_{r}$. To this purpose, consider the following intermediary field defined for $x\in \R \times S \times \R^{d-3}$ by 
\[
h^\alpha_{r}(x):=\int_{\R \times S \times \R^{d-3}} q_{r}\big(G_\alpha(x)-G_\alpha(y)\big)|\det(\text{Jac}_y G_\alpha)|^{1/2}dW_y .
\]
By \eqref{eq:g_alpha_restricted_integral}, the field ${(h^\alpha_{r})}_{|\R \times Q \times \R^{d-3}}$ has the same law as $g^\alpha_{r}$ if $R$ is sufficiently large (to prove this, compute the covariance of these Gaussian fields and use the change of variables $u=G_\alpha(y)$). Hence, in order to prove the lemma, it is sufficient to prove that (for $\gamma$, $a$ and $1-b$ sufficiently small) there exist $c,R_0>0$ such that for any $R\geq R_0$ and $r\in[r_q,R^\gamma]$, there exists $\ell'\in \R$ such that
\begin{equation}\label{eq:mw_comparison_2}
\Pro [ f_r \in E_\ell(R)] \le \min_{\alpha \in A}\Pro [ h^\alpha_{r} \in E_{\ell'}(R) ] + R^{-c}.
\end{equation}

Let us now introduce some stability events for percolation clusters, that will enable us to overcome the fact that stability results such as Lemma~\ref{lem:Cameron-Martin} cannot be applied to $E_\ell(R)$ (because this event is not a union/intersection of a small number of monotonic events).

If $\ell<\ell'$ are two levels, we let $\text{Stab}_{\ell,\ell'}(R)$ denote the event defined as follows (note that this event only depends on the function restricted to the rectangle $[0,3R] \times [0,R] \subset \R^2=\R^2\times\{0\}^{d-2}$): a function $u \in C(\R^d)$ belongs to $\text{Stab}_{\ell,\ell'}(R)$ if every connected component of $\{u \ge \ell\}\cap([0,3R]\times [0,R])$ that contains a continuous path from $[0,3R] \times \{0\}$ to $[0,3R] \times \{R\}$ also contains such a path $\gamma$ with the further property that $u_{|\gamma}\ge \ell'$.

\smallskip

Let us now make the following observation: For every $\delta>0$ and every $u,v \in C(\R^d)$, at least one of the following properties does not hold:
\begin{itemize}[noitemsep]
\item $u \in E_\ell(R)$;
\item $u \in \textup{Stab}_{\ell,\ell+2\delta}(R)$;
\item $\| u - v \|_{\infty,[0,3R]\times Q \times \{0\}^{d-3}} \le \delta$;
\item $v \notin E_{\ell+\delta}(R)$.
\end{itemize}

As a result, for every $\delta>0$ we have
\begin{multline*}
\Pro [ f_r \in E_\ell(R)] \le \min_{\alpha \in A}\Pro [ h^\alpha_{r} \in E_{\ell+\delta}(R) ]\\
+ \max_{\alpha \in A} \Pro \big[ \| f_r - h_r^\alpha \|_{\infty,[0,3R]\times Q \times \{0\}^{d-3}} \ge \delta \big] + \Pro \big[ f_r \notin \textup{Stab}_{\ell,\ell+2\delta}(R) \big].
\end{multline*}

We thus obtain that Lemma \ref{lem:comparison} is a consequence of the following two lemmas:

\begin{lemma}\label{lem:comp1}
Fix some $\varepsilon \in (0,1)$. There exists a constant $\theta > 0$ that depends only on $\varepsilon$ such that the following holds as soon as $\gamma<1-\varepsilon$: There exists $R_0>0$ such that for every $\ell \ge 0$, $R \ge R_0$ and $r \in [r_q,R^\gamma]$,
\[
\Pro \big[ f_r \notin \textup{Stab}_{\ell,\ell+\delta}(R) \big] \le R^{-\theta},
\]
where $\delta:=R^{-2+\theta}$.
\end{lemma}

\begin{lemma}\label{lem:comp2}
For every $\theta>0$, there exists $\eta>0$ (that depends only on $\theta$ and on the dimension $d$) such that the following holds as soon as $\gamma$, $a$ and $1-b$ are less than $\eta$: There exist $c,R_0>0$ such that for every $R \ge R_0$ and $r \in [r_q,R^\gamma]$,
\[
\max_{\alpha \in A} \Pro \big[ \| f_r - h_r^\alpha \|_{\infty,[0,3R]\times Q \times \{0\}^{d-3}} \ge R^{-2+\theta} \big] \le \exp(-R^c).
\]
\end{lemma}

\begin{proof}[Proof of Lemma \ref{lem:comp1}]
In this proof, we use the notion of stratified gradient of some function $u \in C^2(\R^2)$ with respect to the rectangle $[0,3R] \times [0,R]$. The stratified gradient $\nabla_x^s u$ is defined as the usual ($2$-dimensional) gradient if $x$ does not belong to the boundary of the rectangle; it is defined as the one-dimensional gradient $\nabla_x (u_{|L})$ if $x$ belongs to some side $L$ of the rectangle excluding corners, and $\nabla_x^s u := 0$ if $x$ is a corner of the rectangle.

\medskip

Let $\delta >0$.
\begin{claim}\label{cl:morse}
Assume that $f_r \notin \textup{Stab}_{\ell,\ell+\delta}(R)$. Then, there exist a connected component $C$ of $\{f_r \ge\ell\} \cap ([0,3R] \times [0,R])$ and a point $x \in C$ such that:
\begin{itemize}[noitemsep]
\item $C$ contains a continuous path from $B_R:=[0,3R]\times\{0\}$ to $T_R:=[0,3R]\times\{R\}$;
\item $f_r(x) \in [\ell,\ell+\delta]$ and $\nabla_x^s f_r = 0$.
\end{itemize}
\end{claim}
\begin{proof}
Let $K$ denote the union of all connected components of $\{f_r \ge\ell\} \cap ([0,3R] \times [0,R])$ that contain a continuous path from $T_R$ to $B_R$. Our aim is to prove the following claim: Assume that there is no $x \in K$ such that $\nabla_x^s f_r = 0$ and $f_r(x) \in [\ell,\ell+\delta]$. Then, every connected component of $K$ contains a continuous path $\gamma$ from $T_R$ to $B_R$ such that $f_r{|\gamma} \ge \ell+\delta$.

\medskip

Let us prove this claim. To this purpose, let $K^\varepsilon$ (resp.\ $\overline{K}^\varepsilon$) denote the open (resp.\ closed) $\varepsilon$-neighborhood of $K$. We fix some $\varepsilon>0$ such that $(f_r)_{|K^{2\varepsilon}\setminus K} < \ell$ and, by using smooth Urysohn's lemma (applied to the compact set $(K^{2\varepsilon})^c$ included in the open set $(\overline{K}^\varepsilon)^c$, both seen as subsets of $[0,3R] \times [0,R]$), we construct a function $\widetilde{f}_r \in C^2(\R^2)$ such that
\begin{itemize}[noitemsep]
\item $(\widetilde{f}_r)_{|K^\varepsilon}=(f_r)_{|K^\varepsilon}$ and
\item $(\widetilde{f}_r)_{|K^c} < \ell$.
\end{itemize}
We note that that there is no $x \in [0,3R]\times[0,R]$ such that $\widetilde{f}_r(x)\in [\ell,\ell+\delta]$ and $\nabla_x^s \widetilde{f}_r = 0$, and we apply a result from stratified Morse theory to $\widetilde{f}_r$ as follows: By \cite[Proposition in Section 3.2 of Part I]{morse}, there exists a homeomorphism $\varphi$ from $K = \{\widetilde{f}_r \ge\ell\} \cap ([0,3R] \times [0,R])$ to $L := \{\widetilde{f}_r \ge\ell+\delta\} \cap ([0,3R] \times [0,R])$ such that both $\varphi$ and $\varphi^{-1}$ send a point of $B_R$ (resp.\ $T_R$) on a point of $B_R$ (resp.\ $T_R$). The existence of a homeomorphism between $K$ and $L$ implies that the number of connected components of $L$ is the same as the number of connected components of $K$. Since every connected component of $L$ is included in a connected component of $K$, we obtain that every connected component of $K$ contains a component of $L$. Moreover, the property of $\varphi$ implies that every component of $L$ contains a path from $B_R$ to $T_R$. This concludes the proof of the desired result for $\widetilde{f}_r$, which implies the desired result for $f_r$.
\end{proof}

Let us now tile the rectangle $[0,3R]\times[0,R]$ with $\asymp R^2$ unit squares $S_i$ and let us note that, for every $h>0$, there exists $C_h>0$ that depends only on $h$ and $q$ such that
\begin{equation}\label{eq:ns_avec_tau}
\forall i, \quad \Pro [ \exists x \in S_i, \nabla_x^s f_r=0 \text{ and } f_r(x) \in [\ell,\ell+\delta] \big] \le C_h \delta^{1-h}.
\end{equation}

This is for instance written at the end of the proof of \cite[Lemma 7]{NS} (applied to $m=2$, $\beta=1$, $\tau=\delta$ and $t=h/3$ -- let us note that the fact that the constant $C_h$ above -- of which the reader can find an expression in \cite{NS} -- is uniform in $r$ is a consequence of classical Gaussian estimates such as Dudley’s theorem and the BTIS inequality, both applied to the Gaussian field $(f_r(x),\nabla_x^s f_r,(\nabla^s)^2_x f_r)_{x \in S_i}$, see \cite{adler_taylor,AW}).

\medskip

Let us end the proof by using \eqref{eq:ns_avec_tau} as well as the RSW theorem -- Theorem \ref{thm:rsw} (the use of the RSW theorem here is the only reason why Lemma \ref{lem:comp1} and Proposition \ref{prop:mw} are stated for levels $\ell \ge 0$). By the RSW theorem and the independence between sets at distance greater than $r$, there exists a constant $c>0$ that depends only on $\varepsilon$ such that, if $\gamma<1-\varepsilon$ and $r \in [r_q,R^\gamma]$, we have
\begin{multline}\label{eq:rsw_dans_stab}
\forall i, \quad \Pro \big[ \text{$\exists$ a cont.\ path in $\{f_r \ge \ell\}$ from $S_i^r$ to $T_R$ and such a path from $S_i^r$ to $B_R$} \big]\\
\le R^{-c},
\end{multline}
where $S_i^r$ is the set of all points in $\R^2$ at distance less than $r$ from $S_i$.

\medskip

Let $\delta = R^{-2+\theta}$ as in the statement of the lemma. By using \eqref{eq:ns_avec_tau}, \eqref{eq:rsw_dans_stab} and the independence between sets at distance greater than $r$, we have
\[
\forall i, \quad \Pro \big[ \exists x \in S_i \text{ as in Claim \ref{cl:morse}} \big] \le C_h \delta^{1-h} R^{-c} = C_h R^{(-2+\theta)(1-h)-c}.
\]
Choosing for instance $h=\theta=c/1000$ and summing over $i$ imply the desired result.
\end{proof}

Let us now start the proof of Lemma \ref{lem:comp2}. We are going to prove that the infinite norm of $f_{r}-h^\alpha_{r}$ on $[0,3R] \times Q \times [0,1]^{d-3}$ is typically small. We start with a claim.
\begin{claim}\label{cl:var_est}
There exists $C>0$ such that for every vector  $u \in \R^d$ of Euclidean norm $1$, every $r > 0$ and $x \in \R \times Q \times \R^{d-3}$, 
\[
\forall \alpha \in A, \quad \text{\textup{Var}}(f_{r}(x)-h^\alpha_{r}(x))+\text{\textup{Var}}(d_x(f_{r}-h^\alpha_{r})(u)) \leq C \frac{r^{4+d}}{R^{4(b-a)}}.
\]
\end{claim}

Before proving the claim, let us  conclude the proof of Lemma \ref{lem:comp2}.

\begin{proof}[Proof of Lemma \ref{lem:comp2}]
(N.B.: All the constants in this proof are uniform in $\alpha$.) Fix some $\theta>0$ and let $\calK=K\times [0,1]^{d-3} \subseteq  \R \times Q \times [0,1]^{d-3}$ with $K$ a unit cube. Claim \ref{cl:var_est} and Kolmogorov's theorem (see for instance \cite[Appendix A.9]{NS}\footnote{We apply this result to $k=1$ and we use that $\E[\partial^\alpha f(x) \partial^\beta f(y)]=\partial^\alpha_x\partial^\beta_y K(x,y)$ where $K$ is the covariance of $f$, and in particular that $\partial^\alpha_x\partial^\beta_y K(x,y)=0$ if $x=y$ and $|\alpha|+|\beta|$ is odd since $f$ is stationary.}) imply that for all $r\in[r_q,R^\gamma]$,
\[
\E \left[ \|f_r-h^\alpha_{r}\|_{\infty,\calK} \right] = O \big(R^{-2+\theta_1}\big),
\]
where $\theta_1:=\gamma(2+d/2)+2(a+1-b)<\gamma(2+d/2)+3(a+1-b)=:\theta_2$. The BTIS inequality (see for instance \cite[Theorem 2.9]{AW}) gives that for $R$ large enough and some constant $c_0>0$,
\[
\Pro \big[ \|f_r-h^\alpha_{r}\|_{\infty,\calK} \geq R^{-2+\theta_2}\big] \le \exp(-R^{c_0}).
\]
Taking a union bound on unit cubes $\calK$ covering $[0,3R]\times Q \times [0,1]^{d-3}$ gives that for $R$ large enough,
\begin{equation}\label{eq:mw_comparison_1}
\Pro \big[ \|f_r-h^\alpha_{r}\|_{\infty,[0,3R] \times Q \times [0,1]^{d-3}} \geq R^{-2+\theta_2}  \big]\le C_1R^{2+a} \exp(-R^{c_0})\le  \exp(-R^{c_1}),
\end{equation}
for some $C_1,c_1>0$. We obtain the desired result by noting that if $\gamma$, $a$ and $1-b$ are sufficiently small, then $\theta_2\le \theta$.

This leaves the proof of Claim \ref{cl:var_est} to conclude the proof of the lemma.
\end{proof}


\begin{proof}[Proof of Claim \ref{cl:var_est}]
Fix $x \in \R \times Q \times \R^{d-3}$. We start by bounding $\mathrm{Var}(f_{r}(x)-h^\alpha_{r}(x))$. First, we have
\[
f_{r}(x)-h^\alpha_{r}(x)=\int_{\R \times S \times \R^{d-3}} \big[ q_{r}(x-y) - q_{r}(G_\alpha(x)-G_\alpha(y))|\det(\text{Jac}_y G_\alpha)|^{1/2} \big] dW_y
\]
and therefore
\[
\text{Var}(f_r(x)-h^\alpha_{r}(x)) = \int_{\R \times S \times \R^{d-3}} \big[ q_{r}(x-y) - q_{r}(G_\alpha(x)-G_\alpha(y))|\det(\text{Jac}_y G_\alpha)|^{1/2} \big]^2 dy.
\]
We let $L^\alpha_x:=Id_{1}\times J^\alpha_x \times Id_{d-3}$, where $J_x^\alpha \in O(2)$ is the rotation from Lemma \ref{lem:rippling_quads}. Below, we will use several times without mentioning it that the $k^{th}$ derivatives of $q_{r}$ with $k \leq 2$ are bounded. By Items 3.i) and 3.ii) of  Lemma \ref{lem:rippling_quads},  we have
\begin{equation}\label{eq:est_var1}
|\det(\text{Jac}_y G_\alpha)| = 1+O(R^{2(a-b)})
\end{equation}
and
\begin{align}\label{eq:est_var2}
q_{r}(G_\alpha(x)-G_\alpha(y)) &= q_{r}\big(d_xG^\alpha(x-y) + |x-y|^2 O(R^{2(a-b)})  \big)\\
\nonumber&= q_{r}\big(L^\alpha_x(x-y) + |x-y| O(R^{2(a-b)}) + |x-y|^2 O(R^{2(a-b)})  \big)\\
\nonumber&= q_{r}(L^\alpha_x(x-y)) + (|x-y|+|x-y|^2) O(R^{2(a-b)}).
\end{align}
Since $q_{r}$ is radial, $q_{r}(L^\alpha_x(x-y))=q_{r}(x-y)$ and therefore \eqref{eq:est_var1} and \eqref{eq:est_var2} imply that
\begin{align}\label{eq:q-q_circ_G}
q_{r}(x-y) - &q_{r}\big(G_\alpha(x)-G_\alpha(y)\big)|\det(\text{Jac}_y G_\alpha)|^{1/2} \\
\nonumber&=  q_{r}(x-y) - \big[ q_{r} (x-y) + (|x-y|+|x-y|^2) O(R^{2(a-b)}) \big] (1+O(R^{2(a-b)}))\\
\nonumber&= (1+|x-y|+|x-y|^2)O(R^{2(a-b)}).
\end{align}
(The constant 1 comes from the fact that $q_r$ is bounded.) By Item 3.iii) of Lemma~\ref{lem:rippling_quads}, $|G_\alpha(x)-G_\alpha(y)| \asymp |x-y|$ so there exists a constant $C>0$ such that for every $|x-y|\geq Cr$, 
\begin{equation}\label{eq:q_outside}
q_{r}(x-y) = q_{r}(G_\alpha(x)-G_\alpha(y))= 0.
\end{equation}
By combining \eqref{eq:q-q_circ_G} and \eqref{eq:q_outside} (more precisely, by using \eqref{eq:q-q_circ_G} on a ball of radius of order $r$ around $x$ and \eqref{eq:q_outside} on the complement of this ball) we obtain that
\[
\text{Var}(f_{r}(x)-h^\alpha_{r}(x)) = O(R^{4(a-b)}) \int_0^{O(r)} (1+s+s^2)^2s^{d-1}ds
= O \big( \frac{r^{4+d}}{R^{4(b-a)}} \big),
\]
which is the desired result.

\medskip

Let us now estimate $\text{\textup{Var}}(d_x(f_{r}-h^\alpha_{r})(u))$. We have
\[
d_x(f_{r}-h^\alpha_{r})(u)=\int_{\R \times S \times \R^{d-3}}d_{x-y}q_{r}(u)-d_{G_\alpha(x)-G_\alpha(y)}q_{r}\circ d_xG_\alpha(u)|\det(\text{Jac}_y G_\alpha)|^{1/2} dW_y.
\]
A computation similar to \eqref{eq:est_var2} gives
\begin{align}
\nonumber d_{G_\alpha(x)-G_\alpha(y)}q_{r}\circ d_xG_\alpha(u)&= d_{G_\alpha(x)-G_\alpha(y)} q_{r} (L^\alpha_x(u)) + (|x-y|+|x-y|^2)O(R^{2(a-b)})\\
\label{eq:est_var3}&= d_{L^\alpha_x(x-y)} q_{r} (L^\alpha_x(u)) + (|x-y|+|x-y|^2)O(R^{2(a-b)}).
\end{align}
By \eqref{eq:est_var1} and \eqref{eq:est_var3}, we have
\begin{align*}
&d_{x-y}q_{r}(u)-d_{G_\alpha(x)-G_\alpha(y)}q_{r}\circ d_xG_\alpha(u) |\det(\text{Jac}_y G_\alpha)|^{1/2}\\
&= d_{x-y}q_{r}(u)- \big[ d_{L^\alpha_x(x-y)} q_{r} (L^\alpha_x(u)) + (|x-y|+|x-y|^2)O(R^{2(a-b)}) \big] (1+O(R^{2(a-b)}))\\
&= d_{x-y}q_{r}(u)-d_{L_x^\alpha(x-y)}q_{r}(L_x^\alpha(u))+(1+|x-y|+|x-y|^2)O(R^{2(a-b)}) \\
&= (1+|x-y|+|x-y|^2)O(R^{2(a-b)}) .
\end{align*}
(In the last line we use that  $q_{r}$ is radial). From there, the proof is the same as for $\text{Var}(f_{r}(x)-h^\alpha_{r}(x))$.
\end{proof}

\subsection{Proof of Lemma \ref{lem:disjoint}}\label{sec:proof lemma disjoint}

Given $x \in [0,3R] (= [0,3R] \times \{0\}^{d-1})$, we define the events $A^\alpha_\ell(x,R)$ as
\[
\{u\in A^\alpha_\ell(x,R)\}:=\{u\circ(Id_1 \times F_\alpha \times Id_{d-3}) \in A_\ell(x,R)\},
\]
where $A_\ell(x,R) \subset E_\ell(R)$ is as follows: a function  $u \in C(\R^d)$ belongs to $A_\ell(x,R)$ if there exists a connected component $\mathcal C$ of $\{u\geq \ell\}\cap ([0,3R]\times[0,R])$ such that
\begin{enumerate}[noitemsep]
\item[i)] 
$\mathcal C$ contains a crossing from top to bottom in $[0,3R]\times[0,R]$;
\item[ii)] $\mathcal C$ is \textit{not} connected to $[0,3R]\times[0,R]\times\{R^a\}$ by a continuous path in $\{u\geq \ell\}\cap([0,3R]\times[0,R]\times[0,R^a])$;
\item[iii)] $x \in \calC$.
\end{enumerate}

The reader can think about these ``$A$-events'' as the same as the ``$E$-events'' except that they are arm-type events (starting from $x$) rather than crossing-type events. 
We also let
\[
B^\alpha_\ell(x,R) := A^\alpha_\ell(x,R) \setminus \bigcup_{y \in [0,3R] \times \{ 0 \}^{d-1}: \, y_1 < x_1} A^\alpha_\ell(y,R),
\]
which corresponds to the event that $x$ is the ``leftmost point'' for which $A^\alpha_\ell(x,R)$ occurs. In particular, we have $E^\alpha_\ell(R)=\sqcup _{x  \in [0,3R]} B^\alpha_\ell(x,R)$. The following claim translates this decomposition into a tractable formula.

\begin{claim}\label{cl:co-area}
There exists $R_0>0$ such that for every $\ell \in \R$, $R \ge R_0 $ and $r \ge r_q$,
\begin{multline*}
\Pro [ f_{r} \in E^\alpha_\ell(R) ]
=\Pro [ f_{r} \in A^\alpha_\ell(0,R) ]+ \int_0^{3R} \E [ |\partial_{x_1} f_{r}(x)| 1_{B^\alpha_\ell(x,R)} \mid f_{r}(x)=\ell ] \gamma_{f_{r}(x)}(\ell) dx_1,
\end{multline*}
where $\gamma_{f_{r}(x)}$ is the density function of $f_{r}(x)$, and $x:=(x_1,0,\dots,0)$.
\end{claim}

The proof of Claim~\ref{cl:co-area} is given in Section~\ref{ssec:claim_co-area}. Let us now conclude the proof of the lemma. We have
\begin{multline*}
|A|\min_{\alpha \in A} \Pro [ f_{r} \in E^\alpha_\ell(R) ] \leq \sum_{\alpha \in A} \Pro [ f_{r} \in E^\alpha_\ell(R) ]\\
= \sum_{\alpha \in A} \Pro [ f_r\in A^\alpha_\ell(0,R) ] +  \sum_{\alpha \in A}  \int_0^{3R} \E [ |\partial_{x_1} f_{r}(x)| 1_{B^\alpha_\ell(x,R)} \mid f_{r}(x)=\ell ] \gamma_{f_{r}(x)}(\ell) dx.
\end{multline*}

The crucial point is that Item 2 of Lemma \ref{lem:rippling_quads} implies that for every fixed $x$, the $(A^\alpha_\ell(x,R):\alpha\in A)$ are disjoint. This can be seen as follows:
\begin{itemize}
\item Let $\alpha \ne \alpha'$. If $A^\alpha_\ell(x,R) \cap A^{\alpha'}_\ell(x,R)$ holds then there are widthwise crossings of the bottom faces of the ``rippling'' quads $[0,3R] \times F_\alpha([0,R]\times\{0\})$ and $[0,3R] \times F_{\alpha'}([0,R]\times\{0\})$, that both start from $x$. Let $\gamma$ and $\gamma'$ be two such crossings.
\item By Item 2 of Lemma \ref{lem:rippling_quads}, the following (or its analogue obtained by interchanging $(\alpha,\gamma)$ and $(\alpha',\gamma')$) holds: there is a path included in $\gamma \cap ([0,3R] \times F_{\alpha'}([0,R] \times [0,R^a]))$ that intersects $[0,3R] \times F_{\alpha'}([0,R] \times \{R^a\})$. This contradicts ii) in the definition of $A^{\alpha'}(x,R)$.
\item We conclude that the $(A^\alpha_\ell(x,R):\alpha\in A)$ are indeed disjoint.
\end{itemize}
By first using this with $x=0$, we obtain that the first sum is bounded by one. Let us deal with the second sum. Since $B_\ell^\alpha(x,R) \subset A^\alpha_\ell(x,R)$, the $B^\alpha_\ell(x,R)$ are also disjoint and $\sum_{\alpha\in A} 1_{B^\alpha_\ell(x,R)}\leq 1$ so the second sum is at most
\[
\int_0^{3R} \E \left[ |\partial_{x_1} f_{r}(x)| \mid f_{r}(x)=\ell \right] \gamma_{f_{r}(x)}(\ell) dx=3R\E[|\partial_{x_1} f_{r}(0)|]\gamma_{f_r(0)}(\ell)
\]
by stationarity. Since the right-hand side is $O(R)$, all in all, there is some constant $C>0$ such that
\[
|A|\min_{\alpha \in A} \Pro [ f_{r} \in E^\alpha_\ell(R) ] \leq 1 + C R.
\]
(The reader can note that we have used that $r \ge r_q$ in order to obtain that $C$ does not depend on $r$.) This implies the desired result. 

\subsection{Proof of Claim~\ref{cl:co-area}}\label{ssec:claim_co-area}

In this section, we prove Claim~\ref{cl:co-area}. Since in Bernoulli percolation, which is our guiding model, Claim~\ref{cl:co-area} would simply be replaced by a decomposition of a crossing event into arm events, from the point of view of percolation, this section is just a technical obstacle and should be skipped upon first reading. However, we believe that the proof is interesting from the point of view of the geometry of Gaussian fields. It is inspired by the proof of the Kac--Rice formula given in Appendix C of \cite{Letendre}. With this tool in mind, the formula is easily established for fields taking values in finite dimensional function spaces. We must then take the formula to the limit.

\begin{proof}[Proof of Claim~\ref{cl:co-area}]
We prove the claim in a more general setting. Moreover, for brevity, we prove it for $E_\ell(R)$ instead of $E^\alpha_\ell(R)$. The proof in the general case is similar although one must consider curved rectangles instead of flat ones. Throughout the proof, $R>0$ will stay fixed and we consider a compact subset $\calK\subset\R^d$  with smooth boundary that is a neighbourhood of $[0,3R]\times Q=[0,3R] \times [0,R] \times [0,R^a]$. Let $g$ be an a.s.\ $C^3$ centered Gaussian field on $\calK$ with covariance $(x,y) \mapsto K(x,y)=\E[g(x)g(y)]$ in $C^{3,3}(\calK,\calK)$ (i.e., $\partial^\alpha_x\partial^\beta_yK\in C(\calK\times\calK)$ for $|\alpha|,|\beta|\leq 3$) such that
\begin{equation}\label{eq:coarea_1}
\text{$\forall x,y \in\calK$ with $x\ne y$, $(g(x),g(y),d_yg)$ is a non-degenerate Gaussian vector.}
\end{equation}
Our goal is to prove that
\begin{equation}\label{e:arm_decomposition}
\Pro [ g \in E_\ell(R) \setminus A_\ell(0,R) ] = \int_0^{3R} \E \left[ |\partial_{x_1} g(x)| 1_{B_\ell(x,R)} \mid g(x)= \ell \right] \gamma_{g(x)}(\ell) dx_1.
\end{equation}

Before proving \eqref{e:arm_decomposition}, let us use it in order to conclude the proof of the claim.

\begin{proof}[Proof of Claim \ref{cl:co-area}]
In order to prove the claim, we only need to show that $f_r$ satisfies \eqref{eq:coarea_1}. If this were not the case, then there would exist some $x \neq y$ and some real numbers $a,b,c_1,\dots,c_d$ not all equal to $0$ such that $af_r(x)+bf_r(y)+\sum_{i=1}^d c_i \partial_{x_i}f_r(y)=0$. As a result, there would exist some vector $v \in \R^d \setminus \{0\}$ and three real numbers $a,b,c$ not all equal to $0$ such that
\begin{equation}\label{eq:degen_absurdie}
af_r(x)+bf_r(y)+c\partial_vf_r(y)=0.
\end{equation}
To conclude, we need the following lemma.
\begin{lemma}[Lemma A.2 of \cite{bmm}]\label{lem:bmm}
Let $f$ be a centered, stationary and $C^3$ Gaussian field on $\R^2$. If the support of the spectral measure\footnote{The spectral measure is the Fourier transform of the covariance kernel $\kappa : x \mapsto\E[f(y+x)f(y)]$.} of $f$ contains an open set, then for every distinct $x,y \in \R^2$, $(f(x),f(y),d_xf,d_yf)$ is non-degenerate.
\end{lemma}
Let $\calP \subset \R^d$ denote a plane such that $x,y,x+v \in \calP$. Since the covariance kernel of $(f_r)_{|\calP}$ has compact support and is not identically equal to $0$, we obtain that its spectral measure is continuous and not identically equal to $0$, so Lemma~\ref{lem:bmm} implies that \eqref{eq:degen_absurdie} cannot be true.
\end{proof}

Let us now prove \eqref{e:arm_decomposition}. The proof will follow from the co-area formula applied to a well chosen functional defined and differentiable on the complement of some exceptional subset $W\subset C^2(\calK)$. Before moving forward, let us define $W$ and prove a few facts about this set (and \eqref{eq:coarea_1}). Let $\calF$ be the set whose elements are the corners, open sides, open 2-dimensional faces and the interior of $[0,3R]\times Q$. We define $W$ as the set of $u\in C^2(\calK)$ for which there exists $F\in\calF$ for which $u_{|F}$ has a critical point at height $\ell$.

\begin{claim}\label{cl:coarea_1}
\begin{enumerate}[noitemsep]
\item The set of covariance functions $K\in C^{3,3}(\calK,\calK)$ such that the a.s.\ $C^2$ field $g$ with covariance $K$ satisfies $\eqref{eq:coarea_1}$ is open in the $C^{3,3}$ topology;
\item For every $x \in (0,3R)$, $\Pro [g\in W] = \Pro [ g \in W | g(x) = \ell] = 0$;
\item The boundary of $U:=(E_\ell(R)\setminus A_\ell(0,R)) \setminus W$ in $C^2(\calK)$ is included in $W$;
\item $U$ is open in $C^2(\calK)$;
\item For every $x \in (0,3R)$, the boundary of $B_\ell(x,R) \setminus W$ in $\{ u \in C^2(\calK) : u(x)=\ell\}$ is included in $W$.
\end{enumerate}
\end{claim}
\begin{proof}
We prove the points one by one.

For Item 1, proceed as follows. The negation of property \eqref{eq:coarea_1} at $x,y\in\calK$ is equivalent to a polynomial equation in $K(\cdot,\cdot)$ and its derivatives of order up to one in each variable, at $(x,x)$ and $(x,y)$. Since $\calK$ is compact, this condition is open in $C^{3,3}(\calK,\calK)$.

We turn to Item 2. Since $y\mapsto (g(y),d_yg)$ is a.s.\ $C^1$ and we have assumed that for each $y\in\calK$, $(g(y),d_yg)$ is non-degenerate, then, by \cite[Lemma 11.2.10]{adler_taylor}, $g\notin W$ a.s. Moreover, if we condition on the event $\{ g(x) = \ell \}$ then $y \mapsto (g(y),d_yg)$ still satisfies this property on $\calK \setminus \{ x \}$ (indeed, if a Gaussian vector $(X_1,\dots,X_{k+1})$ is non-degenerate, then the law of $(X_1,\dots,X_k)$ is still non-degenerate when we condition on the value of $X_{k+1}$)\footnote{To see this one can use that the variance of $(X_1,\dots,X_k)$ conditioned on the value of $X_{k+1}$ does not depend on this value and that its mean is linear with respect to this value.}. As a result, still by \cite[Lemma 11.2.10]{adler_taylor}, we also have $g \notin W$ a.s.\ under this conditioning.

In order to prove Item 3, let us consider some $u \notin W$ and let us show that there exists $\varepsilon>0$ (that may depend on $u$) such that, if $\|v\|_{C^2} \le \varepsilon$, then $1_{ E_\ell(R)}(u)=1_{ E_\ell(R)}(u+v)$, and similarly for $A_\ell(0,R)$ instead of $E_\ell(R)$. Actually, we will see that a bound on $\|v\|_{C^0}$ will be sufficient. To this purpose, we first note that there exists $\varepsilon>0$ such that, for every $F\in\calF$, $u_{|F}$ has no critical point at any height in $[\ell-\varepsilon,\ell+\varepsilon]$. Let us fix such an $\varepsilon$ and let us apply a result from stratified Morse theory: By \cite[Proposition in Section 3.2 of Part I]{morse}, there exist homeomorphisms $\varphi_\pm$ from $\{ u \ge \ell \} \cap ([0,3R]\times Q)$ to $\{ u \ge \ell\pm\varepsilon \} \cap ([0,3R]\times Q)$ such that, for every $F \in \mathcal{F}$, both $\varphi_\pm$ and $\varphi^{-1}_\pm$ send a point of $F$ to a point of $F$. As a result, 
\begin{equation}\label{eq:E(u+v)}
\text{$1_{ E_\ell(R)}(u)=1_{ E_\ell(R)}(u\pm\varepsilon)$, \quad and similarly for $A_\ell(0,R)$ instead of $E_\ell(R)$.}
\end{equation}
Moreover, by the same arguments as in Claim \ref{cl:morse} (where we use the notations $T_R = [0,3R] \times \{0\}$ and $B_R=[0,3R]\times\{0\}$), all connected components of $\{ u \ge \ell-\varepsilon \} \cap ([0,3R] \times [0,R])$ that contain a path from $B_R$ to $T_R$ also contain such a path $\gamma$ with the further property that $u_{|\gamma} \ge \ell+\varepsilon$, and the same property holds with $B_R$ replaced by $\{0\}$. Let us now consider some $v$ such that $\|v\|_{C^0} \le \varepsilon$ and let us conclude the proof by considering the two following cases.
\begin{itemize}[noitemsep]
\item $u \notin E_\ell(R)$ (resp.\ $A_\ell(0,R)$). In this case, let $C$ be a connected component of $\{ u+v \ge \ell \} \cap ([0,3R] \times [0,R])$ that contains a path from $B_R$ (resp.\ $\{0\}$) to $T_R$. By the above, $C$ also contains a path $\gamma$ from $B_R$ (resp.\ $\{0\}$) to $T_R$ with the further hypothesis that $u_{|\gamma} \ge \ell+\varepsilon$. Since $u$ does not belong to $E_\ell(R)$ (resp.\ $A_\ell(0,R)$), then (by~\eqref{eq:E(u+v)}) we obtain that $\gamma$ is connected to $[0,3R]\times[0,R]\times\{R^a\}$ by a continuous path in $\{u\geq \ell+\varepsilon\}\cap([0,3R]\times[0,R]\times[0,R^a])$. As a result, $C$ is connected to $[0,3R]\times[0,R]\times\{R^a\}$ by a continuous path in $\{u+v\geq \ell\}\cap([0,3R]\times[0,R]\times[0,R^a])$, and we deduce that $u+v \notin E_\ell(R)$ (resp.\ $A_\ell(0,R)$).
\item $u \in E_\ell(R)$ (resp.\ $A_\ell(0,R)$). By \eqref{eq:E(u+v)}, there exists a connected component $C$ of $\{ u \ge \ell-\varepsilon \} \cap ([0,3R] \times [0,R])$ that contains a path from $B_R$ (resp.\ $\{0\}$) to $T_R$ but that is not connected to $[0,3R] \times [0,R] \times \{R^a\}$ by a path in $\{u \ge \ell - \varepsilon \} \cap ([0,3R] \times [0,R] \times [0,R^a])$. We recall that $C$ also contains a path $\gamma$ from  $B_R$ (resp.\ $\{0\}$) to $T_R$ with the further hypothesis that $u_{|\gamma}\ge \ell+\varepsilon$ and we note that $\gamma \subset \{u+v\ge \ell\}$ and that $\gamma$ is not connected to $[0,3R] \times [0,R] \times \{R^a\}$ by a path in $\{u +v \ge \ell \} \cap ([0,3R] \times [0,R] \times [0,R^a])$, and we deduce that $u+v \in E_\ell(R)$ (resp.\ $A_\ell(0,R)$).
\end{itemize}

We now show Item 4. By Item 3, if $u \in U$ then there is a $C^2(\calK)$-neighbourhood of $u$ that is included in $E_\ell(R) \setminus A_\ell(0,R)$. Since $W$ is closed in $C^2(\calK)$ (see e.g.\ the short proof of Lemma C.1 of \cite{BMR}), this implies the desired result.

The proof of Item 5 is essentially the same as the proof of Item 3. The only difference is that one needs to pay attention to what happens close to $x$. One possibility is to declare that two points of $(0,3R)$ that are very close to $x$ are corners of the stratified set and apply Morse results to this new stratified set. \qedhere
\end{proof}

\textbf{A. Proof in the finite-dimensional case.} We are now ready to prove \eqref{e:arm_decomposition}. We start by assuming that $g$ belongs to a finite-dimensional space $V\subset C^2(\calK)$. Thus, there exists a scalar product $\langle\cdot,\cdot\rangle$ on $V$ for which $g$ has the law of the standard Gaussian distribution $\gamma_V(u)du$ on $(V,\langle\cdot,\cdot\rangle)$.

\medskip

For every $u \in E_\ell(R)$, we let $X(u) \in [0,3R] = [0,3R] \times \{0\}^{d-1}$ be the (unique) point such that $u \in B_\ell(X(u),R)$. Recall that we use the notation $U =(E_\ell(R)\setminus A_\ell(0,R)) \setminus W$ and that, by Item 4 of Claim \ref{cl:coarea_1}, $U\cap V$ is open in $V$. The map
\begin{equation}\label{eq:coarea_2}
\Phi:U\cap V\rightarrow (0,3R)
\end{equation}
sending $u$ to the first coordinate of $X(u)$ is well defined and differentiable with differential
\[
d_u\Phi(v)=\frac{-v(X(u))}{(\partial_{x_1}u)(X(u))}.
\]
Thus, the norm of the differential at $u$ is $|\partial_{x_1}u(X(u))|^{-1}$ times the norm of the evaluation map at $X(u)$. By construction of $V$, the norm of the evaluation map is\footnote{Indeed, let $(\Psi_1,\dots,\Psi_{k-1})$ be an orthonormal basis of the kernel of the non-trivial linear form $L\in V^*$ defined by $L(v)=v(x)$ for some $x \in \calK$ and let $\Psi_k\in\textup{Ker}(L)^\perp$ be a unit vector. Then, $$\|L\|=|L(\Psi_k)|=\sqrt{\sum_{j=1}^kL(\Psi_j)^2}=\sqrt{\sum_{j=1}^k \Psi_j(x)^2}=\sqrt{K(x,x)}.$$ In the last equality, we have used that $K(x,y) = \sum_{j=1}^k \Psi_j(x)\Psi_j(y)$ for any $x,y \in \calK$. (A more general form of this result follows from (A.8) and (C.6) from \cite{Letendre}.)\label{foot}}$\sqrt{K(X(u),X(u))}$. Then, the co-area formula (see e.g.\ \cite[Corollary~13.4.6]{BZ}) applied to $\Phi$ yields
\begin{align*}
\Pro [ g \in E_\ell(R) \setminus A_\ell(0,R) ] = \int_U \gamma_V(u) du & = \int_0^{3R} \int_{U_x} \frac{\gamma_V(u)}{\|d_u\Phi\|} \overline{du} dx_1\\
&=\int_0^{3R} \int_{U_x}\frac{|\partial_{x_1}u(x)|\gamma_V(u)}{\sqrt{K(x,x)}} \overline{du} dx_1,
\end{align*}
where the first equality comes from Item 2 of Claim \ref{cl:coarea_1} and
\begin{itemize}[noitemsep]
\item $\gamma_V : U \rightarrow \R_+$ is the density function of $g$;
\item $x = (x_1,0,\dots,0)$ and $U_x:=\Phi^{-1}(x_1)=B_\ell(x,R) \cap (V \setminus W)\subset V_x:=\{u\in V : u(x)= \ell\}$;
\item $\overline{du}$ is the Lebesgue density on $V_x$.
\end{itemize}
As a result, we have
\[
\Pro [ g \in E_\ell(R) \setminus A_\ell(0,R) ] 
=\int_0^{3R} \E [ |\partial_{x_1}g(x)| 1_{g \in B_\ell(x,R)} \mid g(x)= \ell ] \frac{\int_{V_x} \gamma_V(u) \overline{du}}{\sqrt{ K(x,x)}}dx_1.
\]
Next, we observe that $\int_{V_x} \gamma_V(u) \overline{du}$ is the density of $\calN(0,1)$ at $ \ell/\sqrt{K(x,x)}$ (because, by the observations from Footnote \ref{foot}, $V_x$ is a hyperplane of $V$ translated by a vector in $V^\perp_x$ of norm $\ell/\sqrt{K(x,x)}$). This implies that $\int_{V_x} \gamma_V(u) \overline{du}=\sqrt{ K(x,x)}\gamma_{g(x)}(\ell)$ and ends the proof of the claim in the finite dimensional case.

\medskip

\textbf{B. Proof in the infinite-dimensional case.} To cover the general (infinite dimensional) case, we start by arguing that \eqref{e:arm_decomposition} is stable by approximation. To this end, we first show the following claim:

Let $(g_n)_{n\in\N}$ be a sequence of $C^3$ and centered Gaussian fields that converges a.s.\ in $C^3(\calK)$ to $g$ and satisfies \eqref{e:arm_decomposition} for each $n$. Then, $g$ also satisfies \eqref{e:arm_decomposition}.

To show this claim, we use the following three properties:
\begin{itemize}[noitemsep]
\item[i)] The covariance kernels $K_n \in C^{3,3}(\calK,\calK)$ converge in $C^{3,3}(\calK,\calK)$ to $K$;
\item[ii)] The conditioning on $g(x)=\ell$ makes $g$ a Gaussian field with covariance $(y,z) \mapsto K(y,z)-K(y,x)K(x,z)$ and mean $y \mapsto K(x,y)\ell/K(x,x)$;
\item[iii)] For every $x \in \calK$, the law of $g_n$ conditioned on $g_n(x)=\ell$ converges in distribution on the Borel-$\sigma$-algebra of the $C^2$-topology to the law of $g$ conditioned on $g(x)=\ell$.
\end{itemize}
(To prove iii), one can use i), ii) and that the sequence is tight due to Kolmogorov's theorem -- see e.g.\ Appendix A.9 of \cite{NS}.) The points i), ii) and iii) above, Claim \ref{cl:coarea_1} and dominated convergence imply the claim.

\medskip

So let us exhibit such a sequence $(g_n)_n$. (We follow the analogous arguments in the proof of \cite[Lemma 3.1]{MRVK}.) Let us first assume furthermore that $g$ is a.s.\ $C^\infty$ and let $H^{N}(\calK)$ be the $L^2$-Sobolev space of order $N$ on $\calK$ where $N$ is any positive integer such that $H^{N}(\calK) \subset C^3(\calK)$ and the injection is continuous. We observe that $g \in H^{N}(\calK)$ since $g \in C^\infty(\calK)$. Let $(e_k)_k$ be an orthonormal basis for $H^{N}(\calK)$ and let $g_n$ denote the projection of $g$ on the subspace generated by $e_1,\dots,e_n$. Then, $g_n$ converges to $g$ a.s.\ in $H^{N}(\calK)$, and so converges to $g$ a.s.\ in $C^3(\calK)$. This ends the proof in the case where $g$ is $C^\infty$.

\medskip

To finish the argument, we remove the $C^\infty$-smoothness assumption. To this purpose, we use that $g$ is defined on a neighborhood $\calK$ of $[0,3R] \times Q$ rather than on $[0,3R] \times Q$ only (on the contrary, all the above arguments also work with $\calK$ replaced by $[0,3R] \times Q$). Let $\varepsilon>0$ denote the distance between $[0,3R] \times Q$ and $\calK$ and define $g_n$ on $[0,3R] \times Q$ as the convolution between $g$ and a Dirac approximation which is compactly supported in $B(\varepsilon/2)$. Then, $g_n$ is smooth and converges to $g$ a.s.\ in $C^3$, which is the desired property except that $\calK$ is replaced by $[0,3R] \times Q$. But, as we have already observed, all what has been done above also works with $\calK$ replaced by $[0,3R] \times Q$, so this ends the proof.
\end{proof}

\section{A box-crossing property with polynomially many contact points}
\label{sec4}

Let $ d\ge 2$ and let $q$ satisfy Assumption \ref{ass1} for some $\beta>d$. In this section, we only consider $(f_r)_{|\R^2}$ and prove Proposition \ref{prop:polynom_absctract} below. All the constants in this section are universal since they are functions of the RSW constants from Theorem \ref{thm:rsw}. We will often use the gluing constructions described in Section \ref{ss:rsw} in conjunction with the FKG inequality and the RSW theorem without mentioning them explicitely. Recall that we identify all subsets $D \subset \R^2$ with $D\times\{0\}^{d-2} \subset \R^d$ and that $D(R)$ and $D(x,R)$ are Euclidean discs (see Section \ref{ssec:not}).

\begin{proposition}\label{prop:polynom_absctract}\label{prop:poly_many_bis}
There exists some universal $\eta>0$ such that the following holds. For every $r_q \le r \le \rho \le R$ and every continuous path $\mathcal{C}\subset\R^2$ that intersects both $D(R)$ and $\partial D(2R)$, there exists a (deterministic) finite family of points $( y_i)_{i\in I} \subset D(2R)$ at mutual distances at least $30\rho$ and at a distance at most $\rho$ from $\mathcal{C}$ such that
\[
\mathbb P\Bigg[\begin{array}{c} \#\{i\in I:D(y_i,\rho/100)\text{ is connected to $\partial D(3R)$}\\
\text{ in $(D(3R)\cap\{f_r\ge 0\}) \setminus (\cup_{j \ne i} D(y_j,2\rho))$}\} \geq \eta(R/\rho)^\eta \end{array}\Bigg]\ge \eta.
\]
\end{proposition}

Results of a similar flavor have appeared e.g.~in the context of planar random cluster models, see Proposition 7.4 in \cite{duminilcopin2021planar}, cf.~also the proof of Proposition 1.5 therein.

\medskip

In order to prove Proposition~\ref{prop:polynom_absctract}, we introduce the following notion of $R$-good pair.

\begin{definition}[Good pairs] \label{def:goodpair}
Let $x,z\in\R^2$, $R>0$, and $\mathcal C \subset \R^2$ be a continuous path. Then $(x,z)$ is called an $R$-{\em  good  pair} for $\mathcal C$ if it has the following properties:
\begin{enumerate}[noitemsep]
\item $\mathcal C$ intersects both $D(x,2R)$ and $\partial D(x,10R)$;
\item $z \in \partial D(x,10R)$ and is at a distance at least $11R$ from $\mathcal C$.
\end{enumerate}
\end{definition}

Observe for later purposes that if $\mathcal{C}$ is as in the statement of Proposition~\ref{prop:polynom_absctract} and if $R' \le R/100$, then one can always find an $R'$-good  pair $(x,z)$ for $\mathcal C \cap D(1.9R)$ with $x \in \partial D(1.9R+2R')$ and $z \in \partial D(1.9R+12R')$.

\medskip

We start with a first lemma.

\begin{lemma}[First moment estimate]\label{l:first_moment} There exists a universal constant $\eta_0>0$ such that, for all $K\geq 100$, $r_q \le r \le R$, every continuous path $\mathcal{C}$ and every $(KR)$-good pair $(x,z)$ for $\mathcal C$, there exists a finite collection $(y_i,z_i)_{i\in I}$ of $R$-good pairs for $\mathcal C$ such that 
\begin{itemize}[noitemsep]
\item For every $i\in I$, $y_i\in D(x,5KR)$;
\item The $y_i$'s are at mutual distances at least $30R$;
\item Let $H_i$ be the event that there exists a circuit in $\{f_r\ge0\}\cap D(z_i,3R)$ around $D(z_i,2R)$ that is connected to $D(z,2KR)$ in $D(x,9KR)$ by a path in $\{f_r\ge0\}$ staying at a distance at least $12R$ from $\mathcal C$ and at a distance at least $11R$ from the $y_j$'s, $j\neq i$. Then 
\[
\E\left[\#\{i\in I:H_i\text{ occurs}\}\right]\geq \eta_0K^{\eta_0}.
\]
\end{itemize}
\end{lemma}

\begin{remark}\label{rem:last_step***}
We first note, for future reference, that the following analogous (but much more direct) result holds by the RSW theorem. For all $R>0$, $r \ge r_q$, every continuous path $\calC$ and every $R$-good pair $(x,z)$ for $\calC$, there exists a point $y \in D(x,5R)$ at distance at most $R$ from $\calC$ such that the following holds. Let $\widetilde{H}$ denote the event that there exists a path in $\{ f_r \ge 0\} \cap D(x,9R)$ from $D(y,R/100)$ to $D(z,2R)$. Then, there exists a universal constant $\eta_1>0$ such that $\Pro[\widetilde{H}] \ge \eta_1$.
\end{remark}

\begin{figure}[h!]
\centering
\includegraphics[scale=0.36]{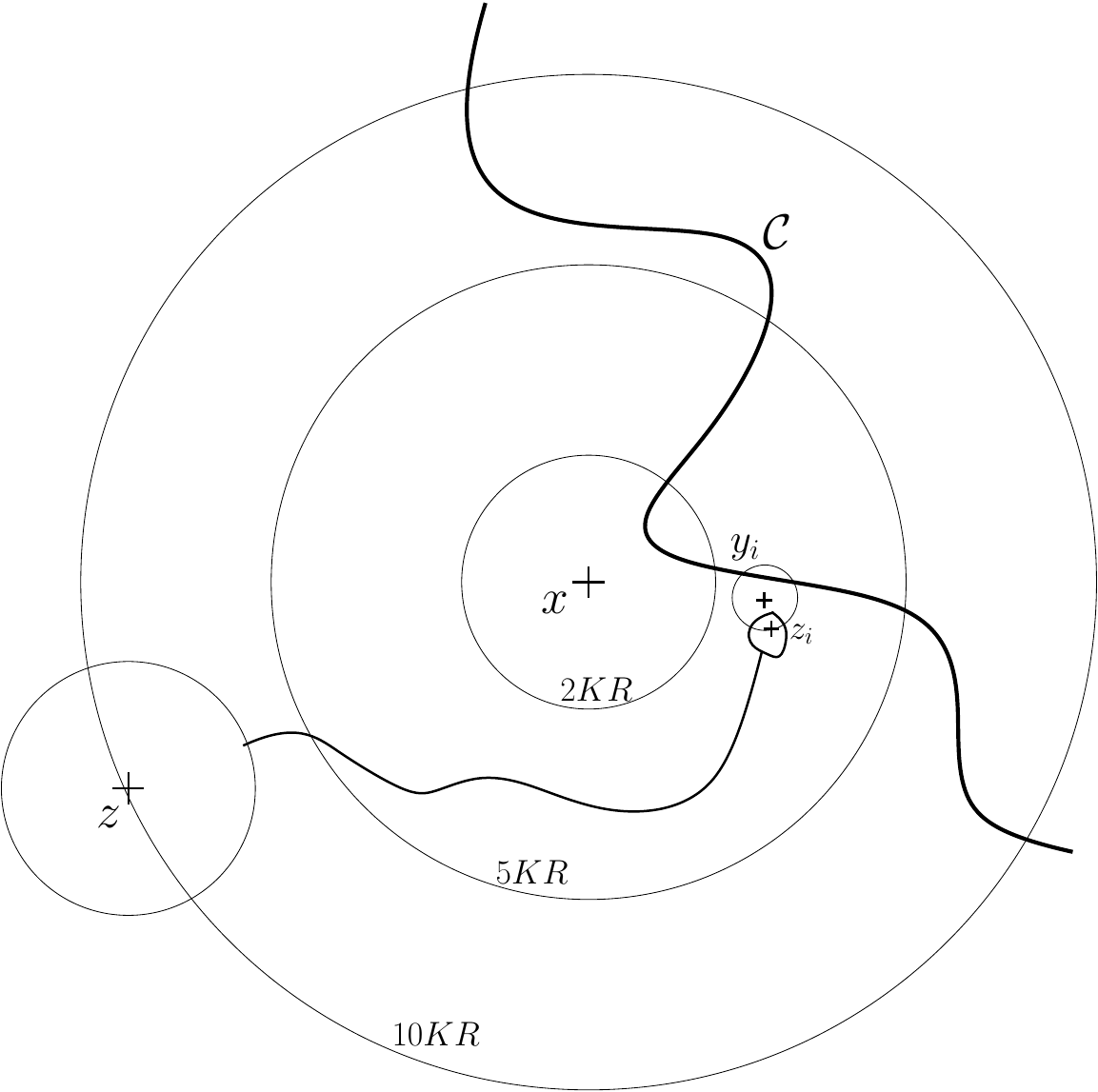}
\includegraphics[scale=0.36]{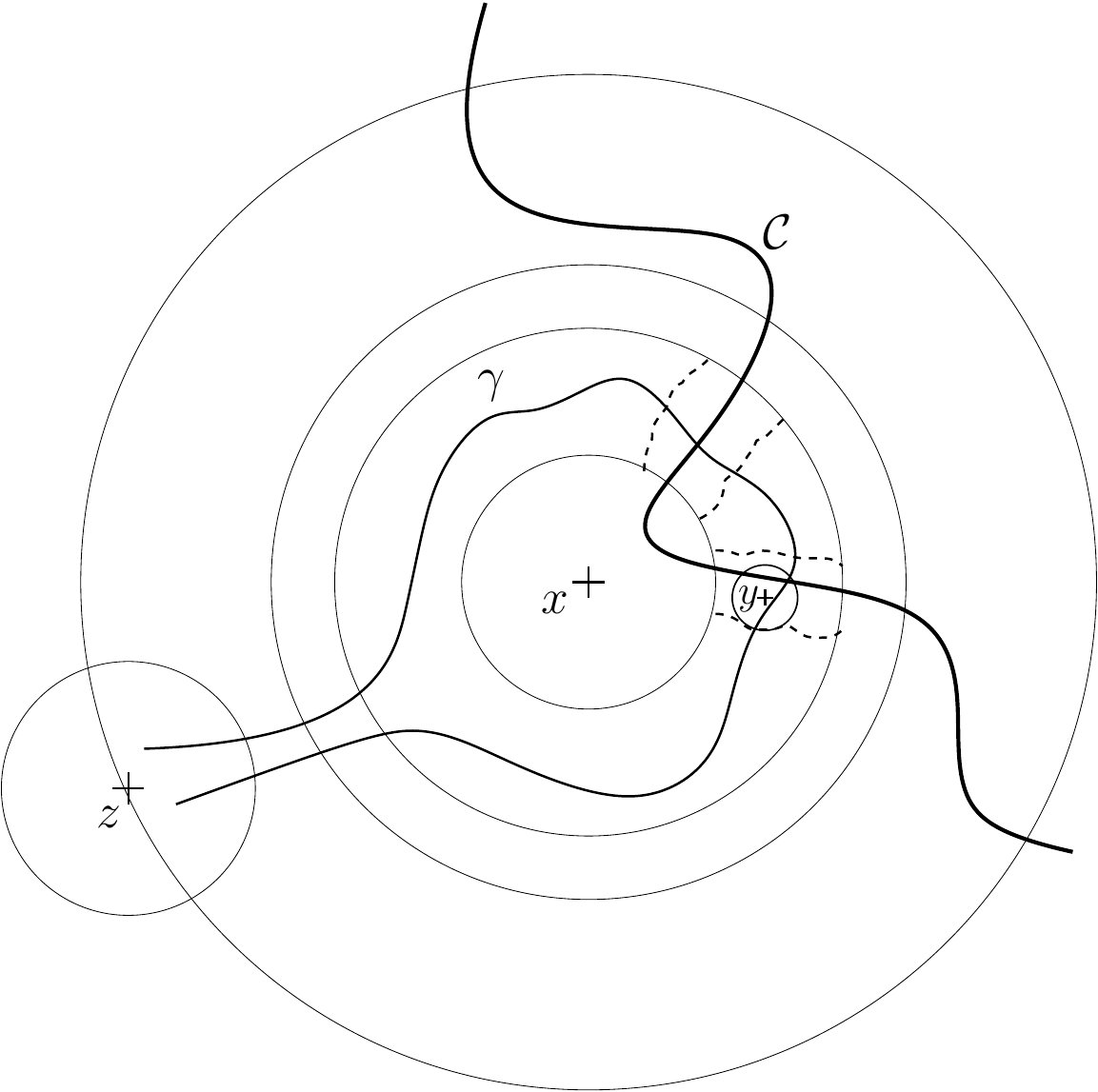}
\caption{(a) The event $H_i$. (b) The path $\gamma$ and (in dashed lines) the set $\Delta$.}\label{fig:tree_expect}
\end{figure}

\begin{proof}[Proof of Lemma \ref{l:first_moment}]
Let $\Delta'$ be the boundary in $D(x,5KR)$ of the union $\Gamma$ of the balls $D(y,11R)$ where $y$ ranges over the set of points in $ D(x,5KR)\cap \tfrac1{100}R\Z^2$ at a distance at most $2R$ from $\mathcal C$. To avoid having to deal with what happens close to $\partial D(x,5KR)$, we also consider $\Delta := \Delta' \cap D(x,4KR)$. Furthermore, let $Y$ be the set of such $y$ for which $\partial D(y,11R)\cap\Delta\neq\emptyset$. 

For each $y\in Y$, let $G(y)$ (resp. $G^\star(y)$) denote the event that $D(y,11R)$ is connected to $D(z,2KR)$ in $D(x,9KR)$ by a path in $\{f_r\ge0\}$ (resp.~$\{f_r\le 0$\}) staying at a distance at least $12R$ from $\mathcal C$ and not intersecting the interior of $\Gamma$.

Suitable applications of the RSW property Theorem~\ref{thm:rsw} (together with Lemma \ref{lem:mani*} and the fact that $f_r$ is $r$-dependent), imply the existence of $c_1>0$ such that with probability at least $c_1$, there exists a path $\gamma$ in $\{f_r=0\}\cap D(x,9KR)$ with endpoints in $D(z,2KR)$ such that $D(z,2KR)\cup \gamma$ contains a circuit in the annulus $D(x,10KR)\setminus D(x,2KR)$ separating its two boundary components. But $\mathcal C$ crosses this annulus and does not intersect $D(z,2KR)$ so it must intersect $\gamma$ (see Figure~\ref{fig:tree_expect}). Considering the first time $\gamma$ intersects $\cup_{y\in Y} D(y,11R)$ gives
\[
\sum_{y\in Y}\Pro [G(y)\cap G^\star(y)]\geq c_1.
\]
By the FKG inequality (see Lemma~\ref{lem:FKG1*}),
\[
\Pro [G(y)\cap G^\star(y)]\leq \Pro [G(y)] \Pro [G^\star(y)]
\]
and by spatial independence (using that $R \geq r$) and the RSW property, as mentioned at the end of Section \ref{ss:rsw} one may infer that $\Pro[G^\star(y)] \leq c_2K^{-c_3},$ so we deduce overall that 
\[
\sum_{y\in Y}\Pro [G(y)]\geq c_4 K^{c_3},
\]
for some $c_2,c_3,c_4>0$. We then extract a family $(y_i)_{i\in I} \subset Y$ of points at a mutual distance at least $30R$ such~that
\[
\sum_{i\in I}\Pro[G(y_i)]\geq c_5 K^{c_3}
\]
for some $c_5>0$. Applying the RSW property again near the starting point of the path described by $G(y_i)$ implies that there exists $c_6>0$ and a family $(z_i)_{i\in I}$ of points such that for each $i\in I$, $z_i$ belongs to $\partial D(y_i,10R)$ and is at a distance at least $11R$ from $\mathcal{C}$ and 
\[\Pro [H_i] \geq c_6 \Pro[G(y_i)],\]
 where $H_i$ is as in the statement of the lemma. The result follows readily.
 \end{proof}

We are now in a position to prove the proposition. We will prove Proposition \ref{prop:poly_many_bis} by constructing a tree whose nodes are the good pairs obtained thanks to Lemma \ref{l:first_moment} and whose edges are open if the corresponding events $H_i$ occur. Then, the existence of many open paths from the root to the leaves will imply the event under consideration in the proposition.

\begin{proof}[Proof of Proposition \ref{prop:poly_many_bis}]
Consider  $\eta_0$ given by Lemma \ref{l:first_moment} and $\eta_1$ given by Remark~\ref{rem:last_step***} and fix $K\geq 100$ such that $\eta_0 K^{\eta_0}>1$. For each integer $j\geq0$, let 
\begin{equation}\label{eq:treescales}
R_j=K^j \rho.
\end{equation}
(So in particular we have $R_0=\rho$.) Define $J$ to be the supremum over $j \geq 0$ such that $R_j \leq R/100$. If $J =  0$ then the result is a direct consequence of the RSW theorem so we assume that $J \ge 1$.

\medskip

We introduce a tree $\mathcal{T}$ of depth $J$ as follows. All the vertices will be pairs $(x,z) \in \R^2$ except leaves which will be points $y \in \R^2$.

\begin{itemize}
\item The root of $\mathcal T$ is a $R_J$-good pair  $(x_{\mathrm{root}},z_{\mathrm{root}})$ for $\mathcal{C} \cap D(1.9R)$ with $x_{\mathrm{root}} \in \partial D(1.9R+2R_J)$ and $z_{\mathrm{root}} \in \partial D(1.9R+12R_J)$ (cf.~the observation following Definition~\ref{def:goodpair} regarding its existence).
\item The vertices of $\mathcal{T}\setminus \{ (x_{\mathrm{root}},z_{\mathrm{root}}) \}$ which are not leaves are defined recursively as follows. For some $0\le j<J$, suppose that we consider a vertex $(x,z) \in \mathcal{T}$ which is a $R_{J-j}$-good pair for $\mathcal{C} \cap D(1.9R)$. Then, we add to $\mathcal{T}$ the vertices $(y_i,z_i)_{i\in I}$ given by Lemma~\ref{l:first_moment} when applied with $(x,z)$, $K$ and $R=R_{J-j}$, and join each of them to $(x,z)$ by an edge. We call the $(y_i,z_i)$'s the {\em descendants} of $(x,z)$. 
Finally, we say that the edge between $(x,z)$ and $(y_i,z_i)$ (where $(y_i,z_i)$ is a descendant of $(x,z)$) is {\em open} if $H_i$ occurs, with $H_i$ as appearing in the statement of Lemma~\ref{l:first_moment}.
\item The leaves are defined as follows. Let us consider a vertex $(x,z) \in \calT$ which is a $\rho$-good pair for $\calC \cap D(1.9R)$ (note that $\rho=R_j$ with $j=0$). Then, we add to $\calT$ the vertex $y$ given by Remark \ref{rem:last_step***} when applied to $(x,z)$, $R=\rho$, and join $y$ to $(x,z)$ by an edge. We say that this edge is \textit{open} if the event $\widetilde{H}$ appearing in Remark \ref{rem:last_step***} occurs.
\end{itemize}

For every $0\le j<J$, let $\mathbf V_j$ be the set of vertices in $\mathcal T$ that are $R_j$-good pairs and are connected by an open path to $(x_{\mathrm{root}},z_{\mathrm{root}})$. Also, let $\mathbf V_J=\{(x_{\mathrm{root}},z_{\mathrm{root}})\}$ and let $\mathbf V_{-1}$ denote the set of leaves which are connected by an open path to $(x_{\mathrm{root}},z_{\mathrm{root}})$.

\medskip

Due to the nested structure of the events $H_i$ and the notion of open edge in $\mathcal T$, a path of open edges between the root and a leaf $y$ 
guarantees the existence of a path from $D(y,\rho)$ to $D(z_{\rm root},2R_J)$ in $\{f_r\ge0\}\cap D(x_{\rm root},10R_J)$. Therefore, upon setting
\[
\mathbf N =\#\Big\{i\in I:D(y_i,\rho) \leftrightarrow \partial D(3 R) \text{ in $\{f_r\ge0\}\cap\big(D(3R)\setminus \cup_{j\ne i}D(y_j,2\rho) \big)$}\Big\},
\]
where $\{ y_i: i \in I\}$ refers to the set of leaves of $\mathcal{T}$, the RSW theorem and the FKG inequality imply that for every $\eta>0$,
\[
\mathbb P[\mathbf N\ge \eta(R/\rho)^\eta
]\ge c_0\mathbb P[|\mathbf V_{-1}|\ge \eta(R/\rho)^\eta],
\]
for some constant $c_0>0$.

\medskip
 
The previous observation implies that it suffices to show the existence of some $\eta>0$ such that
\begin{equation}\label{eq:treegoal}
\Pro [ |\mathbf V_{-1}| \geq \eta(R/\rho)^\eta ] \geq \eta.
\end{equation}
This claim is a direct consequence of repeated applications of the inequality
\begin{equation}\label{eq:h1}
\Pro\left[|\mathbf V_{j-1}|\ge \kappa |\mathbf V_j| \, \big| \, \mathcal{F}_j\right]\geq 1-e^{-c|\mathbf V_j|},
\end{equation}
for all $0 \le j \le J+1$ and some $c>0$ and where
\begin{itemize}[noitemsep]
\item $\mathcal{F}_j$ is the $\sigma$-algebra generated by the events $H_i$ for the $(y_i,z_i)$'s that are $R_i$-good pairs with $i\ge j$ (here and below, we use the convention that $\calF_{J+1}$ is the trivial $\sigma$-algebra);
\item $\kappa := ( 1+ \eta_0 K^{\eta_0})/2$ if $1\le j\le J$ and $\kappa := \eta_1/2$ if $j=0$. Here $\eta_0>0$ is as in Lemma \ref{l:first_moment} and $\eta_1$ is as in Remark \ref{rem:last_step***}.
\end{itemize}

We therefore focus on proving \eqref{eq:h1}. We write the proof for $1\le j\le J$ since the proof in the case $j=0$ is essentially the same (but easier). For each $(x,z)$ in the tree at a graph distance at least $2$ from the leaves, let $\mathbf N_{(x,z)}$ be the number of descendants $(x',z')$ of $(x,z)$ such that the edge of $\mathcal T$ between the two vertices is open.
We claim that, for any $A \in \calF_{j+1}$, any possible outcome $V$ for $\mathbf V_j$, and any $M>0$, the event
\begin{equation}\label{eq:sec4fkglocal***}
\Big\{\sum_{(x,z)\in V} \mathbf N_{(x,z)} > M \Big\}
\end{equation}
is positively correlated with $\{ \mathbf V_j=V \} \cap A$. Before proving this property, let us conclude the proof. By using this property, we have
\[
\Pro\left[|\mathbf V_{j-1}|<\kappa|V|\, \big| \, \mathbf V_j=V, \mathcal{F}_{j+1}\right]\leq \Pro\Big[\sum_{(x,z)\in V} \mathbf N_{(x,z)} <\kappa |V|\Big] \text{ a.s.}
\]
for any possible outcome $V$ for $\mathbf V_j$. By definition of the tree, there exists a constant $C>0$ such that for each $(x,z)\in V$, $0\le \mathbf N_{(x,z)}\le CK^2$. Moreover, the variables $\mathbf N_{(x,z)}$ are independent from each other since the $(x,z)\in V$ are at a distance at least $30R_j$ from each other and $f_r$ is $r$-dependent. Moreover, by Lemma \ref{l:first_moment}, their expectation is at least $\eta_0K^{\eta_0}>\kappa$. By a standard concentration inequality for sums of i.i.d.~random variables (Hoeffding's inequality) applied to the family of variables $(\kappa- \mathbf N_{(x,z)})_{(x,z)\in V}$, we have
\[
\Pro\Big[\sum_{(x,z)\in V} \mathbf N_{(x,z)}<\kappa|V|\Big]\leq \exp\Big(-\frac{2(\eta_0K^{\eta_0}-\kappa)^2}{(CK^2+\kappa)^2}|V|\Big).
\]
Since this is true for any $V$ and since $\mathcal F_j$ is the smallest $\sigma$-algebra containing $\mathcal F_{j+1}$ and $\mathbf V_j$, this implies \eqref{eq:h1}. 

\medskip

It only remains to show that \eqref{eq:sec4fkglocal***} is positively correlated with $\{ \mathbf V_j=V \} \cap A$ (for any $A,V,M$ as above \eqref{eq:sec4fkglocal***}). To this purpose, we observe that
\begin{itemize}
\item $\{ \mathbf V_j = V \} \cap A$ is increasing in the $11R_j$-neighborhood of $\cup_{(x,z) \in V} \{ x \}$ (here we have used that, since $A \in \mathcal{F}_{j+1}$, $\{ \mathbf V_j = V \} \cap A$ can be written as a union of the intersection of the increasing event $\{ \mathbf V_j \supset V \}$ with events that are measurable with respect to the complement of the $11R_j$-neighborhood of $\cup_{(x,z) \in V} \{ x \}$);
\item $\{\sum_{(x,z)\in V} \mathbf N_{(x,z)} > M \}$ is increasing and measurable with respect to $\cup_{(x,z) \in V} D(x,9R_j)$.
\end{itemize}
So the result holds by applying the local FKG inequality Corollary \ref{cor:FKG} to $U=\R^2$, $V=\cup_{(x,z) \in V} D(x,9R_j)$, $\delta=1$, $\phi=1_{\{ \mathbf V_j = V \} \cap A}$ and $\psi=1_{\{\sum_{(x,z)\in V} \mathbf N_{(x,z)} > M \}}$.
\end{proof}

\section{A two-arms estimate in a slab}
\label{sec5}

Let $q : \R^d \rightarrow \R$ satisfying Assumption \ref{ass1} for some $\beta>d$. In this section, we consider two parameters $\gamma,a\in (0,1)$ such that $\gamma<a^2$ and study connectivity properties in $[-4R,4R]^2 \times [0,R^a]$. Our main goal is to prove the following uniqueness quasi-planar result that plays the role of \eqref{eq:2armsintro} from the sketch of proof. Recall that for $d'\leq d$, we routinely view $D\subset\R^{d'}$ as the subset $D \times \{ 0\}^{d-d'}$ and that we let $\calP_t:=\{x \in \R^3 : x_3=t\}$.

\begin{proposition}\label{prop:2arms*}
There exists $b>0$ that depends only on the dimension $d$ such that the following holds if $0<\gamma<a^2<b$. For all $\delta>0$ there exist $R_0,c>0$ such that for every $R \ge R_0$,
\[
\mathbb P\bigg[\begin{array}{c}\text{all the c.c.~of $\{ f_{R^\gamma} \ge R^{-3/2} \} \cap [-2R,2R]^2$ of diameter at least $ \delta R$}\\
\text{belong to the same c.c.\ of $\{ f_{R^\gamma} \ge R^{-3/2} \}\cap[-4R,4R]^2\times[0,R^a]$}\end{array}\bigg] \ge 1-\exp(-R^c).
\]
\end{proposition}

\begin{remark}
The same proof gives Proposition \ref{prop:2arms*} with any level in $[0,R^{-3/2}]$ instead of level $R^{-3/2}$.
\end{remark}

Before proving Proposition~\ref{prop:2arms*}, let us state the following elementary planar existence result which, in combination with Proposition \ref{prop:2arms*}, will help us create large components in $\{ f_{R^\gamma} \ge R^{-3/2}\}$ with high probability.

\begin{lemma}[Existence of a macroscopic planar component]\label{lem:macro*}
For every $\delta > 0$ there exists $R_0>0$ such that for every $R \ge R_0$, 
\[
\mathbb{P}\big[\text{$\{ f_{R^\gamma} \ge R^{-3/2} \} \cap [-R,R]^2$ contains a c.c.~of diameter $ \ge \delta R$}\big] \ge 1-\left( \frac{1}{2} \right)^{\lfloor 1/(2\delta) \rfloor^2-1}.
\]
\end{lemma}

\begin{proof}
For every $\delta>0$, there exist at least $\lfloor 1/(2\delta) \rfloor^2$ squares of side length $\delta R$ which are included in $[-R,R]^2$ and are at a mutual distance larger than $\delta R$. An easy application of Lemma \ref{lem:mani*} and the duality between $\{f_{R^\gamma}>0\}$ and $\{f_{R^\gamma}<0\}$ implies that each square is crossed by $\{f_{R^\gamma}\geq 0\}$ with probability $1/2$. Choose $R_0$ such that $R^\gamma \le \delta R$ for all $R \ge R_0$. Then, for all such $R$,
\[
 \mathbb{P}\big[\text{$\{ f_{R^\gamma} \ge 0\} \cap [-R,R]^2$ contains a c.c.~of diameter at least $\delta R$}\big] \ge 1-\left( \frac{1}{2} \right)^{\lfloor 1/(2\delta) \rfloor^2}.
\]
We conclude by applying Lemma \ref{lem:Cameron-Martin} to the event on the left-hand side with $r=R^\gamma$ and $t=-R^{-3/2}$.
\end{proof}

Before proving Proposition \ref{prop:2arms*}, let us write two remarks that are central in the proof. The first remark is mainly about the use of the local FKG inequality in the proof (which is more subtle than one may expect). The second remark is about the fact that we can use RSW results at level $R^{-3/2}$ instead of $0$.

\begin{remark}\label{rmk:conditioning_and_fkg}
In the proof of Proposition \ref{prop:2arms*}, we will work in $\R^3 \subset \R^d$ and frequently condition the field $f_{R^\gamma}$ with respect to its values on certain subsets. We first observe that, for any subset $U\subseteq\R^3$, the law of $f_{R^\gamma}$ conditioned on $(f_{R^\gamma})_{|U}$ is still Gaussian (to prove this, note for instance that, since the field is a.s.\ continuous, it is determined by its values on a countable number of points).

Moreover, if $V\subseteq\R^3$ is such that for each $x\in U$ and $y\in V$, $\E[f_{R^\gamma}(x)f_{R^\gamma}(y)]=0$, then the law of $(f_{R^\gamma})_{|V}$ is unaffected under the conditioning on $(f_{R^\gamma})_{|U}$. Furthermore, by the local FKG inequality (see Corollary \ref{cor:FKG}), the following property holds, which is tailored to the purposes of the present section:

\medskip 
Fix $W,U_A,U_B\subseteq\R^3$ and let $A \in \mathcal{F}_{U_A}$ and $B \in \mathcal{F}_{U_B}$ (recall the definition of these $\sigma$-algebras from Section \ref{ssec:not}). Assume that $B$ is increasing, that $A$ is increasing on the set of points in $\R^3$ at a distance at most $R^\gamma+1$ from $U_B$ and that $\textup{dist}(U_A,W) \ge 2R^\gamma+1$. Then, applying Corollary \ref{cor:FKG} to the conditional (Gaussian) measure $\prob[\,\cdot\, |\, (f_{R^\gamma})_{|W}]$ with the choices $U=U_A$, $V=U_B$, $\phi=1_A$, $\psi=1_B$ and $\delta=1$, one obtains that
\begin{equation}
\label{eq:conditionalFKGlocal}
\prob[f_{R^\gamma} \in A\cap B \mid (f_{R^\gamma})_{|W}]\geq\prob[f_{R^\gamma} \in A \mid (f_{R^\gamma})_{|W}] \times \prob[f_{R^\gamma} \in B \mid (f_{R^\gamma})_{|W}].
\end{equation}
\end{remark}

\begin{remark}\label{rem:sprinkling_rswetc}
As one can notice in Proposition \ref{prop:2arms*}, our goal in this section is to connect planar components by using paths in $\{ f_{R^\gamma} \ge R^{-3/2} \}$. As a result, it will be very useful to have at our disposal RSW results at level $R^{-3/2}$ instead of $0$. Let us first recall that since the RSW theorem (Theorem \ref{thm:rsw}) holds with universal constants, then it holds with $f_{R^\gamma}$ instead of $f$. Next, by applying Lemma \ref{lem:Cameron-Martin} (to $r=R^\gamma$ and $t=-R^{-3/2}$), we obtain that there exits $R_0>0$ such that the following holds:
\begin{equation}\label{eq:sprinkling_rsw**}
\begin{array}{c}
\text{If $R>R_0$ then Theorem \ref{thm:rsw} holds with $\{ f \ge 0 \}$ replaced by $\{ f_{R^\gamma} \ge R^{-3/2} \}$}\\
\text{in the definition of $\cross(\rho R,R)$.}
\end{array}
\end{equation}
The same argument implies that there exists $R_0>0$ such that we have the following:
\begin{equation}\label{eq:sprinkling_poly**}
\begin{array}{c}
\text{If $R>R_0$ then Proposition \ref{prop:polynom_absctract} holds for any $\rho \in [R^\gamma,R]$}\\
\text{and with $\{ f_r \ge 0 \}$ replaced by $\{ f_{R^\gamma} \ge R^{-3/2} \}$.}
\end{array}
\end{equation}
\end{remark}

Let us now prove Proposition \ref{prop:2arms*}.

\begin{proof}[Proof of Proposition \ref{prop:2arms*}]
Recall that we consider two parameters $0<\gamma<a^2<1$ which are both assumed to be small. Throughout the proof, we will frequently condition on $(f_{R^\gamma})_{|\calP_0}$. We will denote by $\widetilde{\prob}$ the probability law with this conditioning (viewed as a regular conditional probability measure). Consider the event
\begin{equation}\label{eq:mw_def}
\textup{Sprouts}(R):=\left\{\begin{array}{c}\text{every continuous path included in $\{f_{R^\gamma} \geq R^{-3/2} \} \cap [-2R,2R]^2$ of}\\
\text{diameter at least $R^a$ is connected to $\mathcal{P}_{R^{a^2}}$ by a path included in}\\
\text{$\{ f_{R^\gamma} \ge R^{-3/2} \} \cap ([-2R,2R]^2\times[0,R^{a^2}])$ of diameter $\le 3R^a$}\end{array}\right\}.
\end{equation}
At this point, we assume that $a$ and $\gamma/a$ are small enough for Proposition \ref{prop:mw} to apply at scale $R^a$, with truncation exponent $\gamma/a$ instead of $\gamma$. By Proposition~\ref{prop:mw} (at $\ell=R^{-3/2}$) followed by a union bound,  there exist $R_0,c_0>0$ such that if $R\geq R_0$, 
\begin{equation}\label{eq:sprouts_1}
\prob[\textup{Sprouts}(R)]\geq 1-e^{-R^{c_0}}.
\end{equation}
Let $\widetilde{\textup{Sprouts}}(R)$ be the $(f_{R^\gamma})_{|\calP_0}$-measurable event
\begin{equation}\label{eq:sprouts_tilde_def}
\widetilde{\textup{Sprouts}}(R):=\left\{\widetilde{\prob}[\textup{Sprouts}(R)]\geq 1-e^{-R^{c_0}/2}\right\}.
\end{equation}
Then, by \eqref{eq:sprouts_1} and Markov's inequality applied to $\widetilde{\textup{Sprouts}}(R)^c$, one finds for $R\geq R_0$,
\begin{equation}\label{eq:sprouts_2}
\prob[\widetilde{\textup{Sprouts}}(R)]\geq 1-e^{-R^{c_0}/2}.
\end{equation}


Let $\delta>0$. The following lemma is the core of the proof of Proposition \ref{prop:2arms*}. Conditionally on $(f_{R^\gamma})_{|\calP_0}$ and on the (very likely) event $\widetilde{\textup{Sprouts}}$, one may connect two given large connected components of $\{f_{R^\gamma}\geq R^{-3/2}\}\cap[-2R,2R]^2$ by a path in $\{f_{R^\gamma}\geq R^{-3/2}\}\cap([-4R,4R]^2\times[0,R^a])$ with very good probability. To conclude the proof, we will essentially apply this lemma to all pairs of connected components of this set.
\begin{lemma}\label{lemma:two_arms_1}
Suppose that $a$ satisfies $5a/(1-a)<\eta$. There exist $c,R_0>0$ such that the following holds. Assume that $\widetilde{\textup{Sprouts}}(R)$ is satisfied and let $\calC$ and $\calC'$ be two continuous paths included in $\{f_{R^\gamma}\geq R^{-3/2}\}\cap[-2R,2R]^2$ of diameter $\delta R$ and at a mutual distance at least $100\delta R$. If $R \ge R_0$ then
\[
\widetilde{\prob}\left[\calC\leftrightarrow\calC' \text{ in } ([-4R,4R]^2\times[0,R^a])\cap\{f_{R^\gamma}\geq R^{-3/2}\}\right]\geq 1-e^{-R^c}.
\]
\end{lemma}

Before proving the lemma, we briefly conclude the proof of Proposition~\ref{prop:2arms*}. Recall that we have already assumed that $a$ and $\gamma/a$ are small enough. In order to apply Lemma \ref{lemma:two_arms_1}, we assume in addition that $5a/(1-a)<\eta$. This also holds if $a$ is small enough.

\medskip

Let $\delta_0:=\delta/1000$ and observe that if $\mathcal{D}$ and $\mathcal{D}'$ are two connected components of $\{ f_{R^\gamma} \ge R^{-3/2} \} \cap [-2R,2R]^2$ of diameter larger than or equal to $ \delta R$ then there exist two continuous paths $\mathcal{C}$ and $\mathcal{C}'$  of diameter $\delta_0 R$, included in $\mathcal{D},\mathcal{D}'$ respectively and satisfying $\textup{dist}(\mathcal{C},\mathcal{C}') \ge 100\delta_0 R$. As a result, Proposition \ref{prop:2arms*} is a consequence of Lemma \ref{lemma:two_arms_1} and a union bound over all pairs of connected components of $\{ f_{R^\gamma} \ge R^{-3/2} \} \cap [-2R,2R]^2$ of diameter $\ge \delta R$, the number of which can be controlled except on an event of suitably small probability as follows.

\medskip

Let $c>0$ as in Lemma \ref{lemma:two_arms_1}. The probability that there are more than $e^{R^c/2}$ connected components in $\{ f_{R^\gamma}\ge R^{-3/2} \} \cap [-2R,2R]^2$ is less than $e^{-R^c/4}$ if $R$ is sufficiently large. This is a direct consequence of Markov's inequality and the fact that the expectation of the number of connected components of $\{ f_{R^\gamma}\ge R^{-3/2} \} \cap [-2R,2R]^2$ is less than $CR^2$ for some $C>0$ that depends only on $q$, as soon as $R$ is sufficiently large, see Lemma \ref{L:numbercomps}. This completes the proof of Proposition~\ref{prop:2arms*}, subject to the validity of Lemma~\ref{lemma:two_arms_1}.
\end{proof}

In order to prove Lemma~\ref{lemma:two_arms_1}, let us introduce some events and apply Proposition \ref{prop:polynom_absctract} to them. Let $x\in[-2R,2R]^2$ and consider a deterministic path $\mathcal{C}_0\subset D(x,5\delta R)$ of diameter $\delta R$. By Proposition~\ref{prop:polynom_absctract} (that we can apply since $\gamma<a$) and \eqref{eq:sprinkling_poly**}, there exist $R_0>0$, a universal $\eta>0$, and a family $(y_i)_i$ of points at mutual distances at least $30R^a$ and at a distance at most $R^a$ from $\mathcal{C}_0$ such that if $R \ge R_0$ and if we let
\begin{align}\label{eq:poly_def} \nonumber
\text{Conn}_i(R)&:=\Big\{\begin{array}{c} D(y_i,R^a)\text{ is connected~to $\partial D(x,20\delta R)$}\\
\text{ in $(\{f_{R^\gamma} \ge R^{-3/2} \} \cap D(x,20\delta R) ) \setminus ( \cup_{j \ne i} D(y_j,10R^a))$}\end{array}\Big\},\\
\textup{Poly}_{\mathcal{C}_0}(R)&:=\{ \#\{i : \textup{Conn}_i(R) \text{ holds} \} \ge (\delta R^{(1-a)})^\eta \},
\end{align}
then
\begin{equation}\label{eq:poly*}
\Pro [\textup{Poly}_{\mathcal{C}_0}(R)] \geq \eta.
\end{equation}
We are now in shape to prove Lemma \ref{lemma:two_arms_1}.

\begin{proof}[Proof of Lemma~\ref{lemma:two_arms_1}]
Throughout the proof, crucially, the constants do not depend on $\calC$ and $\calC'$. The proof is split into two steps.

\medskip

\textbf{Step 1.} In this step, we consider copies of the event $\textup{Poly}_{\calC}(R)$ at different heights and show, using \eqref{eq:poly*}, that with very good probability, at some height between $R^a/2$ and $R^a$ there exist polynomially many paths that, when projected onto $\calP_0$, connect $\calC$ and $\calC'$ up to distance $O(R^a)$.

\medskip

Let $x,x'\in[-2R,2R]^2$ be such that $\calC\subset D(x,5\delta R)$ and $\calC'\subset D(x',5\delta R)$. Let $(y_i)_i$ (resp.~$(y_j')_j$) be a family of points associated to $\calC$ (resp.~$\calC'$) in the same way as the points associated to $\calC_0$ in the definition of $\textup{Poly}_{\calC_0}(R)$ in \eqref{eq:poly_def} above. For every $i$ (resp.~$j$), let $\mathcal{C}_i$ (resp.~$\mathcal{C}_j'$) be a path included in $\mathcal{C}$ (resp.~$\calC'$) of diameter $R^a$ and at a distance at most $ R^a$ from $y_i$ (resp.~$y_j'$). (These sub-paths exist and are disjoint as $i$ and $j$ vary if $R$ is sufficiently large.)

\medskip

Let $\textup{Circ}_{\mathcal{C},\mathcal{C}'}(R)$ denote the event that
\begin{itemize}[noitemsep]
\item[i)] there is a circuit in $\{ f_{R^\gamma} \ge R^{-3/2} \} \cap ( D(x,20\delta R) \setminus D(x,10\delta R) )$ which surrounds the inner disc $D(x,10\delta R)$, 
\item[ii)] the analogous event holds with $x'$ instead of $x$, and 
\item[iii)] $D(x,10\delta R)$ is connected to $D(x',10\delta R)$ by a path included in $\{ f_{R^\gamma} \ge R^{-3/2} \} \cap [-4R,4R]^2$.
\end{itemize} 
Introduce
\[
\textup{Poly}_{\mathcal{C},\mathcal{C}'}(R):=\textup{Poly}_{\mathcal{C}}(R) \cap \textup{Poly}_{\mathcal{C}'}(R) \cap \textup{Circ}_{\calC,\calC'}(R).
\]

For $0\le k\le \bar k$ with $\bar{k}:=\lfloor R^{a-\gamma}/2 \rfloor$, let $\textup{Poly}_{\mathcal{C},\mathcal{C}'}^k(R)$ (resp.~$\textup{Conn}_i^k(R)$) denote the event $\textup{Poly}_{\mathcal{C},\mathcal{C}'}(R)$ (resp.~$\textup{Conn}_i(R)$) translated by $h_k\textbf{e}_3$, where $h_k:=\frac{R^a}{2}+kR^\gamma$.

\medskip

Note that $\mathcal{P}_{h_0}=\mathcal{P}_{R^a/2}$ and that the planes $\mathcal{P}_{h_k}$ are at mutual distances $R^\gamma$ and are included in $\R^2 \times [R^a/2,R^a]$. Since $f_{R^\gamma}$ is $R^\gamma$-dependent, \eqref{eq:poly*}, the RSW theorem (see \eqref{eq:sprinkling_rsw**}) and standard gluing constructions (as described in Section \ref{ss:rsw}) imply the existence of $c>0$ that depends only on $\delta$ such that for all $R$ sufficiently large
\begin{equation}\label{eq:layers_poly*}
\widetilde{\Pro} \bigg[\bigcup_{ 0\le k\le \bar{k}}\textup{Poly}_{\mathcal{C},\mathcal{C}'}^k(R)\bigg]=\Pro \bigg[\bigcup_{ 0\le k\le \bar{k}}\textup{Poly}_{\mathcal{C},\mathcal{C}'}^k(R)\bigg] \geq 1-e^{-cR^{a-\gamma}}.
\end{equation}

If there exists $k \in \{0,\dots,\bar{k} \}$ such that $\textup{Poly}_{\mathcal{C},\mathcal{C}'}^k(R)$  holds,  let $k_{\rm max}$ be the largest such $k$ and define $\bar{I},\bar{I}'$ by $i \in \bar{I}$ if and only if $\textup{Conn}_i^{k_{\rm max}}(R)$ occurs (and similarly for $\bar{I}'$). For every $k_0 \in \{0,\dots,\bar{k}\}$ and every $\bar{I}_0,\bar{I}_0'$, let
\[
A_{k_0,I_0,I_0'}=\Big(\bigcup_k\textup{Poly}^k_{\calC,\calC'}(R) \Big)\cap\{ k_{\rm max}=k_0,\bar{I}=\bar{I}_0,\bar{I}'=\bar{I}_0'\}.
\]
The event $A_{k_0,I_0,I_0'}$ is measurable with respect to the set
\begin{equation}\label{eq:a_meas}
\{x_3\geq h_{k_0}\}\setminus\Big(\bigcup_i D(y_i,R^a)\cup\bigcup_j D(y_j',R^a)\Big)
\end{equation}
and increasing in the set
\begin{equation}\label{eq:a_increasing}
\Big(\bigcup_{i \in I_0} D(y_i,10R^a)  \cup \bigcup_{j \in I_0'} D(y_j',10R^a)\Big) \times [0,h_{k_0}+R^\gamma/2].
\end{equation}

\medskip

\textbf{Step 2.} In this step we show that, given $k_0\in\{1,\dots,\overline{k}\}$ and $I_0$ a family of indices, it is very likely that there exists some $i\in I_0$ such that $\calC_i$ is connected to a path in $D(y_i,5R^a)\times\{h_{k_0}\}$ of diameter at least $R^a$. Here the difficulty is that we want to do this under $\widetilde{\prob}$, i.e., conditionally on $(f_{R^\gamma})_{|\calP_0}$. In doing so we lose the independence of $f_{R^\gamma}$ restricted to the tubes $D(y_i,5R^a)\times[0,h_{k_0}]$. This lack of independence is replaced with Claim \ref{cl:two_arm_1} below.

\medskip

Fix $k_0\in\{1,\dots,\overline{k}\}$ and $I_0$ a family of indices such that
\begin{equation}\label{eq:index_set_size}
|I_0|\geq (\delta R^{(1-a)})^\eta\, .
\end{equation}
Denote  the elements of $I_0$ by $i_1,\dots,i_{|I_0|}$. For each $l\in\{1,\dots,|I_0|\}$, let $H_l$ be the event that there exists an index $1\leq l'\leq l$ such that $\calC_{i_{l'}}$ is connected by a path in $(D(y_{i_{l'}},5R^a)\times[0,h_{k_0}])\cap\{f_{R^\gamma}\geq R^{-3/2}\}$ to a circuit surrounding $D(y_{i_{l'}},R^a)\times\{h_{k_0}\}$ in $(D(y_{i_{l'}},5R^a)\times\{h_{k_0}\})\cap\{f_{R^\gamma}\geq R^{-3/2}\}$. 

\begin{claim}\label{cl:two_arm_1}
Let $c_0$ as in \eqref{eq:sprouts_1}. There exist $c,R_0>0$ such that the following holds if $R\geq R_0$. Let $l\in\{1,\dots,|I_0|-1\}$ and assume that
\[
\widetilde{\prob}\left[H_l\right]\leq 1-e^{-R^{c_0}/4}.
\]
Then, on $\widetilde{\textup{Sprouts}}(R)$,
\[
\widetilde{\prob}\left[H_{l+1}\, |\, H_l^c\right]\geq cR^{-5a}.
\]
Moreover, on $\widetilde{\textup{Sprouts}}(R)$, we have $\widetilde{\prob}[H_1]\geq cR^{-5a}$.
\end{claim}

Before proving Claim \ref{cl:two_arm_1}, let us complete the proof of Lemma~\ref{lemma:two_arms_1}. By Claim \ref{cl:two_arm_1}, we deduce that on $\widetilde{\textup{Sprouts}}(R)$, the probability of the event $H_{|I_0|}$ that there exists $i\in I_0$ for which $\calC_i$ is connected by a path in $(D(y_i,5R^a)\times[0,h_{k_0}])\cap \{f_{R^\gamma} \ge R^{-3/2} \}$ to a circuit surrounding $D(y_i,R^a) \times \{h_{k_0}\}$ in $(D(y_i,5R^a) \times \{h_{k_0}\})\cap\{f_{R^\gamma}\geq R^{-3/2}\}$ satisfies
\[
\widetilde{\prob}[H_{|I_0|}] \geq 1-(1-c R^{-5a})^{|I_0|}\geq 1-e^{-R^{c'}}
\]
for $R$ large enough and some $c'>0$ since $|I_0|\geq (\delta R^{(1-a)})^\eta$ by \eqref{eq:index_set_size} and since we have assumed that $5a<\eta(1-a)$. Clearly, the same holds for the analogous construction with $\calC'$ instead of $\calC$. In addition to $k_0$ and $I_0$, let $I_0'$ be a set of indices $j$ of the family $(y_j')_j$ satisfying $|I_0'|\geq (\delta R^{(1-a)})^\eta$. Let $B_{k_0,I_0,I_0'}$ be the event that $H_{|I_0|}$ holds and that the analogous event for $\calC'$ holds as well. By union bound, we deduce that on the event $\widetilde{\textup{Sprouts}}(R)$, for $R$ large enough,
\begin{equation}\label{eq:final_up_lower_bound}
\widetilde{\prob}[B_{k_0,I_0,I_0'}]\geq 1-2e^{-R^{c'}}.
\end{equation}
Now this event is increasing and measurable with respect to the set
\[
\Big( \bigcup_{i\in I_0} D(y_i,5R^a)\cup \bigcup_{j\in I_0} D(y_j',5R^a) \Big) \times [0,h_{k_0}].
\]
Since, $A_{k_0,I_0,I_0'}$ is measurable with respect to the set \eqref{eq:a_meas} and increasing on the set \eqref{eq:a_increasing} we deduce by applying \eqref{eq:conditionalFKGlocal} that on the event $\widetilde{\textup{Sprouts}}(R)$, for $R$ large,
\[
\widetilde{\prob}[A_{k_0,I_0,I_0'}\cap B_{k_0,I_0,I_0'}]\geq \widetilde{\prob}[B_{k_0,I_0,I_0'}] \widetilde{\prob}[A_{k_0,I_0,I_0'}] \ge (1-2e^{-R^{c'}})\widetilde{\prob}[A_{k_0,I_0,I_0'}].
\]
Recall now that the disjoint union of the $A_{k_0,I_0,I_0'}$ over all the choices of $k_0$, $I_0$ and $I_0'$ is $\bigcup_{k=1}^{\overline{k}}\textup{Poly}_{\calC,\calC'}^k(R)$. Consequently, summing over all these possible choices, we deduce that, on the event $\widetilde{\textup{Sprouts}}(R)$, for $R$ large enough,
\begin{align}\label{eq:two_arm_proof_1}
\nonumber \widetilde{\prob}\bigg[\bigsqcup_{k_0,I_0,I_0'}A_{k_0,I_0,I_0'}\cap B_{k_0,I_0,I_0'}\bigg]&\geq (1-2e^{-R^{c'}})\widetilde{\prob}[\cup_{k=1}^{\overline{k}}\textup{Poly}_{\calC,\calC'}^k(R)]\\
&\overset{\eqref{eq:layers_poly*}}{\geq} (1-2e^{-R^{c'}})(1-e^{-c_2R^{a-\gamma}})\, .
\end{align}
But notice that the event $\bigsqcup_{k_0,I_0,I_0'}(A_{k_0,I_0,I_0'}\cap B_{k_0,I_0,I_0'})$ implies that $\calC$ and $\calC'$ are connected by a continuous path in $([-4R,4R]^2 \times [0,R^a]) \cap \{ f \ge R^{-3/2}\}$. Therefore, \eqref{eq:two_arm_proof_1} completes the proof of Lemma~\ref{lemma:two_arms_1} (assuming Claim~\ref{cl:two_arm_1} holds true).
\end{proof}

It remains to give the proof of Claim \ref{cl:two_arm_1}.
\begin{proof}[Proof of Claim \ref{cl:two_arm_1}]
Recall the definition \eqref{eq:sprouts_tilde_def} of $\widetilde{\textup{Sprouts}}(R)$. Assume that
\[
\widetilde{\prob}[H_l]\leq 1-e^{-R^{c_0}/4}.
\]
Let $\textup{Up}_{l+1}$ be the event that $\calC_{i_{l+1}}$ is connected to $D(y_{i_{l+1}},5R^a)\times\{R^{a^2}\}$ by a continuous path in $\{f_{R^\gamma}\geq R^{3/2}\}\cap (D(y_{i_{l+1}},5R^a)\times[0,R^{a^2}])$. Then, $\textup{Sprouts}(R)\subset \textup{Up}_{l+1}$ so that, on $\widetilde{\textup{Sprouts}}(R)$, for $R$ large enough,
\[
\widetilde{\prob}[\textup{Up}_{l+1} \mid H_l^c]\geq \widetilde{\prob}[\textup{Sprouts}(R) \mid H_l^c] \ge 1-\frac{\widetilde{\prob}[\textup{Sprouts}(R)^c]}{\widetilde{\prob}[H_l^c]}\geq 1-e^{-R^{c_0}/4}.
\]
In particular, by union bound and FKG, we deduce that there exist a constant $c_1>0$ and a (random $(f_{R^\gamma})_{|\calP_0}$-measurable) $u_{l+1}\in D(y_{i_{l+1}},5R^a)\times\{R^{a^2}\}$ such that the event $\textup{Up}^\star_{l+1}$ that $\calC_{i_{l+1}}$ is connected to $u_{l+1}$ by a continuous path in $\{f_{R^\gamma}\geq R^{-3/2}\}\cap D(y_{i_{l+1}},5R^a)\times[0,R^{a^2}]$ satisfies, for $R$ large enough,
\begin{equation}\label{eq:claim_two_arm_1}
\widetilde{\prob}[\textup{Up}^\star_{l+1}\, |\, H_l^c]\geq c_1R^{-2a}.
\end{equation}
Let 
\[
\textup{Tube}_{l+1}
:=\Bigg\{\begin{array}{c}\exists \text{ a c.c.\ of } \{ f_{R^\gamma} \ge R^{-3/2} \} \cap (D(y_{i_{l+1}},5R^a) \times [R^{a^2},h_{k_0}])\text{ that contains both}\\
\text{$u_{l+1}$ and a circuit in $\{f_{R^\gamma} \ge R^{-3/2} \} \cap (D(y_{i_{l+1}},5R^a)\times \{h_{k_0}\})$}\\
\text{that surrounds $D(y_{i_{l+1}},R^a) \times \{h_{k_0}\}$} \end{array}\Bigg\}.
\]
Note that $\textup{Up}_{l+1}^\star\cap\textup{Tube}_{l+1} \subset H_{l+1}$ so our goal will be to show that, for some constant $c>0$ and $R$ large enough,
\begin{equation}\label{eq:claim_two_arm_2}
\widetilde{\prob}[\textup{Up}_{l+1}^\star\cap\textup{Tube}_{l+1}\, |\, H_l^c]\geq c R^{-5a}.
\end{equation}

Let $\calE_{l+1}=\R^3 \setminus (D(y_{i_{l+1}},10 R^a)\times [0,h_{k_0}])$. Note that $H_l$ is measurable with respect to $(f_{R^\gamma})_{|\calE_{l+1}}$ while $\textup{Tube}_{l+1}$ is measurable with respect to $f_{R^\gamma}$ restricted to $D(y_{i_{l+1}},5 R^a)\times [0,h_{k_0}]$. On the other hand, by the RSW theorem (see \eqref{eq:sprinkling_rsw**}) and by using standard gluing constructions from Section \ref{ss:rsw} (more precisely, by using three times Item vi) of this section at scale $R^a$: once in $\calP_{h_k}$ and twice in a vertical plane in order to connect $u_i$ to a well-chosen point of $\calP_{h_k}$), one can show that $\textup{Tube}_{l+1}$ occurs with probability $\Omega(R^{-3a})$ so that, for some constant $c_2>0$,
\begin{equation}\label{eq:claim_two_arm_3}
\widetilde{\prob}[\textup{Tube}_{l+1} \, | \, (f_{R^\gamma})_{|\calE_{l+1}}]=\prob[\textup{Tube}_{l+1}]\geq c_2 R^{-3a}.
\end{equation}

Now, note that the events $\textup{Up}^\star_{l+1}$ and $\textup{Tube}_{l+1}$ are both increasing and recall that $\gamma<a^2$. Hence, for $R$ large enough,
\begin{align*}
\widetilde{\prob}[\textup{Up}^\star_{l+1}\cap\textup{Tube}_{l+1}\, |\, (f_{R^\gamma})_{|\calE_{l+1}}]&\overset{\eqref{eq:conditionalFKGlocal}}{\geq}\widetilde{\prob}[\textup{Up}^\star_{l+1}\, |\, (f_{R^\gamma})_{|\calE_{l+1}}]\widetilde{\prob}[\textup{Tube}_{l+1}\, |\, (f_{R^\gamma})_{|\calE_{l+1}}]\\
&\overset{\eqref{eq:claim_two_arm_3}}{\geq} c_2 R^{-3a}\widetilde{\prob}[\textup{Up}^\star_{l+1}\, |\, (f_{R^\gamma})_{|\calE_{l+1}}]
\end{align*}
so that
\begin{align*}
\widetilde{\prob}[\textup{Up}^\star_{l+1}\cap\textup{Tube}_{l+1}\, |\, H_l^c]&=\frac{1}{\widetilde{\prob}[H_l^c]}\widetilde{\E}[\widetilde{\prob}[\textup{Up}^\star_{l+1}\cap\textup{Tube}_{l+1}\, |\, (f_{R^\gamma})_{|\calE_{l+1}}]\mathbf{1}_{H_l^c}]\\
&\geq c_2R^{-3a}\frac{1}{\widetilde{\prob}[H_l^c]}\widetilde{\E}[\widetilde{\prob}[\textup{Up}^\star_{l+1}\, |\, (f_{R^\gamma})_{|\calE_{l+1}}]\mathbf{1}_{H_l^c}]\\
&\geq c_2R^{-3a}\widetilde{\prob}[\textup{Up}^\star_{l+1}\, |\, H_l^c]\\
&\overset{\eqref{eq:claim_two_arm_1}}{\geq} c_1c_2R^{-5a},
\end{align*}
which yields \eqref{eq:claim_two_arm_2} as required.

The remaining statement, i.e.~the proof that on $\widetilde{\textup{Sprouts}}(R)$, $\widetilde{\prob}[H_1]\geq c R^{-5a}$, follows from the same construction.
\end{proof}

\begin{figure}
\begin{center}
\includegraphics[scale=0.5]{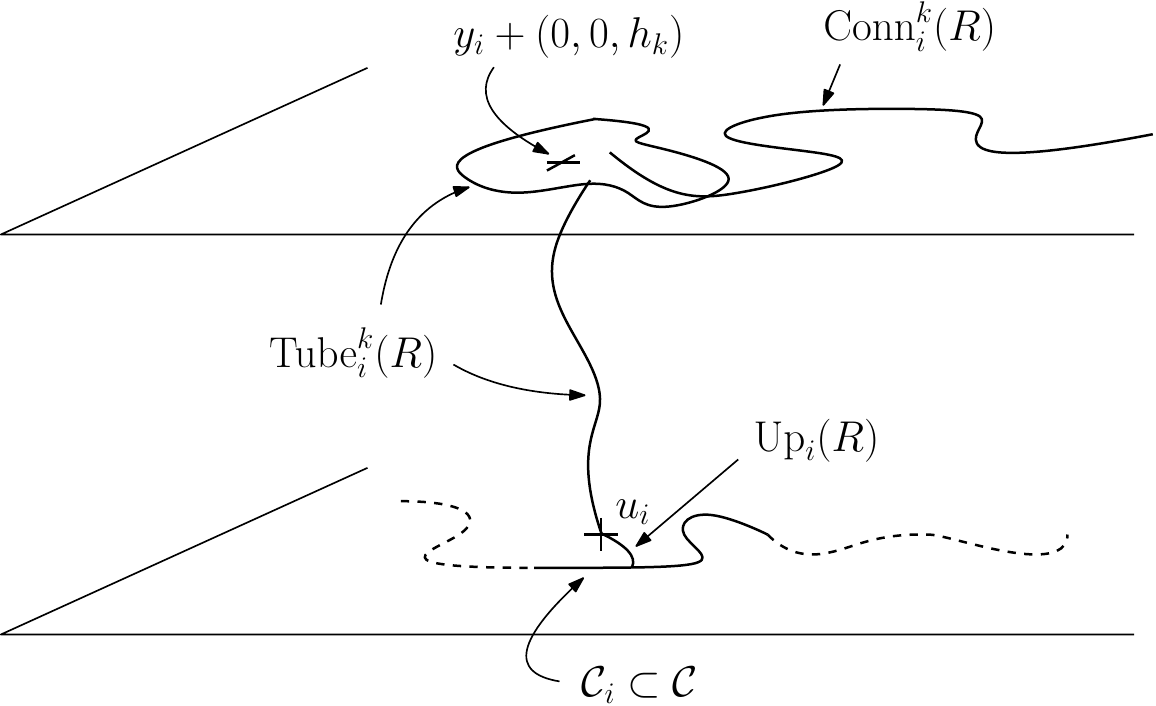}
\caption{The events $\textup{Up}_i(R)$, $\textup{Tube}_i^k(R)$ and $\textup{Conn}_i^k(R)$.\label{fig:2arms}}
\end{center}
\end{figure}

We conclude this section with a result analogous to Proposition \ref{prop:2arms*}. Below, given some $d' \in [3,d]$, $F_R^1,\dots,F_R^{N_{d'}}$ denote all the $2$-dimensional faces of the cube $[-R/2,R/2]^{d'}$, i.e.\ the $F_R^i$'s are the sets obtained from the sets $[-R/2,R/2]^2\times\{\pm R/2\}^{d'-2}$ by permuting the coordinates. Moreover, for every $i \in \{1,\dots,N_{d'}\}$ we let $\tilde{F}_R^i \subset F_R^i$ be the concentric square with side length equal to $R/2$.

\begin{proposition}\label{prop:2armsbis*}
There exists $b>0$ that depends only on the dimension $d$ such that the following holds if $0<\gamma<b$. Let $d' \in [3,d]$. For all $\delta>0$ there exist $R_0,c>0$ such that for every $R \ge R_0$,
\[
\mathbb P\Bigg[\bigcap_{i,j}\Bigg\{ \begin{array}{c}\text{all the c.c.~of $\{ f_{R^\gamma} \ge R^{-3/2} \} \cap \tilde{F}_R^i$ of diameter at least $ \delta R$}\\
\text{and all the c.c.~of $\{ f_{R^\gamma} \ge R^{-3/2} \} \cap \tilde{F}_R^j$ of diameter at least }\\
\text{$ \delta R$ belong to the same c.c.\ of $\{ f_{R^\gamma} \ge R^{-3/2} \}\cap [-R/2,R/2]^{d'}$}\end{array}\Bigg\} \Bigg] \ge 1-e^{-R^c},
\]
where the intersection in the probability is taken over the pairs of indices $1 \le i,j \le N_{d'}$ for which $F_R^i$ and $F_R^j$ share a side.
\end{proposition}

\begin{proof}
As in Proposition \ref{prop:2arms*}, we also consider some $a \in (0,1)$ such that $0 < \gamma < a^2 < b$ and let $h_k:=\frac{R^a}{2}+kR^\gamma$. The proof is essentially the same as Proposition \ref{prop:2arms*} but we need to replace the construction in $[-4R,4R]^2 \times \{h_k\}$ by a construction in the $2$-dimensional faces of $[-(R/2-h_k),R/2-h_k]^{d'}$ that correspond to $F_R^i$ and $F_R^j$.

\medskip

The only new technicality (which appears only if $i \ne j$, so let assume this) is about how to adapt point iii) in the definition of $\textup{Circ}_{\mathcal{C},\mathcal{C}'}(R)$ (see the beginning of the proof of Lemma~\ref{lemma:two_arms_1} for the definition of this event). In order to connect two macroscopic paths which belong to two adjacent $2$-dimensional faces -- which is what will replace $\textup{Circ}_{\mathcal{C},\mathcal{C}'}(R)$, one would probably like to use a box-crossing property for $f_{R^\gamma}$ retricted to the union of two orthogonal half-spaces: $(\R \times \{0\} \times [0,+\infty)) \cup (\R \times [0,+\infty)\times \{0\})$. Such a property is probably tractable (at least for $\gamma$ sufficiently small) but it is sufficient for us to use the following weaker result (together with Lemma \ref{lem:Cameron-Martin} as in Remark \ref{rem:sprinkling_rswetc}).
\begin{lemma}\label{lemma:angle_crossing}
Given some $r \ge 1$, consider the following union $\mathcal{S}_r:=([0,r] \times \{0\} \times [0,r]) \cup ([0,r] \times [0,r] \times \{0\})$ of two orthogonal squares. There exists $c>0$ such that, if $R \ge 1$ satisfies $r_q \le R^\gamma \le r$, then
\[
\Pro \Big[ \begin{array}{c}\text{$\exists$ a c.\ path in $\mathcal{S}_r \cap \{f_{R^\gamma} \ge 0 \}$ that connects the}\\
\text{two sides $[0,r]\times \{0\} \times \{r\}$ and $[0,r]\times \{r\} \times \{0\}$}\end{array} \Big] \ge \frac{c}{\log(r)+R^\gamma}.
\]
\end{lemma}
The proof of Lemma \ref{lemma:angle_crossing} is a variation of a standard percolation argument. We present it in Appendix \ref{ss:angle_crossing} below. Using Lemma \ref{lemma:angle_crossing} with $r$ of the order of $\delta R$, one can follow the proof of Proposition \ref{prop:2arms*} in order to prove Proposition \ref{prop:2armsbis*}. The only difference is that we need to replace point iii) in the definition of $\textup{Circ}_{\mathcal{C},\mathcal{C}'}(R)$ by an event of probability at least $c/R^\gamma$. As a result, we need to replace the lower bound from \eqref{eq:layers_poly*} by $1-e^{-c_2R^{a-2\gamma}}$. Taking $a$ small enough followed by $\gamma$ small enough, we obtain that the analogue of Lemma \ref{lemma:two_arms_1} holds with $\mathcal{C} \subset \tilde{F}_R^i$ and $\mathcal{C}' \subset \tilde{F}_R^j$ where $F_R^i$ and $F_R^j$ share a side. The rest of the proof is exactly the same as Proposition \ref{prop:2arms*} so we omit the details.
\end{proof}

\section{Renormalization}\label{s:RS}
\label{sec6}
We now supply a suitable renormalization scheme, with two purposes in mind. First, an application of this scheme and of the $2$-arms estimate Proposition \ref{prop:2arms*} yields the proof of existence of an unbounded nodal component, in a thick (i.e.\ Item i) of Theorem \ref{thm:main}) see Section \ref{sec7}. Later, the scheme will also be used in Section \ref{sec8} in the course of proving Item ii) Theorem~\ref{thm:main} to generate an ambient cluster with good properties.

\begin{definition}[Renormalization scheme]\label{def:renormalization_scheme}
Let $(\Omega,\calF,\prob)$ be a probability space equipped with a measure preserving action of $\Z^d$ which we denote by $(\tau_x)_{x \in \Z^d}$. Call \textit{renormalization scheme} the following data:
\begin{itemize}[noitemsep]
\item A {\em relative scale parameter} $\lambda\in\N$ such that $\lambda \ge 2$. For each $n\in\N$, we let $L_n:=\lambda^n$;
\item An initial event $G_{0,0} \in \mathcal{F}$;
\item For each $n\in\N \setminus \{0\}$, an event $H_n \in \mathcal{F}$;
\item A {\em range parameter} $\rho\in\N \setminus \{0\}$;
\item A {\em separation parameter} $\sigma \in \N \setminus \{0\}$.
\end{itemize}
Given this data, define a family of events $(G_{n,x})$ indexed by $n\in\N$ and $x\in L_n\Z^d$ in the following way (the following approach was pioneered in \cite{MR2680403,MR2891880}). The event $G_{0,0}$ is the initial event; for each $n\in\N$ and $x\in L_n\Z^d$, 
\[
G_{n,x}:=\tau_x^{-1}(G_{n,0})
\] and for each $n \in \N \setminus \{0\}$,
\begin{equation}
\label{eq:cascades}
G_{n,0}:=H_n\cap\bigcap_{x_1,x_2} (G_{n-1,x_1}\cup G_{n-1,x_2})
\end{equation}
where the intersection runs over the pairs $x_1,x_2\in L_{n-1}\Z^d\cap  [-4\rho L_n,4\rho L_n]^d$ such that
\[\textup{dist}(x_1,x_2)\geq \sigma L_{n-1}.
\]
\end{definition}

\begin{remark} In absence of $H_n$, the events $G_{n,0}$ in \eqref{eq:cascades} are called \textit{cascading} in the context of \cite[Section 3]{MR2891880}; for the benefits of adding events $H_n$ at renormalized scales, see \cite[Section 2]{DGRS20}. The events $G_{n,0}$ and $H_n$ will be typical in applications, see \eqref{eq:stretch_1}, \eqref{eq:stretch_2}, \eqref{eq:bound qn} below. One central aspect of the definition of the events $G_{n,x}$ is that if $x_1 \in L_{n-1}\Z^d\cap  [-4\rho L_n,4\rho L_n]^d$ and if $G_{n,0} \setminus G_{n-1,x_1}$ holds, then $G_{n-1,x_2}$ holds for all $x_2 \in L_{n-1}\Z^d\cap  [-4\rho L_n,4\rho L_n]^d$ which are sufficiently far from $x_1$ in the sense that $\textup{dist}(x_1,x_2)\geq \sigma L_{n-1}$. This property propagates down to level $n=0$. As to the events $H_{n}$, if one removes from $[-4\rho L_n,4\rho L_n]^d$ the bad regions (where $G_{k,y}^c$ occurs for some $y \in [-4\rho L_n,4\rho L_n]^d$ and $k < n$), then any point $z$ in the remaining region ``sits'' in a tower of renormalized events (suitable shifts of $H_k$ for $k \leq n$), which all occur. This observation motivates the notion of a black vertex below.
\end{remark}



\begin{proposition}[Stretched-exponential decay of probabilities]\label{lem:stretch_exp_decay}
Fix a renormalization scheme with data $\lambda$, $G_{0,0}$, $(H_n)_n$, $\rho$ and $\sigma$. Assume that for any $n\in\N$ and $x_1,x_2\in L_n\Z^d$ such that $\textup{dist}(x_1,x_2)\geq \sigma L_n$, the events $G_{n,x_1}$ and $G_{n,x_2}$ are independent. For each $n\in\N$, set $q_n:=\prob[G_{n,0}^c]$ and assume that
\begin{equation}\label{eq:stretch_1}
q_0\leq \overline{q}_0:=\frac{1}{4(3\rho\lambda)^{2d}}\, .
\end{equation}
Also, assume that for all $n \ge 1$,
\begin{equation}\label{eq:stretch_2}
\mathbb{P}[H_{n}^{\mathsf{c}}]\leq 
\overline{q}_0 2^{-2^n}.
\end{equation}
Then, for all $n\geq 0$,
\begin{equation}\label{eq:bound qn}
q_n \leq (2\overline{q}_0)2^{-2^{n}}\, .
\end{equation}
\end{proposition}
\begin{proof}
We prove the result by induction on $n$. For each $n\in\N$, apply a union bound over the pairs $(x_1,x_2)$ defining $G_{n,0}$ to get
\[
q_{n+1}=\prob[G_{n+1,0}^c]\leq \prob[H_{n+1}^c]+(3\rho\lambda)^{2d} q_n^2\, .
\]
By induction (or \eqref{eq:stretch_1} for $n=0$) and \eqref{eq:stretch_2}, $q_n\leq (2\overline{q}_0) 2^{-2^n}$ and $\prob[H_{n+1}^c]\leq \overline{q}_0 2^{-2^{n+1}}$. Plugging these estimates in the previous displayed equation implies that
\[
q_{n+1}\leq \overline{q}_0 2^{-2^{n+1}}+4(3\rho\lambda)^{2d}\overline{q}_0^2\times 2^{-2^{n+1}}\, .
\]
By \eqref{eq:stretch_1} we obtain the desired bound for $q_{n+1}$.
\end{proof}

We now introduce the notion of black vertex.

\begin{definition}[Black vertex]\label{defi:black_vertex}
This notion depends on a scale parameter $n \in \N$. We call a vertex $x\in\Z^d$ {\em black} (at scale $n$) if $G_{0,x}$ is satisfied and if, for each $m\in\{1,\dots,n\}$, there exists $y \in L_m\Z^d$ such that  $\tau_y^{-1}(H_m)$ holds and $x\in y+[-4\rho L_m,4 \rho L_m]^d$.
\end{definition}

Below, we consider ``connected subsets'' and ``paths''  for the usual hypercubic lattice with vertex-set $\Z^d$.

The following (deterministic) result is an adaptation of \cite[Lemma~8.6]{DPR18.2}. 

\begin{proposition}[Geometric properties of the renormalization scheme]\label{prop:RS_geometry}
Fix a renormalization scheme with data $\lambda$, $G_{0,0}$, $(H_n)_n$, $\rho$ and $\sigma$. Assume that
\begin{equation}\label{eq:RS_cond_1}
\lambda \rho \geq 100 \sigma \text{ and } \rho \ge 2.
\end{equation}
Let $n\geq 1$ and assume $G_{n,0}$ is satisfied. Then, for any connected sets $S_1,S_2\subset \Z^d \cap [-\rho L_n,\rho L_n]^d$ such that $\textup{diam}(S_1),\textup{diam}(S_2)\geq 10 \sigma L_{n-1}$, there exists a path of black vertices in $\Z^d \cap [-4\rho L_n,4\rho L_n]^d$ whose endpoints belong to $S_1$ and $S_2$ respectively.
\end{proposition}

\begin{remark}
The proof of the previous proposition is a little easier in the particular case $d=2$ (which is the only case used in the proof that $\ell_c<0$), so the reader only interested in the existence of the unbounded nodal component can use planarity to simplify the proof below. Let us also stress that these methods are rather robust and do not in fact rely on the specific symmetries of $\R^d$ (essentially, as long as balls have polynomial volume growth and the base geometry does not exhibit ``large bottlenecks'', these techniques are likely to apply, cf.~\cite{DPR18.2}).
\end{remark}

\begin{proof}
We proceed by induction on $n$ and prove the property for $G_{n,x}$ and $S_1,S_2\subset x + [-\rho L_n,\rho L_n]^d$, with $x\in L_n\Z^d$ instead of $x=0$. 

\medskip

Let us start with the case $n=1$. For this case, we use neither \eqref{eq:RS_cond_1} nor the fact that $S_1,S_2$ are connected. Let $S_1,S_2\subset \Z^d \cap [-\rho L_1,\rho L_1]^d$ be two sets of diameter at least $10\sigma L_0 = 10\sigma$. There exist two paths of vertices in $\Z^d \cap [-2\rho L_1,2\rho L_1]^d$, at a mutual distance at least $\sigma$, whose endpoints belong to $S_1$ and $S_2$ respectively\footnote{One can prove this as follows: Let $x_i,y_i \in S_i$ such that $\textup{dist}(x_i,y_i) \ge 10\sigma$, $i \in \{1,2\}$ and choose the indexation of the points so that $\textup{dist}(x_1,y_2),\textup{dist}(x_2,y_1)\ge \textup{dist}(x_1,x_2)$. Then, one can let one of the two paths be an approximation of the segment $[x_1,x_2]$, let $U$ be the $5\sigma$-neighborhood  of this path and use that 

i) since $\textup{dist}(x_1,y_1),\textup{dist}(x_2,y_2) \ge 10\sigma$ and $\textup{dist}(x_1,y_2),\textup{dist}(x_2,y_1)\ge \textup{dist}(x_1,x_2)$, we have $y_1,y_2 \notin U$; 

ii) $(\Z^d \cap [-2\rho L_1,2\rho L_1]) \setminus U$ is connected.}. In particular, if $G_{1,0}$ holds then one of these two paths must be entirely black. This proves the property for $n=1$ (and with $[-2\rho L_n,2\rho L_n]^d$ instead of $[-4\rho L_n,4\rho L_n]^d$). The proof for translates of $G_{1,0}$ is identical.

\medskip

Let $n \ge 2$, assume that the proposition holds for $n-1$ and let $S_1,S_2\subset \Z^d \cap [-\rho L_n,\rho L_n]^d$ be two connected sets of diameter at least $10 \sigma L_{n-1}$. By the same reasoning as above, we get that there exist two sequences of vertices $(y_j)_j$ in $L_{n-1}\Z^d \cap [-2\rho L_n,2\rho L_n]^d$ such that for each $j$, \[\sum_{i=1}^d|(y_j)_i-(y_{j+1})_i|=L_{n-1},\] at mutual distances at least $\sigma L_{n-1}$ and whose endpoints are at a distance at most $L_{n-1}$ from $S_1$ and $S_2$ respectively. By definition of $G_{n,0}$, at least one of these, say $(y_j)_{1 \le j \le N}$, must be such that $G_{n-1,y_j}$ is satisfied for each $j$. Without loss of generality, we assume that $y_1$ (resp.\ $y_N$) is at a distance at most $L_{n-1}$ from $S_1$ (resp. $S_2$).

Let us note that since $G_{n,0}\subset H_n$, black vertices at scale $n-1$ in $[-4\rho L_n,4\rho L_n]$ are also black at scale $n$.

Let us first deal with the case where the path is of length 1, i.e.~that $N = 1$. Since $S_1$ and $S_2$ are connected, and since $\lambda\rho \ge 100\sigma$, $\rho \ge 2$, and $y_1$ is at a distance smaller than or equal to $ L_{n-1}$ from $S_1$ and $S_2$, there exist $S_1' \subset S_1$ and $S_2' \subset S_2$ which are connected, of diameter larger than or equal to $ 10\sigma L_{n-2}$ and included in $y_1 + [-\rho L_{n-1},\rho L_{n-1}]^d$. By the induction hypothesis, there exists a path of black vertices in $y_1+[-4\rho L_{n-1},4 \rho L_{n-1}]^d$ that connects $S_1'$ and $S_2'$. This ends the proof in the case $N=1$ since these paths are necessarily included in $[-4\rho L_n,4 \rho L_n]^d$.

Let us now consider the general case of an arbitrary path, i.e.~let us assume that $N \ge 2$. Since $\lambda \rho \ge 100\sigma$ and by the induction hypothesis, for every $j \in \{1,\dots,N-1\}$ there exists a connected set
\[
T_j \subset \Z^d \cap (y_{j+1}+[-\rho L_{n-1},\rho L_{n-1}]^d) \cap (y_j+[-\rho L_{n-1},\rho L_{n-1}]^d)
\]
made of black vertices and of diameter larger than or equal to $10\sigma L_{n-2}$. By the induction hypothesis again, for every $j \in \{1,\dots,N-2\}$, $T_j$ and $T_{j+1}$ are connected by a black path included in $y_{j+1}+[-4 \rho L_{n-1},4\rho L_{n-1}]^d$. Moreover, by reasoning as in the case $N=1$, we obtain that $T_1$ (resp.\ $T_{N-1}$) is connected to $S_1$ (resp.\ $S_2$) by a path of black vertices included in $y_j+[-4 \rho L_{n-1},4\rho L_{n-1}]^d$ with $j=1$ (resp.\ $j=N$). This concludes the proof.
\end{proof}

\section{Existence of an unbounded component in a slab}\label{sec:existence*}
\label{sec7}

We now investigate the existence of unbounded components in $\{f \ge \ell\}$ intersected with a thick slab, for some (small) $\ell > 0$. The main result of this section is the following:

\begin{theorem}\label{thm:existence*}
Let $d \ge 3$. There exists $\beta_0 >0$ such that the following holds. Let $q$ satisfying Assumption~\ref{ass1} for some $\beta > \beta_0$. Then, there exist $\ell,L>0$ such that a.s.\ the set
\[
\{f \ge \ell \} \cap (\mathbb R^2 \times [0,L])
\]
contains an unbounded connected component.
\end{theorem}

Before proving Theorem \ref{thm:existence*}, let us deduce Item i) of Theorem \ref{thm:main} from it.

\begin{proof}[Proof of Item i) of Theorem \ref{thm:main}]
By Theorem \ref{thm:existence*} (together with translation invariance and equality in law of $f$ and $-f$), there exist $\delta,L>0$ such that, for every $\ell \in [-\delta,\delta]$, almost surely there is an unbounded component in both $\{ f \ge \ell \} \cap (\R^2 \times [1,L-1])$ and $\{ f \le \ell \} \cap (\R^2 \times [1,L-1])$. Now, the result is a direct consequence of Lemma \ref{lem:mani_d*} (applied to $d'=3$) and Lemma \ref{L:top} (applied to $E=\R^2\times (0,L)$ and $\Sigma=\{ f = \ell \} \cap E$).
\end{proof}

\subsection{Definition of (very) good points}

Let $a,\gamma \in (0,1)$ such that $\gamma<a^2$ and such that the hypotheses of Proposition \ref{prop:2arms*} hold.

\begin{definition}[good point] \label{defi:good1*}
Let $\delta,R>0$. A point $x \in \R^2$ is called ($\delta$-){\em good} at scale $R$ if the following three properties occur:
\begin{enumerate}[noitemsep]
\item[i)] there exists a connected component of $\{ f_{R^\gamma}\ge R^{-3/2} \} \cap (x+[-R,R]^2)$ of diameter larger than or equal to $ \delta R$;
\item[ii)] all the connected components of $\{ f_{R^\gamma}\ge R^{-3/2} \} \cap (x+[-2R,2R]^2)$ of diameter larger than or equal to $\delta R$ belong to the same connected component of $\{ f_{R^\gamma} \ge R^{-3/2} \}\cap (x + [-4R,4R]^2\times[0,R^a])$.
\item[iii)] $\|f-f_{R^{\gamma}}\|_{\infty,B} \le \tfrac12R^{-3/2}$, where $B=x+[-4R,4R]^2 \times [0,R^a]$.
\end{enumerate}
\end{definition}

\subsection{The renormalization scheme}

Let $\gamma$ and $a$ as above, assume that
\[
\beta>\frac{5}{2\gamma}+\frac{d}{2}
\]
and define a renormalization scheme (see Definition \ref{def:renormalization_scheme}) depending on two parameters $\delta>0$ and $R \ge 1$ as follows.

\begin{itemize}
\item Consider the probability space used in the rest of the paper with the $\Z^2$ action in which $x\in\Z^2$ acts by translating the field by $Rx$.
\item Fix $\rho = 2$, $\sigma=1000$ and $\lambda=10^{10}$.
\item Set $G_{0,0}$ to be the event that  the point $0$ satisfies Items i) and ii) from Definition~\ref{defi:good1*} (recall that these items depend on $\delta$ and $R$). Since these two items are defined in terms of $f_{R^\gamma}$ on $[-4R,4R]^2 \times \R^{d-2}$, the event $G_{0,0}$ is measurable with respect to the white noise $W$ restricted to $[-5R,5R]^2\times \R^{d-2}$ (if $R$ is large enough). By Lemma \ref{lem:macro*} and Proposition \ref{prop:2arms*}, for every $h>0$ there exist $\delta,R_0$ such that for all $R \ge R_0$,
\begin{equation}\label{eq:EX_G_bound}
\prob[G_{0,0}^c] \le h.
\end{equation}
\item For each $n\ge0$, write $f_{(n)}:=f_{(R L_n)^\gamma}$ and if $n\ge1$,
\[
H_n:=\{\|f_{(n)}-f_{(n-1)}\|_{\infty,B_n}<\eps_n\},
\]
where 
\[
B_n:=[-8\rho RL_n+,8\rho RL_n]^2\times[0,R^a] \quad \text{and} \quad \eps_n:=(R L_{n-1})^{-\gamma(\beta-\frac{d}{2})+1}.
\] 
\item Define the events $G_{n,x}$ for $n\in\N$ and $x\in L_n\Z^2$ as in Definition \ref{def:renormalization_scheme}.
\end{itemize}

We start with the following lemma.

\begin{lemma}\label{lem:ind_renorm*}
Consider the renormalization scheme defined above. For every $n \in \N$, the event $G_{n,0}$ is measurable with respect to the white noise $W$ restricted to
\[
\left[- 18\rho L_n,18\rho L_n\right]^2 \times \R^{d-2}.
\]
By our choice of $\rho$ and $\sigma$, we have $\sigma > 36\sqrt{2}\rho$. As a result, $G_{n,x_1}$ and $G_{n,x_2}$ are independent for all $x_1,x_2\in L_n\Z^2$ such that $\textup{dist}(x_1,x_2)\geq \sigma L_n$.
\end{lemma}

\begin{proof}
We construct by induction
 a sequence of positive real numbers $(r_n)_n$ such that for each $n\in\N$, $G_{n,0}$ is measurable with respect to $W$ restricted to
\[
[-r_nR,r_nR]^2 \times \R^{d-2}.
\]
For $n=0$, as explained after the definition of $G_{0,0}$ we may take $r_0=5$. Assume that we have constructed $r_n$. Then, $G_{n+1,0}$ is measurable with respect to $H_{n+1}$ and all the events $G_{n,x}$ for $x\in L_n\Z^2
\cap [-4\rho L_{n+1},4\rho L_{n+1}]^2$. The event $G_{n,x}$ is the translate of $G_{n,0}$ by $Rx$ so it is measurable with respect to the white noise restricted to 
\[
[-(4\rho L_{n+1}+r_n)R,(4\rho L_{n+1}+r_n)R]^2 \times \R^{d-2}.\] 
The event $H_n$ is measurable with respect to the white noise $W$ on 
\[
[-9\rho L_n R,9\rho L_n R]^2\times \R^{d-2}
\] if $R$ is larger than some constant.
Altogether, we may take
\[
r_{n+1}:=9\rho L_{n+1}+r_n.
\]
Since $r_0 \le 9\rho$, the sequence $(r_n)_n$ satisfies 
\[
r_n \le \sum_{k=0}^n 9\rho\lambda^k \le \frac{9\rho \lambda}{\lambda-1}L_n \le 18\rho L_n.\qedhere
\]
\end{proof}

\begin{corollary}\label{cor:cond_renorm*}
There exist $\delta>0$ and $R\ge 1$ such that
\begin{itemize}[noitemsep]
\item[i)] the renormalization scheme satisfies the conditions of Propositions~\ref{lem:stretch_exp_decay} and~\ref{prop:RS_geometry};
\item[ii)] $\sum_{n\ge 1} \eps_n \le \tfrac{R^{-3/2}}{2}$.
\end{itemize}
\end{corollary}

\begin{proof}
We begin with i). By Lemma \ref{L:rest_C1}, there exist $c,C>0$ such that if $RL_{n-1} \ge C\log(RL_n)$ then
\begin{equation}\label{eq:EX_H_bound}
\prob[H_n^c]\leq e^{-c (RL_{n-1})^2}.
\end{equation}
As a result, if $R$ is sufficiently large then
\[
\prob[H_n^c]\leq \frac{1}{4(3\rho\lambda)^4} 2^{-2^n}
\]
for every $n\ge 1$. The rest of the assumptions of Propositions~\ref{lem:stretch_exp_decay} and~\ref{prop:RS_geometry} follow from our choices for $\lambda,\rho,\sigma$, Lemma~\ref{lem:ind_renorm*} and \eqref{eq:EX_G_bound}.

\medskip

As for ii), $\beta>\frac{5}{2\gamma}+\frac{d}{2}$ implies that
\[
\sum_{n\geq 1}\eps_n \le \frac{R^{-\gamma(\beta-\frac{d}{2})+1}}{1-\lambda^{-\gamma(\beta-\frac{d}{2})+1}} = o(R^{-3/2}). \qedhere
\]
\end{proof}

\subsection{Proof of the theorem}

We start with a lemma.

\begin{lemma}\label{lem:unbounded_black*}
Let $\delta>0$ and $R \ge 1$. If there is an infinite path $\pi_1,\pi_2,\dots$ of the lattice $\Z^2$ such that for every $i$, $R\pi_i$ is $\delta$-very good, then there is an unbounded component in $\{ f \ge \tfrac12R^{-3/2} \} \cap (\R^2 \times [0,R^a])$.
\end{lemma}
\begin{proof}
Let $\pi_1,\pi_2,\dots$ be such a path. By i) and ii) from Definition \ref{defi:good1*}, for every $j$ there exists a connected component of $\{ f_{R^\gamma} \ge R^{-3/2} \}$ of diameter larger than or equal to $\delta R$ in both $R\pi_j + [-R,R]^2$ and in $R\pi_{j+1} + [-R,R]^2$, and any two such components are connected in $\{f_{R^\gamma}\geq R^{-3/2} \}\cap(R\pi_j + [-4R,4R]^2 \times [0,R^a])$. We obtain that there exists an unbounded component in
\[
\{ f_{R^\gamma} \ge R^{-3/2} \} \cap (\cup_{j \ge 1}  \pi_j + [-4R,4R]^2 \times [0,R^a]).
\]
By iii) from Definition \ref{defi:good1*} there exists an unbounded component in
\[
\{ f \ge\tfrac12R^{-3/2} \} \cap (\cup_{j \ge 1}  \pi_j + [-4R,4R]^2 \times [0,R^a]).\qedhere
\]
\end{proof}

We now have all the tools to prove Theorem \ref{thm:existence*}.

\begin{proof}[Proof of Theorem \ref{thm:existence*}]
Consider the renormalization scheme defined above and $R,\delta$ as in Corollary \ref{cor:cond_renorm*}. We prove the theorem with $\ell=\tfrac12R^{-3/2}$ and $L=R^a$. Let $m_n$ be the largest integer such that $L_{m_n} \le 2^n$, let $\Lambda_n:=L_{m_n}\Z^2 \cap [-2^n,2^n]^2$, and consider the following event:
\[
A_n:=\bigcap_{x \in \Lambda_n} \Big( G_{m_n,x}\cap\bigcap_{k \ge m_n+1} \tau_x^{-1}(H_k) \Big).
\]
We first note that by Corollary \ref{cor:cond_renorm*}, Proposition \ref{lem:stretch_exp_decay} applies and
\[
\Pro [A_n^c] \le C \Big(2^{1-2^{m_n}}+\sum_{k \ge m_n+1} 2^{-2^{k}}\Big)=O(2^{-2^{-{m_n}}})
\]
for some $C>0$ that depends only on $d,\lambda,\rho$. As a result, a.s.\ there exists $n_0$ such that $A_n$ holds for all $n \ge n_0$.

\medskip

\textbf{A first claim: ``black points imply very good points''.}  For some $n \ge 1$, assume that $A_n$ is satisfied and let $x \in \cup_{y \in \Lambda_n} (y+[-4\rho L_{m_n},4\rho L_{m_n}]^2)$ be a black vertex at scale $m_n$ (recall Definition \ref{defi:black_vertex}). We claim that $Rx$ is a very good point. To show this claim, we note that $G_{0,x}$ is satisfied and that for each $m\geq 1$, there exists $y\in L_m\Z^2$ such that $x\in y+[-4\rho L_m,4\rho L_m]^2$ and $\tau_y^{-1}(H_m)$ is satisfied. In particular,
\[
\|f-f_{R^\gamma}\|_{\infty,B}<\sum_{n\geq 0}\eps_n \le \tfrac12R^{-3/2},
\]
where 
\[
B=Rx+([-4R,4R]^2\times [0,R^a]).
\]
The event $G_{0,x}$ implies that $Rx$ satisfies Items i) and ii) of Definition \ref{defi:good1*} and the above implies that it satisfies Item iiii) of Definition \ref{defi:good1*}. Altogether, $Rx$ is a very good point.

\medskip

\textbf{Conclusion via a second claim.} Consider some $n_0 \ge 1$ and let us assume that $A_n$ holds for all $n \ge n_0$. Lemma \ref{lem:unbounded_black*} and the following claim enable us to conclude: There is an infinite path $\pi_1,\pi_2,\dots \in \Z^2$ from $\partial [-2^{n_0-1},2^{n_0-1}]^2$ to infinity such that for every $i$, $R\pi_i$ is very good.

\medskip

To derive this second claim, for all $n \ge n_0$ we construct a path $\theta_n \subset \Z^2 \cap [-2^n,2^n]^2$ made of black vertices at level $m_n$ and such that
\begin{itemize}[noitemsep]
\item $\theta_{n_0}$ connects $\partial [-2^{n_0-1},2^{n_0-1}]^2$ to $\partial [-2^{n_0},2^{n_0}]^2$;
\item if $n \ge n_0$, then $\theta_{n+1}$ connects $\theta_n$ to $\partial [-2^{n+1},2^{n+1}]^2$.
\end{itemize}
The construction of the infinite path $\pi_1,\pi_2,\dots$ then follows from the existence of the paths $\theta_n$ and from the previous claim.

\medskip

We only explain how to construct the path $\theta_{n_0}$ since the construction for other $n$'s follows the same lines. To this purpose, first consider a path $x_1,\dots,x_k \in \Lambda_{n_0}$ in the lattice $L_{m_{n_0}}\Z^d$ such that
\begin{itemize}[noitemsep]
\item $x_1+[-L_{m_{n_0}},L_{m_{n_0}}]^2$ intersects $\partial [-2^{n_0-1},2^{n_0-1}]^2$ and
\item $x_k+[-L_{m_{n_0}},L_{m_{n_0}}]^2$ intersects $\partial [-2^{n_0},2^{n_0}]^2$.
\end{itemize}
By Proposition \ref{prop:RS_geometry}, for every $i \in \{1,\dots,k-1\}$ there exists a connected set
\[
T_i \subset \Z^2 \cap (x_{i+1}+[-\rho L_{m_{n_0}},\rho L_{m_{n_0}}]^2) \cap (x_i+[-\rho L_{m_{n_0}},\rho L_{m_{n_0}}]^2)
\]
made of black vertices and of diameter larger than or equal to $10\sigma L_{m_{n_0}-1}$ (recall that we set $\rho=2,\sigma=1000,\lambda=10^{10}$). By Proposition~\ref{prop:RS_geometry} once more, for every $i \in \{1,\dots,k-2\}$, $T_i$ and $T_{i+1}$ are connected by a black path included in $x_{i+1}+[-4 \rho L_{m_{n_0}},4\rho L_{m_{n_0}}]^2$. Still by Proposition \ref{prop:RS_geometry}, $T_1$ (resp.\ $T_{k-1}$) is connected to $\partial [-2^{n_0-1},2^{n_0-1}]^2$ (resp.\ $\partial [-2^{n_0},2^{n_0}]^2$) by a path of black vertices included in $x_i+[-4 \rho L_{m_{n_0}},4\rho L_{m_{n_0}}]^2$ with $i=1$ (resp.\ $i=k$). This ends the construction of $\theta_{n_0}$ and the proof of the theorem.
\end{proof}

\section{Existence of a giant component ``almost everywhere''}\label{sec:existence*}
\label{sec8}

In this section, we prove the following result.

\begin{theorem}\label{thm:existencebis*}
Let $d \ge 3$ and $d' \in [3,d]$. There exists $\beta_0>0$ such that the following holds. Let $q$ satisfy Assumption \ref{ass1} for some $\beta > \beta_0$. There exist $\ell,N > 0$ such that a.s.\ there is an unbounded component component $\calC$ of $\{ f \ge \ell \} \cap \R^{d'}$ and some (random) $R_0$ such that for every $R \ge R_0$, $\calC$ instersects all the (Euclidean, closed) balls of radius $(\log R)^N$ that are included in $[-R,R]^{d'}$.
\end{theorem}

Exactly like in Section \ref{sec7}, Item ii) of Theorem \ref{thm:main} is a direct consequence of Theorem~\ref{thm:existencebis*}. We omit this short argument and refer to Section \ref{sec7}. Let us now prove Theorem~\ref{thm:existencebis*} (the proof is very similar to the proofs of the previous section). We write the proof in the case $d'=d$ to simplify the notations and since the proof in the general case is exactly the same. 

\subsection{Definition of (very) good points}

We start by defining a slightly different notion of good and very good points. The essential difference is that we ask connected components to exist on each ``face'' of the boxes.

\medskip

Below, for every $x \in \R^d$, $F_R^1(x),\dots,F_R^{N_d}(x)$ are the $2$-dimensional faces of the hypercube $x+[-R/2,R/2]^d$ and for every $i \in \{1,\dots,N_d\}$, $\tilde{F}_R^i(x) \subset F_R^i(x)$ is the concentric square with side length equal to $R/2$. Let $\gamma$ be such that Proposition \ref{prop:2armsbis*} holds.

\begin{definition}[Good point] \label{defi:good1bis*}
Let $\delta,R>0$. A point $x \in \R^d$ is called ($\delta$-){\em good} at scale $R$ if the following properties occur:
\begin{enumerate}[noitemsep]
\item[i)] for every $i \in \{1,\dots,N_d\}$, there exists a connected component of $\{ f_{R^\gamma}\ge R^{-3/2} \} \cap \tilde{F}_R^i(x)$ of diameter larger than or equal to $\delta R$;
\item[ii)] all the above connected components belong to the same connected component of $\{ f_{R^\gamma} \ge R^{-3/2} \}\cap (x + [-R/2,R/2]^d)$;
\item[iii)]$\|f-f_{R^{\gamma}}\|_{\infty,x+[-R/2,R/2]^d} \le \tfrac12R^{-3/2}$.
\end{enumerate}
\end{definition}

\subsection{The renormalization scheme}

Concerning the renormalization scheme, it accommodates the fact that we now work with a $d$-dimensional renormalization instead of a planar one. In particular, the renormalization space is $\mathbb Z^d$ instead of $\mathbb Z^2$, $B_n$ is defined differently, and $G_{0,0}$ depends on the white noise in $[-2R,2R]^d$. Finally, the fact that $G_{0,0}^c$ has small probability will be a consequence of Proposition~\ref{prop:2armsbis*}.

\medskip

Let $\gamma$ as above, assume that
\[
\beta>\frac{5}{2\gamma}+\frac{d}{2},
\]
and define a renormalization scheme depending on two parameters $\delta>0$ and $R \ge 1$ as follows.

\begin{itemize}
\item Consider the probability space used in the rest of the paper with the $\Z^d$ action given by $x\mapsto Rx$.
\item Fix $\rho = 2$, $\sigma=1000d$ and $\lambda=10^{10}d$.
\item Set $G_{0,0}$ to be the event that  the point $0$ satisfies Items i) and ii) from Definition~\ref{defi:good1bis*}. The event $G_{0,0}$ is measurable with respect to the white noise $W$ restricted to $[-2R,2R]^d$ (if $R$ is large enough). By Lemma \ref{lem:macro*} and Proposition \ref{prop:2armsbis*}, for every $h>0$ there exist $\delta,R_0$ such that for all $R \ge R_0$,
\begin{equation}\label{eq:EX_G_boundbis}
\prob[G_{0,0}^c] \le h.
\end{equation}
\item For each $n\ge0$, write $f_{(n)}:=f_{(R L_n)^\gamma}$ and if $n\ge1$,
\[
H_n:=\{\|f_{(n)}-f_{(n-1)}\|_{C^1(B_n)}<\eps_n\}
\]
where 
\[
B_n:=[-8\rho RL_n,8\rho RL_n]^d \quad \text{and} \quad \eps_n:=(R L_{n-1})^{-\gamma(\beta-\frac{d}{2})+1}.
\] 
\item Define the events $G_{n,x}$ for $n\in\N$ and $x\in L_n\Z^d$ as in Definition \ref{def:renormalization_scheme}.
\end{itemize}

The proofs of the following results are exactly the same as those of Lemma \ref{lem:ind_renorm*} and Corollary \ref{cor:cond_renorm*} (the statements are also very similar, except the appearance of the dimension $d$ instead of $2$). We omit the proofs.

\begin{lemma}\label{lem:ind_renormbis*}
Consider the renormalization scheme defined above. For $n \in \N$, the event $G_{n,0}$ is measureable with respect to the white noise $W$ restricted to
$
\left[- 18\rho L_n,18\rho L_n\right]^d.
$
By our choice of $\rho,\sigma$, we have $\sigma > 36\sqrt{d}\rho$. As a result, $G_{n,x_1}$ and $G_{n,x_2}$ are independent for all $x_1,x_2\in L_n\Z^2$ such that $\textup{dist}(x_1,x_2)\geq \sigma L_n$.
\end{lemma}

\begin{corollary}\label{cor:cond_renormbis*}
We can choose $\delta>0$, $R\ge 1$ such that
\begin{itemize}[noitemsep]
\item[i)] the renormalization scheme satisfies the conditions of Propositions~\ref{lem:stretch_exp_decay} and~\ref{prop:RS_geometry};
\item[ii)] $\sum_{n\ge 1} \eps_n \le\tfrac12R^{-3/2}$.
\end{itemize}
\end{corollary}

\subsection{Proof of the theorem}

We start with a lemma.

\begin{lemma}\label{lem:unbounded_blackbis*}
Let $\delta>0$, $R \ge 1$ and let $\mathcal{C}$ be a connected subset of $\Z^d$ such that, for every $x \in \mathcal{C}$, $Rx$ is $\delta$-very good. Then, there is a connected component of $\{ f \ge \tfrac{R^{-3/2}}{2} \} \cap (\cup_{x \in \mathcal{C}} Rx+[-R/2,R/2]^d)$ that intersects $Rx+[-R/2,R/2]^d$ for every $x \in \mathcal{C}$.
\end{lemma}
The proof is essentially the same as the proof of Lemma~\ref{lem:unbounded_black*} and is omitted here. We now give the proof of the theorem.


\begin{proof}[Proof of Theorem \ref{thm:existencebis*}]
Recall that we write the proof in the case $d'=d$ since the proof in the general case is exactly the same. Consider the renormalization scheme defined above and $R,\delta$ as in Corollary \ref{cor:cond_renormbis*}. We prove the theorem with $\ell=\tfrac12R^{-3/2}$. The proof is essentially the same as that of Theorem~\ref{thm:existence*}, so we leave the details to the reader and only mention the two main differences:
\begin{itemize}
\item We let $m_n$ be the largest integer such that
\[
L_{m_n} \le n^N,
\]
where $N>0$ is any number such that
\[
(2^n/n^N)^d (2^{1-2^{m_n}}+\sum_{k \ge m_n+1} 2^{-2^{k}})
\]
is summable. Moreover, we let $\Lambda_n=L_{m_n}\Z^d \cap [-2^n,2^n]^d$.
\item Given some $n_0$, we construct some paths $\theta_n \subset \Z^d \cap [-2^n,2^n]^d$, $n \ge n_0$, made of black vertices at level $m_n$ and such that
\begin{itemize}[noitemsep]
\item $\theta_{n_0}$ connects $\partial [-2^{n_0-1},2^{n_0-1}]^d$ to $\partial [-2^{n_0},2^{n_0}]^d$;
\item if $n \ge n_0$, then $\theta_{n+1}$ connects $\theta_n$ to $\partial [-2^{n+1},2^{n+1}]^d$;
\item for every $n \ge n_0$, $\theta_n$ intersects $x+[-\rho L_{m_n},\rho L_{m_n}]^d$ for \textit{every} $x \in \Lambda_n$.
\end{itemize}
\end{itemize}
This completes the proof.
\end{proof}

\appendix

\section{The Gaussian FKG inequalities}\label{ssec:app_fkg}

This section is devoted to the proofs of the FKG inequalities Lemma \ref{lem:FKG1*} and Corollary~\ref{cor:FKG} (as well as generalizations and finite dimensional analogues). The proofs have been provided to us by Matthis Lehmkühler and are included here with his permission. Below and in all the appendix, we equip $C(\R^d)$ with the topology of uniform convergence on every compact subset.

\subsection{Proof of the continuous FKG inequality}\label{sssec:FKG}

In this section, we prove Lemma \ref{lem:FKG1*}. The proof is based on the following result by Pitt.

\begin{lemma}[Finite dimensional Gaussian FKG inequality, \cite{Pit}]\label{lem:disFKG1*}
Let $\phi,\psi : \R^n \rightarrow \R$ be two non-decreasing bounded measurable functions and let $X$ be a centered Gaussian vector such that for every $i,j \in \{1,\dots,n\}$, $\E[X_i X_j] \ge 0$. Then,
\[
\E \left[ \phi(X)\psi(X) \right] \ge \E \left[ \phi(X) \right] \E \left[ \psi(X) \right].
\]
\end{lemma}

\begin{proof}[Proof of Lemma \ref{lem:FKG1*}]
This is a direct consequence of Lemma \ref{lem:disFKG1*} and of the two following claims, where $\phi : C(\R^d) \rightarrow \R$ and $f$ are as in the statement of Lemma \ref{lem:FKG1*}. 
\begin{claim}\label{cl:approxFKG1}
Assume furthermore that $\phi$ is continuous. Then, there exists a sequence of non-decreasing functions $\phi_n : C(\R^d) \rightarrow \R$ bounded by $\| \phi \|_\infty$ that depend on only finitely many points\footnote{I.e.\ for all $n$ there exist $x_1,\dots,x_{m_n} \in \R^d$ and a measurable function $\widetilde{\phi}_n : \R^{m_n} \rightarrow \R$ such that $\phi_n(u)=\widetilde{\phi}_n(u(x_1), \dots, u(x_{m_n}))$.} and such that for every $u \in C(\R^d)$, we have $\phi_n(u)\rightarrow \phi(u)$.
\end{claim}

\begin{claim}\label{cl:approxFKG2}
There exists a sequence of non-decreasing continuous functions $\phi_n : C(\R^d) \rightarrow \R$ bounded by $\| \phi \|_\infty$ such that a.s.\ we have $\phi_n(f)\rightarrow \phi(f)$.
\end{claim}

\begin{proof}[Proof of Claim \ref{cl:approxFKG1}]
Let $\chi_\eps$ be a smooth approximation of the identity with compact support. Also, let $\theta_R$ be a smooth function which is compactly supported and equal to $1$ in $[-R,R]^d$. For every $n,R,\eps$ and for every continuous function $u : \R^d \rightarrow \R$ we let
\[
u_{R,n,\eps}(x) := \theta_R(x) \frac{1}{n^d} \sum_{v \in \frac{1}{n}\Z^d} u(v)\chi_\eps(x-v) \hspace{1em} \text{and}  \hspace{1em} \phi_{n,R,\eps}(u) := \phi(u_{R,n,\eps}).
\]
Then, $u_{R,n,\eps}$ converges to $\theta_R \times (u\star\chi_\eps)$ as $n$ goes to infinity, $\theta_R \times (u\star\chi_\eps$) converges to $u\star\chi_\eps$ as $R$ goes to infinity, and $u\star\chi_\eps$ converges to $u$ as $\eps$ goes to $0$. Moreover, $\phi_{n,R,\eps}$ is non-decreasing in $u$ since $u \mapsto u_{R,n,\eps}$ is non-decreasing.
\end{proof}

\begin{proof}[Proof of Claim \ref{cl:approxFKG2}]
We follow the proof of the lemma from Section 3 of \cite{Pit} (which is the analogous result for functions on $\R^n$ rather than $C(\R^d)$). First, exactly as in \cite{Pit}, we note that we can assume that $\phi$ is the indicator of some event $A$ and we work in this case.

\medskip

By regularity of the law of $f$ (because $C(\R^d)$ is a Polish space), there exist two sequences of compact sets $C_n,K_n \subset C(\R^d)$ and a sequence of open sets $O_n \subset C(\R^d)$ with
\begin{equation}\label{eq:regmeas*}
K_n \subset  C_n \cap A; \hspace{1em} A \subset O_n; \hspace{1em} \Pro [f \notin C_n]<1/n^2; \hspace{1em} \Pro [f \in O_n \setminus K_n]<1/n^2.
\end{equation}
For every $\eps,M>0$ and $u \in C(\R^d$), we define the following non-decreasing neighborhood of $u$:
\[
R^M(u,\eps):=\{ v \in C(\R^d) : \forall x \in [-M,M]^d, v(x) > u(x)-\eps \}.
\]
As in \cite{Pit}, we fix
\begin{itemize}[noitemsep]
\item[i)] two sequences $\eps_n,M_n$ such that, for every $u \in K_n$, $C_n \cap R^{M_n}(u,2\eps_n)\subset O_n$,
\item[ii)] a finite cover $\{ R^{M_n}(u^i,\eps_n)\}_{i=1}^{N_n}$ of $K_n$,
\end{itemize}
and we let
\[
L_n := \bigcup_{i=1}^{N_n} R^{M_n}(u^i,\eps_n).
\]
Then, we define the following continuous and non-decreasing functions:
\[
\phi_n : u \mapsto \Big(1 - \eps^{-1}_n\inf_{v \in L_n} \| v-u \|_{\infty,[-M_n,M_n]^d} \Big)_+.
\]
It remains to show that $\phi_n(f)$ converges to $\phi(f)$ a.s. To prove this, we use \eqref{eq:regmeas*} and we observe that
\begin{multline*}
\mathbf 1_{K_n}(u) \le \mathbf 1_{L_n}(u) \le \phi_n(u) \le \mathbf 1\{\exists i, \inf_{v \in R^{M_n}(u^i,\eps_n)} \| v-u \|_{\infty,[-M_n,M_n]^d} < \eps_n\}\\
\le \mathbf 1\{\exists i, u \in R^{M_n}(u^i,2\eps_n) \} \le \mathbf 1_{O_n}(u) + \mathbf 1_{C_n^c}(u).
\end{multline*}
\end{proof}
This concludes the proof of Lemma~\ref{lem:FKG1*}.\end{proof}

\subsection{The local FKG inequalities}

In this section, we prove discrete and continuous local FKG inequalities (Lemmas \ref{lem:disFKG2*} and \ref{lem:localFKG}). Corollary \ref{cor:FKG} is a direct consequence of Lemma \ref{lem:localFKG}.

\medskip

Let us start with some notations. For every $\phi : \R^n \rightarrow \R$ and for every $i \in \{1,\dots,n\}$ such that $\phi$ is monotonic in $i$, we let $\epsilon_i^\phi=0$ if $\phi$ does not depend on $x_i$ (in the sense that if $x$ and $y$ agree outside of the $i^{th}$ coordinate then $\phi(x)=\phi(y)$); $\epsilon_i^\phi=1$ if $\phi$ depends on $x_i$ and is non-decreasing in $x_i$; $\epsilon_i^\phi=-1$ if $\phi$ depends on $x_i$ and is non-increasing in $x_i$.

\begin{lemma}[Finite dimensional local Gaussian FKG inequality]\label{lem:disFKG2*}
Let $\phi,\psi : \R^n \rightarrow \R$ be two bounded measurable functions and let $X$ be a centered Gaussian vector. Assume that for every $(i,j) \in \{1,\dots,n\}^2$ one of the following properties hold:
\begin{itemize}[noitemsep]
\item $\E[X_i X_j] = 0$,
\item $\phi$ does not depend on $x_i$,
\item $\psi$ does not depend on $x_j$,
\item $\phi$ is monotonic in $x_i$, $\psi$ is monotonic in $x_j$ and $\E[X_i X_j] \epsilon_i^\phi \epsilon_j^\psi \ge 0$.
\end{itemize}
Then,
\[
\E \left[ \phi(X)\psi(X) \right] \ge \E \left[ \phi(X) \right] \E \left[ \psi(X) \right].
\]
\end{lemma}

\begin{proof}
The proof is essentially the same as Lemma \ref{lem:disFKG1*} (proven in \cite{Pit}). Pitt proves Lemma \ref{lem:disFKG1*} by first showing it with the further hypothesis that $\phi$ and $\psi$ are continuous and then shows an approximation result. The proof of Lemma \ref{lem:disFKG2*} with the further hypothesis that $\phi$ and $\psi$ are continuous is exactly the same as in \cite{Pit}. Then, we can conclude by using the following claim, where $\phi$ and $X$ are as in the statement of Lemma~\ref{lem:disFKG2*}.

\begin{claim}\label{cl:approxFKGbis}
Let $I$ (resp.\ $J$) denote the set of coordinates on which $\phi$ is non-decreasing (resp.\ non-increasing).\footnote{Note that $I \cap J$ is the set of coordinates on which $\phi$ does not depend.} There exists a sequence of continuous functions $\phi_n : \R^n \rightarrow \R$ bounded by $\| \phi \|_\infty$ that are non-decreasing (resp.\ non-increasing) in $I$ (resp. $J$) and such that a.s.\ $\phi_n(X) \rightarrow \phi(X)$. 
\end{claim}

\begin{proof}
The proof is essentially the same as that of Claim \ref{cl:approxFKG2}. As previously, we can assume that $\phi$ is the indicator of some event $A$ and we consider $C_n,K_n \subset \R^n$ two sequences of compact sets and $O_n \subset \R^n$ a sequence of open sets that satisfy \eqref{eq:regmeas*}. We then define the following set for every $x \in \R^d$ and $\eps>0$:
\[
R(x,\eps) := \Bigg\{ y \in \R^k :
\begin{array}{l} \forall i \in I \setminus J, y_i > x_i-\eps,\\
\forall i \in J \setminus I, y_i < x_i+\eps,\\
\forall i \notin I \cup J, |x_i-y_i|< \eps.
\end{array} \Bigg\}.
\]
As previously, we fix
\begin{itemize}[noitemsep]
\item a sequence $\eps_n>0$ such that for every $x \in K_n$, $C_n \cap R(x,\eps_n) \subset O_n$,
\item a finite cover $\{ R(x^i,\eps_n)\}_{i=1}^{N_n}$ of $K_n$,
\end{itemize}
and we let
\[
L_n := \bigcup_{i=1}^{N_n} R(x^i,\eps_n) \quad \text{and} \quad \phi_n(x) := \Big(1 - \eps^{-1}_n\inf_{y \in L_n} \| y-x \|_{\infty} \Big)_+.
\]
As in the proof of Claim \ref{cl:approxFKG2}, the functions $\phi_n$ satisfy the desired properties.
\end{proof}
This concludes the proof of Lemma~\ref{lem:disFKG2*}.\end{proof}

We end the section by showing a general local FKG inequality in the continuum. As previously, we need some notation. If $x \in \R^d$ and $\delta > 0$, we let $Q_\delta(x):=x+[-\delta,\delta]^d$. Monoticity properties for functions on $C(\R^d)$ are defined above Corollary \ref{cor:FKG}. Moreover, we say that a function $\phi : C(\R^d) \rightarrow \R$ does not depend on some set $U \subset \R^d$ if $\phi(u) = \phi(v)$ for every $u,v$ that agree on $U^c$. Let $\phi : C(R^d) \rightarrow \R$ and $U \subset \R^d$. If $\phi$ is monotonic in $U$, we let $\epsilon_U^\phi = 0$ if $\phi$ does not depend on $U$ (in the sense that $\phi(u)=\phi(v)$ for every $u,v$ that agree outside of $U$); $\epsilon_U^\phi=1$ if $\phi$ depends on $U$ and is non-decreasing in $U$; $\epsilon_U^\phi=-1$ if $\phi$ depends on $U$ and is non-increasing in $U$.

\begin{lemma}[Local continuous Gaussian FKG inequality]\label{lem:localFKG}
Let $\phi,\psi : C(\R^d) \rightarrow \R$ be two bounded measurable functions and let $f$ be a centered continuous Gaussian field on $\R^d$. Assume that there exists $\delta>0$ such that for every $(x,y) \in (\R^d)^2$ one of the following properties holds:
\begin{itemize}[noitemsep]
\item $\forall (s,t) \in Q_\delta(x) \times Q_\delta(y), \E[f(s)f(t)] = 0$,
\item $\phi$ does not depend on $Q_\delta(x)$,
\item $\psi$ does not depend on $Q_\delta(y)$,
\item $\phi$ is monotonic in $Q_\delta(x)$, $\psi$ is monotonic in $Q_\delta(y)$, the sign of $\E[f(s)f(t)]$ is constant in $Q_\delta(x) \times Q_\delta(y)$, and
\[
\E[f(x)f(y)]\epsilon_{Q_\delta(x)}^\phi \epsilon_{Q_\delta(y)}^\psi \ge 0.
\]
\end{itemize}
Then,
\[
\E \left[ \phi(f)\psi(f) \right] \ge \E \left[ \phi(f) \right] \E \left[ \psi(f) \right].
\]
\end{lemma}

\begin{proof}
Let $\eta>0$, let $\rho : C(\R^d) \rightarrow \R$ be a bounded measurable function and let $U,V \subset \R^d$ defined by $x \in U$ (resp.\ $V$) if $\phi$ is non-decreasing (resp.\ non-increasing) in $Q_\eta(x)$. Lemma~\ref{lem:localFKG} is a direct consequence of Lemma~\ref{lem:disFKG2*} and of the two following claims.

\begin{claim}\label{cl:approxFKG1bis}
Assume furthermore that $\rho$ is continuous. Then, there exists a sequence of functions $\rho_n : C(\R^d) \rightarrow \R$ bounded by $\| \rho \|_\infty$ that depend on only finitely many points and such that:
\begin{itemize}[noitemsep]
\item for every $u \in C(\R^d)$, we have $\rho_n(u)\rightarrow \rho(u)$,
\item for every $x \in U$ (resp.\ $V$), $\rho_n$ is non-decreasing (resp.\ non-increasing) in $Q_{\eta/2}(x)$.
\end{itemize}
\end{claim}

\begin{claim}\label{cl:approxFKG2bis}
There exists a sequence of continuous functions $\rho_n : C(\R^d) \rightarrow \R$ bounded by $\| \rho \|_\infty$ such that
\begin{itemize}[noitemsep]
\item a.s.\ we have $\rho_n(f)\rightarrow \rho(f)$,
\item for every $x \in U$ (resp.\ $V$), $\rho_n$ is non-decreasing (resp.\ non-increasing) in $Q_{\eta/2}(x)$.
\end{itemize}
\end{claim}

\begin{proof}[Proof of Claim \ref{cl:approxFKG1bis}]
The proof is the same as Claim \ref{cl:approxFKG1} with the further hypothesis that the support of $\chi_\eps$ is included in $[-\eta/2,\eta/2]^d$.
\end{proof}

\begin{proof}[Proof of Claim \ref{cl:approxFKG2bis}]
The proof is essentially a combination of the proofs of Claims \ref{cl:approxFKG2} and \ref{cl:approxFKGbis}. More precisely, we define $I,J \subset (\eta/4)\Z^d$ by $x \in I$ (resp.\ $J$) if $\rho$ is non-decreasing (resp.\ non-decreasing) in $Q_{\eta/4}(x)$ and we follow the proof of Claim \ref{cl:approxFKG2bis} but replacing the sets $R^M(u,\eps)$ by
\[
\Bigg\{ v \in C(\R^d) :
\begin{array}{l} \forall x \in (I \setminus J) \cap [-M,M]^d, \; \forall y \in x+[0,\eta/4)^d, \; v(y)> v(x)-\eps,\\
\forall x \in (J \setminus I) \cap [-M,M]^d, \; \forall y \in x+[0,\eta/4)^d, \; v(y) < u(y)+\eps,\\
\forall x \in (I \cup J)^c \cap [-M,M]^d, \; \forall y \in x+[0,\eta/4)^d, \; |v(y)-u(y)|< \eps.
\end{array} \Bigg\}.
\]
The rest of the proof is the same so we omit the details.
\end{proof}
This concludes the proof of Lemma \ref{lem:localFKG}.
\end{proof}

\section{Other properties of Gaussian fields}

\subsection{Approximation by truncation and sprinkling}\label{ssec:approx}

In this section, we present two approximation results that essentially come from \cite{MV}. The first one is based on Kolmogorov and BTIS lemmas and shows that the field does not vary too much under truncation. The second one is based on a Cameron--Martin type argument and shows that the probability of monotone events does not change much under a small sprinkling.

\begin{lemma}\label{L:rest_C1}
Assume that $q$ satisfies Assumption \ref{ass1} for some $\beta>d$. Then, there exist $c,C > 0$ such that, for all $r,R \geq 1$ and $t \ge C \log R$,
\begin{equation}
\mathbb{P}\big[\Vert f-f_r \Vert_{\infty,B(R)}\geq tr^{-(\beta-\frac d2)}\big] \leq e^{-ct^2}.
\end{equation} 
\end{lemma}

\begin{proof}
The proof follows the lines of that of Proposition~3.11 in \cite{MV}, noting that Assumption \ref{ass1} implies that
$$
\sup_x \sup_{|\alpha|\leq 1} \mathbb{E}\big[ (\partial^{\alpha} (f-f_r )(x))^2\big] =O(r^{d-2\beta}),
$$
which, upon applying the Kolmogorov and BTIS lemmas, yields the desired bound.
\end{proof}

\begin{lemma}[Proposition 3.6 in \cite{MV}]\label{lem:Cameron-Martin}
Let $q$ satisfying Assumption \ref{ass1} for some $\beta > d$. There exist $C,r_0,R_0 > 0$ such that for every $r \in [r_0,\infty]$, $R \ge R_0$, $t \in \R$ and every monotonic event $A \in \mathcal{F}_{D(0,R) \times [0,1]^{d-2}}$, we have
\[
\left| \Pro [ f_r \in A ] - \Pro [ f_r + t \in A ] \right| \leq C |t| R. 
\]
\end{lemma}

It is important to note that $C$ does not depend on $r \ge r_0$.

\begin{proof}
The proof is the same as the one of the analogous result Proposition 3.6 in \cite{MV}. Indeed, the conditions $q$ non-identically equal to $0$ and $q\geq 0$ as well as the decay condition of Assuption \ref{ass1} imply that $\mathcal{F}(q_r\star q_r)$ is well-defined and larger than some constant $c$ that does not depend on $r$ in some neighbourhood of $0$ (as soon as $r$ is sufficiently large). The only difference in the proof is that since $f_r$ is defined in $\R^d$, the function $h$ that one should choose is rather
\[
\mathcal{F}[ R^{2+a} c^{-d}\mathbf 1_{[0,cR]^2 \times \times [0,c]^{d-2}} ]. \qedhere
\]
\end{proof}

\subsection{A bound on the number of connected components of the excursion sets}

\begin{lemma}\label{L:numbercomps}
Let $q$ satisfying Assumption \ref{ass1} for some $\beta>d$. There exists $C>0$ such that for every $r \ge r_q$ and $\ell \in \R$, the expectation of the number of connected components of $\{f_r \ge \ell \}\cap \mathbb R^2$ that intersect $D(0,1)$ is less than $C$.
\end{lemma}

\begin{proof}
Let $N_2(r)$ (resp.\ $N_1(r)$) denote the number of critical points of the restriction of $f_r$ to $D(0,1)$ (resp.\ $\partial D(0,1)$). Then, the number considered in the statement is less than or equal to $N_1(r)+N_2(r)$. Moreover, by Kac--Rice formula (see e.g.~Theorem 6.2 and Proposition 6.5 of \cite{AW}) and stationarity, we have
\begin{align*}
\E[ N_1(r)]&=\int_0^{2\pi}\frac{1}{\sqrt{2\pi|\partial_\theta^2\kappa_r(0)|}}\E \big[|\partial_\theta^2f_r(e^{i\theta})| \mid \partial_\theta f_r(e^{i\theta})=0 \big] d\theta\\
&= \frac{2\pi}{\sqrt{2\pi|\partial_\theta^2\kappa_r(0)|}}\E [|\partial_\theta^2f_r(0)|],
\end{align*}
 where $\kappa_r=q_r \star q_r$ is the covariance function of $f_r$. Similarly, setting $g_r = (f_r)_{|\R^2}$ and using stationarity, we have
\begin{align*}
\E[N_2(r)]&=\int_{D(0,1)}\frac{1}{\sqrt{(2\pi)^2\det(\textup{Cov}(\nabla g_r(0)))}}\E \big[|\nabla^2 g_r(x)| \mid \nabla g_r(x)=0 \big]dx\\
&= \frac{\pi}{\sqrt{(2\pi)^2\det(\textup{Cov}(\nabla g_r(0)))}}\E[|\nabla^2 g_r(0)|].
\end{align*}
The above quantities are bounded uniformly in $r \ge r_q$.
\end{proof}

\section{A percolation estimate: proof of Lemma \ref{lemma:angle_crossing}}\label{ss:angle_crossing}

In this section we provide a proof of Lemma \ref{lemma:angle_crossing} based on a classical method from planar percolation. We expect that a stronger result could be reached using a deformation estimate as in the proof of Lemma \ref{lem:comparison} to compare the probability of crossing a flat rectangle to the probability of crossing the same rectangle bent into the shape of a quarter of a cylinder. However, this argument relies on full rotational symmetry and on Gaussianity so it does not invite generalization as much as the argument presented here.

\begin{proof}[Proof of Lemma \ref{lemma:angle_crossing}]
The proof is via a second moment method involving arm events in half-planes. Recall that $B(R)$ is the Euclidean ball centered at $0$.

\medskip

An arm from scale $r_1$ to scale $r_2$ (where $0 < r_1 \le r_2$) is a path in $\{ f_{R^\gamma} \ge 0 \}$ from $\partial B(r_1)$ to $\partial B(r_2)$. Let $\alpha_{r_1,r_2}$ denote the probability that there exist two arms from scale $r_1$ to scale $r_2$ that are included in the half planes $\R \times \{0\} \times \R_+$ and $\R \times \R_+ \times \{0\}$ respectively. We first note that by the FKG inequality (Lemma \ref{lem:FKG1*}), the RSW theorem (Theorem \ref{thm:rsw}) and standard gluing properties (see Item iv) at the end of Section \ref{ss:rsw}), for all $r_3 \ge r_2 \ge r_1 > 0$,
\begin{equation}\label{eq:angle_crossing_2}
\alpha_{r_1,r_2}\alpha_{r_2,r_3} \le C\alpha_{r_1,r_3}
\end{equation}
for some universal $C>0$. Similarly, if $0 < r_1 \le 4r_2$, then $\alpha_{r_1/2,2r_2}$ is of the same order as $\alpha_{r_1,r_2}$. We will use the latter below without mentioning it.

\begin{claim}
There exists a universal constant $c>0$ such that if $r_2 \ge R^\gamma$ and $1 \le r_1 \le r_2$ then
\begin{equation}\label{eq:angle_crossing_1}
\alpha_{r_1,r_2} \ge cr_1/r_2.
\end{equation}
\end{claim}
\begin{proof}
By RSW and Lemma \ref{lem:mani*} and since $f_{R^\gamma}$ is $R^\gamma$-dependent, for any $r \ge 2R^\gamma$ the probability that an $r \times r$ square is crossed by a path in $\{f_{R^\gamma} = 0 \}$ is larger than some universal positive constant. Any ball centered at the extremity of this path touching a given side is then connected to the opposite side by a path in $\{f_{R^\gamma}\geq 0\}$ and a path in $\{f_{R^\gamma}\leq 0\}$. These two events are negatively correlated by FKG. By using this observation and the union bound, we deduce that the probability that there exists an arm from scale $r_1$ to scale $r_2\geq R^\gamma$ in $\R\times\R_+$ is at least $c(r_2/r_1)^{-1/2}$ for some universal constant $c>0$. Applying FKG once again we obtain, as required,
\[
\alpha_{r_1,r_2}\geq (c(r_2/r_1)^{-1/2})^2=c^2(r_1/r_2). \qedhere
\]
\end{proof}

Now, let $\mathcal{S}_r^1:=[0,r] \times \{0\} \times [0,r]$ and $\mathcal{S}_r^2 := [0,r] \times [0,r] \times \{0\}$ (so that $\mathcal{S}_r= \mathcal{S}_r^1 \cup \mathcal{S}_r^2$), let $I$ be the set of points of integer coordinates that belong to $[r/4,3r/4] \times \{0\} \times \{0\}$ (so in particular $I$ is included in the common side of $\mathcal{S}_r^1$ and $\mathcal{S}_r^2$) and for every $i \in I$ let $A_i$ (resp.~$B_i$) denote the event that there is a path in $\{ f_{R^\gamma} \ge 0 \} \cap \mathcal{S}_r^1$ (resp.~$\{ f_{R^\gamma} \ge 0 \} \cap \mathcal{S}_r^2$) from $i$ to the opposite side of $\mathcal{S}_r^1$ (resp.\ $\mathcal{S}_r^2$). Finally, let $X$ denote the number of points $i \in I$ such that both $A_i$ and $B_i$ hold. Note that if $X>0$ then the event from the lemma holds. We now apply the second moment method to $X$.

\medskip

We first note that by FKG and RSW, $\E[X] \ge c'r\alpha_{1,r}$ for some $c'>0$. In order to estimate $\E[X^2] = \sum_{i,j} \mathbf 1_{A_i \cap B_i}\mathbf 1_{A_j \cap B_j}$ we need to consider separately the cases $|i-j| \ge 2R^\gamma$ and $|i-j| < 2R^\gamma$. When $|i-j| \ge 2R^\gamma$ (which corresponds to the first sum below), three independent events happen: two events between scale $1$ and $|i-j|/2$ and an event between scales $|i-j|$ and $r$), so we obtain that
\[
\E[X^2] \le C'r\sum_{R^\gamma \le k \le r} \alpha_{1,k}^2 \alpha_{k,r} + C'r \sum_{1 \le k \le R^\gamma} \alpha_{1,r},
\]
for some $C'>0$. By the quasi-multiplicativity property \eqref{eq:angle_crossing_2}, the above is at most a positive constant times
\[
r\alpha_{1,r}^2 \sum_{R^\gamma \le k \le r} 1/\alpha_{k,r} + r\alpha_{1,r} R^\gamma.
\]
Moreover, since by \eqref{eq:angle_crossing_1}, $\alpha_{k,r} \ge ck/r$ (so in particular  $r\alpha_{1,r} \ge c$), we obtain that the above is at most a positive constant times
\[
r^2\alpha_{1,r}^2(\log(r) + R^\gamma).
\]
Finally, by the Cauchy--Schwarz inequality,
\[
\Pro[X>0] \ge \frac{\E [X]^2}{\E [X^2]} \ge \frac{c''}{\log(r) + R^\gamma}
\]
for some $c''>0$, which is the desired result.
\end{proof}

\section{A topological lemma}

\begin{lemma}\label{L:top}
Let $E = \R^d$ or $\R^d \times (L,L')$ for some $d \ge 1$ and $L<L'$ and let $\Sigma \subset E$ be a $C^1$-smooth (not necessarily connected nor compact) hypersurface without boundary. Assume that there exist two distinct unbouded connected components of $E\setminus\Sigma$. Then, $\Sigma$ has an unbounded connected component.
\end{lemma}

\begin{proof}
Let $\Sigma_1,\Sigma_2,\dots$ be all the bounded components of $\Sigma$. Our goal is to show that $\Sigma \setminus (\cup_i \Sigma_i)$ is non-empty.

By the Jordan--Brouwer separation theorem, for every $i$, $E_i := E \setminus \Sigma_i$ has two connected components $B_i$ and $U_i$, with $B_i$ bounded and $U_i$ unbounded. Moreover, $\partial U_i=\partial B_i=\Sigma_i$. Let $\overline{B}_i$ be the closure of $B_i$.

Let us show that each $\overline{B}_i$ is included in a finite number of other $\overline{B}_j$'s. To this purpose, fix some index $i$, let $R>0$ sufficiently large so that $B_i \subset [-R,R]^d$ and one of the unbounded components of $\Sigma$ intersects $[-R,R]^d$. Then, all the indices $j$ that are such that $\overline{B}_i \subset \overline{B}_j$ must satisfy that $\Sigma_j \cap [-R,R]^d \ne \emptyset$. Since the hypersurface is ($C^1$)-smooth, there are only finitely many such $j$'s.

We let $\overline{B}_{i_1},\overline{B}_{i_2},\dots$ be all the $\overline{B}_i$'s that are maximal for the partial order given by the inclusion. We note that the $\overline{B}_{i_k}$'s are disjoint and that (as a consequence of the claim from the previous paragraph) $\cup_k \overline{B}_{i_k}=\cup_i \overline{B}_i$.

Let $F = E \setminus (\cup_i \overline{B}_i)$. Then, $F \setminus \Sigma$ is not connected (indeed, $(F \setminus \Sigma) \subset (E \setminus \Sigma)$ and $F \setminus \Sigma$ contains the unbounded components of $E \setminus \Sigma$).

Let us now show that $F$ is a connected set. This will end the proof that $\Sigma \setminus (\cup_i \Sigma_i)$ is non-empty since $F \setminus \Sigma$ is not connected and $F \cap \Sigma_i= \emptyset$ for every $i$.

Let $g : F \rightarrow \{0,1\}$ be a continuous function. We want to show that $g$ is constant. But, by the Jordan--Brouwer separation theorem, we can continuously extend $g$ to all the $\overline{B}_{i_k}$'s\footnote{To prove this, consider $x,y \in \Sigma_k$, consider $U_x$ (resp.\ $U_y$) a neighborhood of $x$ (resp.\ $y$), consider a continuous path $\gamma$ on $\Sigma_{i_k}$ from $x$ to $y$, use the Jordan--Brouwer separation theorem to say that $\gamma$ is included in the boundary of both $F$ and $F^c$, and use local charts to find a continuous path included in $F$ from $U_x \cap F$ to $V_y \cap F$.}, so we can continuously extended $g$ to the connected set $E$. As a result, $g$ is constant.
\end{proof}

\footnotesize
\bibliographystyle{alpha}
\bibliography{sources_final}

\end{document}